\documentclass[12pt,leqno]{article}
\usepackage{amssymb,amsfonts,amsmath,amsthm,amscd,bbold,mathrsfs}
\usepackage{MnSymbol}
\usepackage{color}
\usepackage{psfrag}
\usepackage{graphicx}
\usepackage{centernot}
\usepackage{stmaryrd}
\def\CC{\leavevmode\setbox0=\hbox{h}\dimen=\ht0 \advance \dimen by-1ex\rlap{\raise.9\dimen\hbox{\kern .15em \char'27}} C}

\def\Va{\stackrel{\cV^a}{\longleftrightarrow} \hspace{-2.8ex} \mbox{\f /}\;\;\;}

\def\phiv{\stackrel{\phi \ge v}{\longleftrightarrow} \hspace{-2.8ex} \mbox{\f /}\;\;\;}

\def\U{\stackrel{U}{\longleftrightarrow} \hspace{-2.8ex} \mbox{\f /}\;\;\;}  %%U

\def\IP{{\mathbb P}}
\def\IR{{\mathbb R}}
\def\IN{{\mathbb N}}
\def\IL{{\mathbb L}}

\def\IZ{{\mathbb Z}}
\def\IA{{\mathbb A}}

\def\IB{{\mathbb B}}
\def\Ib{{\mathbb b}}

\def\n{\noindent}
\def\dsl{\textstyle\sum\limits}

\def\dis{\displaystyle}

\def\fr{\mbox{\footnotesize $\dis\frac{1}{2}$}}

\def\ov{\overline}
\def\ve{\varepsilon}
\def\f{\footnotesize}

\def\r{\rightarrow}
\def\point{{\mbox{\large $.$}}}
\def\wh{\widehat}
\def\wt{\widetilde}

\def\cA{{\cal A}}
\def\cB{{\cal B}}
\def\cC{{\cal C}}

\def\cE{{\cal E}}
\def\cL{{\cal L}}

\def\cE{{\cal E}}
\def\cI{{\cal I}}
\def\cJ{{\cal J}}

\def\cF{{\cal F}}
\def\cR{{\cal R}}
\def\cS{{\cal S}}
\def\cG{{\cal G}}

\def\cU{{\cal U}}
\def\cV{{\cal V}}

\dimendef\dimen=0

\newtheorem{theorem}{Theorem}[section]
\newtheorem{lemma}[theorem]{Lemma}
\newtheorem{corollary}[theorem]{Corollary}
\newtheorem{proposition}[theorem]{Proposition}
\newtheorem{remark}[theorem]{Remark}

\begin{document}

\noindent
~

\bigskip
\begin{center}
{\bf ON THE COST OF THE BUBBLE SET FOR \\ RANDOM INTERLACEMENTS}
\end{center}

\begin{center}
Alain-Sol Sznitman
\end{center}

%\begin{center}
%Preliminary Draft
%\end{center}

\begin{abstract}
The main focus of this article concerns the strongly percolative regime of the vacant set of random interlacements on $\IZ^d$, $d \ge 3$. We investigate the occurrence in a large box of an excessive fraction of sites that get disconnected by the interlacements from the boundary of a concentric box of double size. The results significantly improve our findings in \cite{Szni15}. In particular, if, as expected, the critical levels for percolation and for strong percolation of the vacant set of random interlacements coincide, the asymptotic upper bound that we derive here matches in principal order a previously known lower bound. A challenging difficulty revolves around the possible occurrence of droplets that could get secluded by the random interlacements and thus contribute to the excess of disconnected sites, somewhat in the spirit of the Wulff droplets for Bernoulli percolation or for the Ising model. This feature is reflected in the present context by the so-called {\it bubble set}, a possibly quite irregular random set. A pivotal progress in this work has to do with the improved construction of a coarse grained random set accounting for the cost of the bubble set. This construction heavily draws both on the {\it method of enlargement of obstacles} originally developed in the mid-nineties in the context of Brownian motion in a Poissonian potential in \cite{Szni97a}, \cite{Szni98a}, and on the {\it resonance sets} recently introduced by Nitzschner and the author in \cite{NitzSzni20} and further developed in a discrete set-up by Chiarini and Nitzschner in \cite{ChiaNitz}.

\end{abstract}

\vspace{4cm}

\n
Departement Mathematik\\ %\hfill June 2021\\
ETH Z\"urich\\
CH-8092 Z\"urich\\
Switzerland

\newpage
\thispagestyle{empty}
~

\newpage
\setcounter{page}{1}

\setcounter{section}{-1}
\section{Introduction}
Both random interlacements and the Gaussian free field are models with long range dependence for which the Dirichlet energy plays an important role. The study of largely deviant events and their impact on the medium in the context of the percolation of the vacant set of random interlacements, or in the closely related level-set percolation of the Gaussian free field, has attracted much attention over the recent years. In particular, important progress has been achieved in the understanding of various instances of atypical disconnection of macroscopic bodies, see for instance \cite{ChiaNitz20a}, \cite{ChiaNitz20b}, \cite{ChiaNitz}, \cite{GoswRodrSeve21},  \cite{LiSzni14},  \cite{Nitz18a}, \cite{NitzSzni20}, \cite{Szni15}, \cite{Szni17}, \cite{Szni19b}, \cite{Szni21a}, \cite{Szni21b}.

\medskip
Our main focus in this article lies in the strongly percolative regime of the vacant set of random interlacements on $\IZ^d$, $d \ge 3$. Specifically, we investigate the occurrence in a large box centered at the origin of an excessive fraction of sites that get disconnected by the interlacements from the boundary of a concentric box of double size. Our results significantly improve our findings in \cite{Szni21b}. In particular, if, as expected, the critical levels for percolation and for strong percolation of the vacant set of random interlacements coincide, the asymptotic upper bound that we derive here matches in principal order a previously known lower bound. As often the case, the derivation of the lower bound suggests a scheme on how to produce the largely deviant event, and a matching upper bound confers some degree of pertinence to this scheme. A challenging difficulty in the derivation of an upper bound for the present problem revolves around the possible occurrence of droplets that could get secluded by the random interlacements and thus contribute to the excess of disconnected sites, somewhat in the spirit of the Wulff droplet in the context of Bernoulli percolation or for the Ising model, see \cite{Cerf00}, \cite{Bodi99}. In the present context this feature is reflected by the so-called {\it bubble set}. A pivotal progress in this work has to do with an improved construction of a coarse grained random set accounting for the cost of the bubble set. The bubble set is quite irregular and our construction heavily draws on a combination of ideas and techniques from the {\it method of enlargement of obstacles} in \cite{Szni97a}, \cite{{Szni98a}}, and from the {\it resonance sets} in \cite{NitzSzni20}, \cite{ChiaNitz}.

\medskip
We now describe the results in more details. We let $\cI^u$ stand for the random interlacements at level $u \ge 0$ in $\IZ^d$, $d \ge 3$, and $\cV^u = \IZ^d \backslash \cI^u$ for the corresponding vacant set. We refer to \cite{CernTeix12}, \cite{DrewRathSapo14c}, and Section 1 of \cite{Szni17} for background material. We are interested in the {\it strong percolative regime} of the vacant set, i.e.~we assume that
\begin{equation}\label{0.1}
0 < u < \ov{u} \; (\le u_*),
\end{equation}
where $u_*$ and $\ov{u}$ respectively denote the critical level for the percolation and for the strong percolation of $\cV^u$. When $u < u_*$ the infinite cluster $\cC^u_\infty$ of $\cV^u$ exists and is unique a.s., see \cite{Szni10}, \cite{SidoSzni09}, \cite{Teix09a}, and informally $u < \ov{u}$ corresponds to the local presence and local uniqueness of the infinite cluster. We refer to (2.3) of \cite{Szni17} or (1.26) of \cite{Szni21b} for the precise definition of $\ov{u}$. It is expected (and presently the object of active research) that $u_* = \ov{u}$. The corresponding equality has been established in the closely related model of the level-set percolation of the Gaussian free field, see \cite{DumiGoswRodrSeve20}.

\medskip
We write $\theta_0$ for the percolation function
\begin{equation}\label{0.2}
\theta_0(a) = \IP[0 \Va \infty], \; a \ge 0,
\end{equation}
where $\{0 \Va \infty\}$ stands for the event that $0$ does not belong to an infinite connected component of $\cV^a$. The function $\theta_0$ is known to be non-decreasing, left-continuous, identically equal to $1$ on $(u_*, \infty)$ with a possible (although not expected from simulations) jump at $u_*$. It is continuous elsewhere, see \cite{Teix09a}. We also refer to \cite{Szni21} for the $C^1$-property of $\theta_0$ below a certain level $\wh{u}$, which is also expected to coincide with $u_*$. Nevertheless, the behavior of $\theta_0$ in the vicinity of $u_*$ is poorly understood so far, and the convexity of $\theta_0$ below $u_*$ is presently unclear, see Remark \ref{rem5.2} 3). We write $\wh{\theta}$ for the function
\begin{equation}\label{0.3}
\wh{\theta}(a) = \theta_0(a) \,1 \{a < \ov{u}\} + 1\{a \ge \ov{u}\}, \; a \ge 0,
\end{equation}
and denote by $\ov{\theta}_0$ the right-continuous modification of $\theta_0$. One has
\begin{equation}\label{0.4}
\ov{\theta}_0 \le \wh{\theta} \quad \mbox{(with equality of the two functions if $\ov{u}$ and $u_*$ coincide)}.
\end{equation}

\begin{center}
\psfrag{1}{$1$}
\psfrag{0}{$0$}
\psfrag{a}{$a$}
\psfrag{wt}{$\wh{\theta}$}
\psfrag{t0}{$\theta_0$}
\psfrag{u*}{$u_*$}
\psfrag{u}{$\ov{u}$}
\includegraphics[width=8cm]{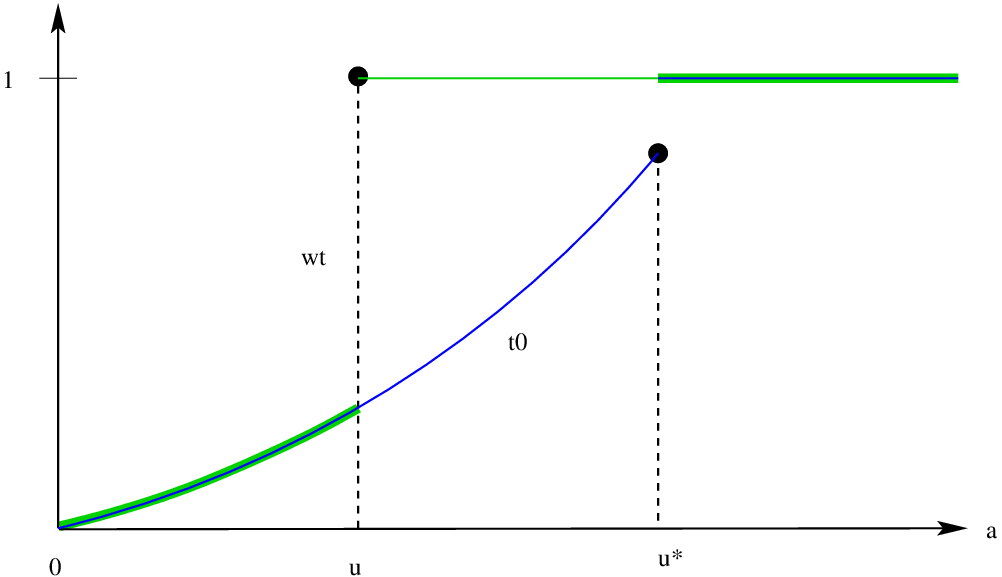}
\end{center}

\bigskip
\begin{tabular}{ll}
Fig.~1: & A heuristic sketch of the functions $\theta_0$ (with a possible but not expected\\
& jump at $u_*$) and $\wh{\theta}$. If $\ov{u} = u_*$ holds (as expected), then $\wh{\theta}$ coincides with\\
&the right-continuous modification $\ov{\theta}_0$ of $\theta_0$. 
\end{tabular}

\bigskip\medskip\n
We now describe the kind of excess disconnection event that we are interested in. Given $N \ge 1$, we consider
\begin{equation}\label{0.5}
D_N = [-N,N]^d \cap \IZ^d,
\end{equation}
and view this discrete box centered at the origin as the discrete blow-up  $(ND) \cap \IZ^d$ of the continuous shape
\begin{equation}\label{0.6}
D = [-1,1]^d \; (\subseteq \IR^d).
\end{equation}

\n
In addition, for $r \ge 0$, we write $S_r = \{x \in \IZ^d$; $|x|_\infty = r\}$ for the set of points in $\IZ^d$ with sup-norm equal to $r$, and define
\begin{equation}\label{0.7}
\mbox{$\cC^u_r =$ the connected component of $S_r$ in $\cV^u \cup S_r$ (so $S_r \subseteq \cC^u_r$ by convention)}.
\end{equation}

\n
Our focus lies in the set of points in $D_N$ that get disconnected by $\cI^u$ from $S_{2N}$, i.e.~$D_N \backslash \cC^u_{2N}$. We also consider its subset $D_N \backslash \cC_N^u$ of points in the interior of $D_N$ disconnected by $\cI^u$ from $S_N$. We are interested in their ``excessive presence''. Specifically, we consider
\begin{equation}\label{0.8}
\nu \in [\theta_0(u), 1),
\end{equation}
and the excess events (where for $F$ finite subset of $\IZ^d$, $|F|$ denotes the number of points in $F$):
\begin{equation}\label{0.9}
\cA_N = \{| D_N \backslash \cC^u_{2N}| \ge \nu\, |D_N| \} \supseteq \cA^0_N = \{| D_N \backslash \cC^u_N | \ge \nu \,|D_N|\}.
\end{equation}

\n
An asymptotic lower bound for $\IP[\cA^0_N]$ was derived in (6.32) of \cite{Szni21a}. Together with Theorem 2 of \cite{Szni21} it shows that for $0 < u < u_*$ and $\nu \in [\theta_0(u),1)$,
\begin{align}
& \liminf\limits_N \; \mbox{\f $\dis\frac{1}{N^{d-2}}$}\; \log \IP [\cA^0_N] \ge - \ov{J}_{u,\nu}, \; \mbox{where} \label{0.10}
\\[1ex]
&\ov{J}_{u,\nu} = \min \Big\{\mbox{\f $\dis\frac{1}{2d}$}  \;\dis\int_{\IR^d} | \nabla \varphi |^2 dz; \; \varphi \ge 0, \varphi \in D^1 (\IR^d), \; \strokedint_D \ov{\theta}_0\big((\sqrt{u} + \varphi)^2\big) \,dz \ge \nu\Big\}, \label{0.11}
\end{align}

\n
with $\strokedint_D \dots$ the normalized integral $\frac{1}{|D|} \int_D \dots$, $|D| = 2^d$ the Lebesgue measure of $D$, and $D^1(\IR^d)$  the space of locally integrable functions $f$ on $\IR^d$ with finite Dirichlet energy that decay at infinity, i.e.~such that $\{|f| > a\}$ has finite Lebesgue measure for $a > 0$, see Chapter 8 of \cite{LiebLoss01}.

\medskip
Our main result in this work is Theorem \ref{theo4.1} that we describe further below. Its principal application is Theorem \ref{theo5.1}, which states that for  $0 < u < \ov{u}$ and $\nu \in [\theta_0(u), 1)$,
\begin{align}
& \limsup\limits_N \; \mbox{\f $\dis\frac{1}{N^{d-2}}$} \;\log \IP [\cA_N] \le - \wh{J}_{u,\nu}, \; \mbox{where} \label{0.12}
\\[1ex]
&\wh{J}_{u,\nu} = \min \Big\{\mbox{\f $\dis\frac{1}{2d}$}  \;\dis\int_{\IR^d} | \nabla \varphi |^2 dz; \; \varphi \ge 0, \varphi \in D^1 (\IR^d), \; \strokedint_D \wh{\theta}\big((\sqrt{u} + \varphi)^2\big) \,dz \ge \nu\Big\}, \label{0.13}
\end{align}
(the existence of minimizers is shown as in Theorem 2 of \cite{Szni21}).

\medskip
In particular, if as expected the equality $\ov{u} = u_*$ holds, the combination of (\ref{0.10}) and (\ref{0.13}) proves that
\begin{equation}\label{0.14}
\begin{array}{l}
\lim\limits_N  \mbox{\f $\dis\frac{1}{N^{d-2}}$}\; \log \IP [\cA^0_N] = \lim\limits_N \; \mbox{\f $\dis\frac{1}{N^{d-2}}$}\, \log \IP [\cA_N] = - \ov{J}_{u,\nu},
\\[1ex]
\mbox{for $0 < u < u_*$ and $\nu \in [\theta_0(u),1)$} .
\end{array}
\end{equation}

\n
As often the case for large deviation asymptotics, the derivation of the lower bound involves a ``scenario that produces the deviant event of interest''. This is indeed the case in (\ref{0.10}), which is proved through the change of probability method. In this light the quantity $(\sqrt{u} + \varphi)^2 (\frac{\cdot}{N})$ can heuristically be interpreted as the slowly varying level of {\it tilted interlacements} (see \cite{LiSzni14}) that appear in the derivation of the lower bound (see Section 4 and Remark 6.6 2) of \cite{Szni21a}). The minimizers in (\ref{0.11}) do not exceed the value $\sqrt{u}_* - \sqrt{u}$, see Theorem 2 of \cite{Szni21}. The regions where the minimizers of (\ref{0.11}) reach the value $\sqrt{u}_* - \sqrt{u}$, if they exist, could reflect the presence of droplets secluded by the interlacements that might share the burden of creating an excess fraction of volume of disconnected points in $D_N$. This occurrence, which does not happen when $\nu$ is small, see \cite{Szni21b}, but might take place when $\nu$ is close to $1$, would exhibit some similarity to the Wulff droplet in the case of Bernoulli percolation or for the Ising model, see in particular Theorem 2.12 of \cite{Cerf00}, and also \cite{Bodi99}. Additional information on the behavior of $\theta_0$ near $u_*$ would be helpful in this matter.  We also refer to Remark \ref{rem5.2} 3).

\medskip
Theorem \ref{theo5.1} that proves (\ref{0.12}) is a significant improvement on the main Theorem 4.3 of \cite{Szni21b}. It replaces the function $\theta^*$ from \cite{Szni21b} (defined similarly as $\wh{\theta}$ with $(\sqrt{u} + c_0 (\sqrt{\ov{u}} - \sqrt{u}))^2 < \ov{u}$ playing the role of $\ov{u}$) by the smaller function $\wh{\theta}$ in the variational problem (\ref{0.13}). This not only leads to a sharper asymptotic upper bound, but also if the equality $\ov{u} = u_*$ holds, to the true exponential rate of decay. Informally, Theorem \ref{theo5.1} replaces the dimension dependent constant $c_0 \in (0,1)$ from \cite{Szni21b} by the value $1$.

\medskip
In this light Theorem \ref{theo4.1} is the central result of the present article. It subsumes Theorem 3.1 of \cite{Szni21b}. Theorem \ref{theo4.1} pertains to the construction of a coarse grained random set $C_\omega$, which plays a crucial role in the assignment of a cost to the bubble set $Bub$, see (\ref{1.26}). Following \cite{Szni21b}, the bubble set is obtained by paving $D_N$ with boxes $B_1$ of size $L_1 \simeq N^{2/d}$ and retaining those $B_1$-boxes that are not met ``deep inside'' by a random set $\cU_1$, see (\ref{1.22}). This random set $\cU_1$, following \cite{NitzSzni20}, is obtained by an exploration starting outside $[-3N,3N]^d$ with $B_0$-boxes of size $L_0 \simeq N^{\frac{1}{d-1}}$ (much smaller than $L_1$) that are ``good'' (so-called $(\alpha, \beta, \gamma)$-good, where $\alpha > \beta > \gamma$ lie in $(u,\ov{u})$ and should be thought of as being close to $\ov{u}$), and have a ``local level'' below $\ov{u}$ (actually below $\beta$). The random set $\cU_1$ brings along a profusion of ``highways'' in the vacant set $\cV^u$ and permits to exit $[-2N, 2N]^d$ and thus reach $S_{2N}$ when starting in $D_N$.

\begin{center}
\psfrag{DN}{$D_N$}
\includegraphics[width=6cm]{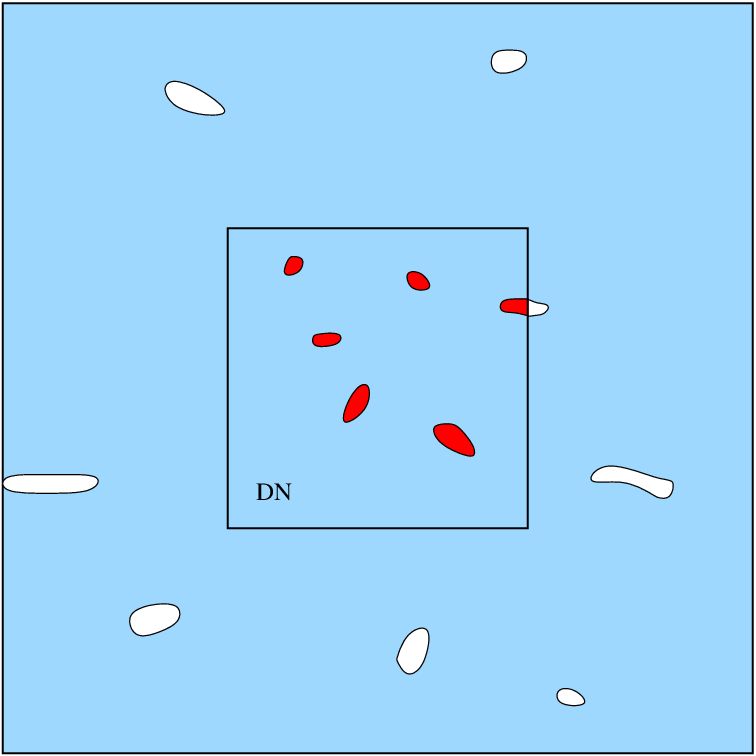}
\end{center}

\medskip
\begin{center}
\begin{tabular}{ll}
Fig.~2: & An informal illustration of the bubble set $B u b$ from (\ref{1.26}) in red. \\
& The light blue region consists of $B_1$-boxes where the random \\
& set $\cU_1$ in (\ref{1.22}) enters deeply enough in $B_1$.
\end{tabular}
\end{center}

\n
The bubble set is quite irregular and lacks inner depth. In Theorem \ref{theo4.1} we obtain a pivotal improvement on the results of Section 3 of \cite{Szni21b}. Given an arbitrary $a \in (0,1)$ we construct a random set $C_\omega$ contained in $[-4N, 4N]^d$. The set $C_\omega$ can take at most $\exp\{o(N^{d-2})\}$ possible shapes,  further, when thickened in scale $L_1$ it has a small volume compared to $|D_N|$, it is made of well-spaced $B_0$-boxes that are $(\alpha,\beta,\gamma)$-good with local level bigger or equal to $\beta$, and (see (\ref{4.5}) v)):
\begin{equation}\label{0.15}
\begin{array}{l}
\mbox{the set of points in the bubble set $Bub$, where the equilibrium potential $h_{C_\omega}$}\\
\mbox{of $C_\omega$ takes a value smaller than $a$ has small volume relative to $|D_N|$}.
\end{array}
\end{equation}

\n
The crucial difference, when compared to Section 3 of \cite{Szni21b}, is that we can now choose $a$ in (\ref{0.15}) arbitrarily close to $1$, whereas in \cite{Szni21b} $a$ was at most equal to the dimensional constant $c_0 \in (0,1)$ of \cite{Szni21b}. We refer to the discussion below Theorem \ref{theo4.1} for an outline of its proof. Quite interestingly the construction of the coarse grained random set $C_\omega$ brings in full swing both the {\it method of enlargement of obstacles}, specifically the capacity and volume estimates from Chapter 4 \S 3 of \cite{Szni98a}, or Section 2 of \cite{Szni97a}, as well as the {\it resonance sets} of \cite{NitzSzni20}, \cite{ChiaNitz}. We refer to Sections 2 and 3 for an implementation of these concepts in the present context.

\medskip
Coming back to Theorem \ref{theo5.1}, its proof can be adapted to the case where one replaces $\cC^u_{2N}$ by $\cC^u_{mN}$ with $m > 2$ integer in the definition of $\cA_N$  (hence leading to a larger event). However, the case where $\cC^u_{2N}$ is replaced by $\cC^u_\infty$ (the infinite cluster) remains open, see Remark \ref{rem5.2} 2). Further, Theorem \ref{theo5.1} has a direct application to the simple random walk. Informally, the simple random walk corresponds to the singular limit $u \r 0$ for random interlacements, see Section 7 of \cite{Szni15} or the end of Section 6 of \cite{Szni17}. In Corollary \ref{cor5.3} we obtain an asymptotic upper bound on the exponential rate of decay of the probability that the trajectory of the simple random walk disconnects a positive fraction $\nu$ of sites of $D_N$ from $S_{2N}$. It is an open question whether a matching asymptotic lower bound holds as well, see Remark \ref{rem5.4}.

\medskip
We should also mention that combined with \cite{Szni15}  the methods developed here and in \cite{Szni21b} ought to be pertinent to handle similar questions in the context of the level-set percolation of the Gaussian free field. Denoting by $\phi$ the Gaussian free field on $\IZ^d$, $d \ge 3$, and for $h$ in $\IR$ by $\cC^{\ge h}_{2N}$ the connected component of $S_{2N}$ in $\{\phi \ge h\} \cup S_{2N}$, one would now look at the points of $D_N$ that get disconnected from $S_{2N}$ by the sub-level set $\{\phi < h\}$, namely $D_N\, \backslash \,\cC^{\ge h}_{2N}$. In this model, the critical value $h_*$ for the percolation of $\{\phi \ge h\}$ lies in $(0,\infty)$, see \cite{DrewPrevRodr18}, and the value $\ov{h} \le h_*$ from \cite{Szni15} corresponding to the strong percolation of $\{\phi \ge h\}$ when $h < \ov{h}$, is known to coincide with $h_*$ by \cite{{DumiGoswRodrSeve20}}. Thus for $h < h_*$ and $\nu \in [\theta_0^{G} (h), 1)$ (with $\theta_0^{G} (v) = \IP^{G} [0\; {\phiv} \infty]$ for $v \in \IR$), one would expect the asymptotics
\begin{equation}\label{0.15a}
\begin{array}{l}
\lim\limits_N \; \mbox{\f $\dis\frac{1}{N^{d-2}}$} \; \log \IP^{G} [\{|D_N \, \backslash \, \cC^{\ge h}_{2N}| \ge \nu \,|D_N|\}] = -\ov{J}\,^{G}_{h,\nu}, \; \mbox{with}
\\[2ex]
\ov{J}\,^{G}_{h,\nu} = \min\Big\{ \mbox{\f $\dis\frac{1}{4d}$} \; \dis\int_{\IR^d} \;|\nabla \varphi |^2 \, dz; \; \varphi \ge 0, \varphi \in D^1(\IR^d), \, \dis\strokedint_D \;\ov{\theta}\,^{G}_{0}  (h + \varphi) \,dz \ge \nu\Big\}
\end{array}
\end{equation}

\n
where $\ov{\theta}\,^{G}_{0}$ stands for the right-continuous modification of $\theta^{G}_{0}$ (which only possibly differs from $\theta^G_0$ at $h_*$, see Lemma A.1 of \cite{Abac19}, and is expected to coincide with $\theta^{G}_{0}$).

\medskip
We will now describe the organization of the article. Section 1 collects some notation and recalls various facts about the simple random walk, potential theory, and random interlacements. It also recalls some lemmas from \cite{Szni21b}, in particular Lemmas 1.1 and 1.2 which will be used in Chapter 4 in the construction of $C_\omega$ (Lemma 1.2 is actually based on \cite{AsseScha20}). Section 2 develops {\it capacity and volume estimates} for certain {\it rarefied boxes} originating from the method of enlargement of obstacles, see Chap. 4 \S 3 of \cite{Szni98a} or Section 2 of \cite{Szni97a}. Section 3 contains an adaptation to our set-up of controls attached to the {\it resonance sets} developed in \cite{NitzSzni20} and \cite{ChiaNitz}. Section 4 is the main body of the article. It develops the construction of the coarse grained random set $C_\omega$, see Theorem \ref{theo4.1}. Then Section 5 builds on Theorem \ref{theo4.1} and the results in Section 4 of \cite{Szni21b}. Theorem \ref{theo5.1} proves the main asymptotic upper bound corresponding to (\ref{0.12}). The application to the simple random walk is contained in Corollary \ref{cor5.3}. The Appendix sketches the proof of Proposition \ref{prop3.1}, which pertains to the resonance sets considered here.

\medskip
Finally, let us state our convention concerning constants. Throughout the article we denote by $c, \wt{c}, c'$ positive constants changing from place to place that simply depend on the dimension $d$. Numbered constants such as $c_0,c_1,c_2,\dots$ refer to the value corresponding to their first appearance in the text. The dependence on additional parameters appears in the notation.

\bigskip\n
{\bf Acknowledgements:} The author wishes to thank Alessio Figalli for his very helpful comments about the minimizers of the variational problem in (\ref{0.11}) and the issue of knowing whether they reach the maximum value $\sqrt{u}_* - \sqrt{u}$ for $\nu$ close to $1$, see also Remark \ref{rem5.2} 3).

\section{Notation, useful results, and random sets}
\setcounter{equation}{0}

In this section we introduce further notation. We also collect several facts concerning random walks, potential theory, and random interlacements. We introduce the length scales $L_0$ and $L_1$ see (\ref{1.9}), (\ref{1.10}), as well as the random sets $\cU_1, \cU_0 = \IZ^d \backslash \cU_1$, see (\ref{1.22}), (\ref{1.23}), and the ``bubble set'' $B u b$, see (\ref{1.26}). They play an important role in this article. We also recall Lemmas 1.1 and 1.2 from \cite{Szni21b} (Lemma 1.2 is in essence due to \cite{AsseScha20}): they will enter the construction of the random set $C_\omega$ in Section 4.

\medskip
We begin with some notation. We denote by $\IN = \{0,1,2, \dots\}$ the set of non-negative integers and by $\IN_* = \{1,2, \dots\}$ the set of positive integers. For $(a_n)_{n \ge 1}$ and $(b_n)_{n \ge 1}$ positive sequences, $a_n \gg b_n$ or $b_n = o(a_n)$ means that $\lim_n b_n / a_n = 0$. We write $|\cdot |_1$, $|\cdot |$, and $|  \cdot |_\infty$ for the $\ell_1$-norm, the Euclidean norm, and the supremum norm on $\IR^d$. Throughout we tacitly assume that $d \ge 3$. Given $x \in \IZ^d$ and $r \ge 0$, we let $B(x,r) = \{y \in \IZ^d; |y - x|_\infty \le r\}$ stand for the closed ball of radius $r$ around $x$ in the supremum distance. Note that $D_N$ in (\ref{0.5}) coincides with $B(0,N)$.  Given $L \ge 1$ integer, we say that a subset $B$ of $\IZ^d$ is a box of size $L$ if $B = x + \IZ^d \cap [0, L)^d$ for an $x$ in $\IZ^d$ (which actually is unique). We write $x_B$ for this $x$ and refer to it as the base point of $B$. We sometimes write $B = x_B + [0,L)^d$ in place of $x_B + \IZ^d \cap [0,L)^d$ when no confusion arises. We write $A \subset\subset \IZ^d$ to state that $A$ is a finite subset of $\IZ^d$. We denote by $\partial A = \{y = \IZ^d \backslash A$; $\exists x \in A$, $|y - x| = 1\}$ and $\partial_i \,A = \{x \in A$; $\exists y \in \IZ^d \backslash A$; $|y - x| = 1\}$ the boundary and the inner boundary of $A$. When $f,g$ are functions on $\IZ^d$, we write $\langle f,g \rangle = \Sigma_{x \in \IZ^d} f(x) \ g(x)$ when the sum is absolutely convergent. We also use the notation $\langle \rho, f \rangle$ for the integral of a function $f$ (on an arbitrary space) with respect to a measure $\rho$ when this quantity is meaningful.

\medskip
Concerning connectivity properties, we say that  $x,y$ in $\IZ^d$ are neighbors when $|y - x| = 1$ and we call $\pi: \{0,\dots , n\} \rightarrow \IZ^d$ a path when $\pi(i)$ and $\pi(i-1)$ are neighbors for $1 \le i \le n$. For $A,B,U$ subsets of $\IZ^d$ we say that $A$ and $B$ are connected in $U$ and write $A \stackrel{U}{\longleftrightarrow} B$ when there is a path with values in $U$, which starts in $A$ and ends in $B$. When no such path exists we say that $A$ and $B$ are not connected in $U$ and we write $A  \U B$.

\medskip
We then recall some notation concerning the continuous-time simple random walk. For $U \subseteq \IZ^d$, we write $\Gamma(U)$ for the set of right-continuous, piecewise constant functions from $[0,\infty)$ to $U \cup \partial U$ with finitely many jumps on any finite interval that remain constant after their first visit to $\partial U$. We denote by $(X_t)_{t \ge 0}$ the canonical process on $\Gamma(U)$. For $U \subset \subset \IZ^d$ the space $\Gamma(U)$ conveniently carries the law of certain excursions contained in the trajectories of the interlacements. We also view the law of $P_x$ of the continuous-time simple random walk on $\IZ^d$ with unit jump rate, starting at $x \in \IZ^d$, as a measure on $\Gamma (\IZ^d)$. We write $E_x$ for the corresponding expectation. We denote by $(\cF_t)_{t \ge 0}$ the canonical right-continuous filtration, and by $(\theta_t)_{t \ge 0}$ the canonical shift on $\Gamma(\IZ^d)$. Given $U \subseteq \IZ^d$, we write $H_U = \inf\{t \ge 0; X_t \in U\}$ and $T_U = \inf\{t \ge 0; X_t \notin U\}$ for the respective entrance time in $U$ and exit time from  $U$. Further, we let $\wt{H}_U$ stand for the hitting time of $U$, that is, the first time after the first jump of $X_\point$ when $X_\point$ enters $U$.

\medskip
We write $g(\cdot,\cdot)$ for the Green function of the simple random walk and $g_U(\cdot,\cdot)$ for the Green function of the walk killed upon leaving $U (\subseteq \IZ^d)$:
\begin{equation}\label{1.1}
g(x,y) = E_x \Big[\dis\int^\infty_0 1\{X_s = y\} \,ds\Big], \; g_U(x,y) = E_x \Big[\dis\int^{T_U}_0 1\{X_s = y\} \,ds\Big], \; \mbox{for $x,y \in \IZ^d$}.
\end{equation}

\n
Both $g(\cdot,\cdot)$ and $g_U(\cdot,\cdot)$ are known to be finite and symmetric, and $g_U(\cdot,\cdot)$ vanishes if one of its arguments does not belong to $U$. When $f$ is a function on $\IZ^d$ such that $\Sigma_{y \in \IZ^d}\, g(x,y) \,|f(y)| < \infty$ for all $x \in \IZ^d$, we write 
\begin{equation}\label{1.2}
G f(x) = \dis\sum\limits_{y \in \IZ^d} g(x,y) \,f(y), \; \mbox{for $x \in \IZ^d$}.
\end{equation}

\n
Likewise, when $f$ is a function on $\IZ^d$ such that $\Sigma_{y \in \IZ^d} \,g_U(x,y) \,|f(y)| < \infty$ for all $x$ in $\IZ^d$, we define $G_U f(x)$ in an analogous fashion with $g_U(\cdot,\cdot)$ in place of $g(\cdot,\cdot)$.

\medskip
Due to translation invariance $g(x,y) = g(x-y, 0)$ and one knows that $g(x,y) \sim \frac{d}{2} \,\Gamma(\frac{d}{2} - 1) \,\pi^{- d/2} |y - x|^{2-d}$, as $|y - x| \rightarrow \infty$ (see Theorem 1.5.4, p.~31 of \cite{Lawl91}). We denote by $c_*$ the positive constant
\begin{equation}\label{1.3}
c_* = \sup\{g(x,y) \,|y - x|^{d-2}; \,x,y \in \IZ^d\} \in (0,\infty).
\end{equation}
We will also use the fact that for a suitable dimension dependent constant $c_\Delta > 0$
\begin{equation}\label{1.4}
\mbox{for all $R \ge 1$ and $x,y \in B(0,R), \;\; g(x,y) \ge g_{B(0,2R)} (x,y) \ge c_\Delta\, g(x,y)$}
\end{equation}
(this follows for instance by adapting the proof of Theorem 4.26 a) on p.~121 of \cite{Barl17}).

\medskip
Given $A \subset \subset \IZ^d$, we write $e_A$ for the equilibrium measure of $A$, and ${\rm cap}(A)$ for its total mass, the capacity of $A$. So $e_A(x) = P_x [\wt{H}_A = \infty] \,1_A (x)$, for $x \in \IZ^d$, and $e_A$ is supported by the inner boundary of $A$. Further, one knows that
\begin{align}
&\mbox{$G e_A = h_A$, where $h_A(x) = P_x [H_A < \infty]$, $x \in \IZ^d$, is the equilibrium potential of $A$} \label{1.5}\\
&\mbox{(in particular $P_x[H_A < \infty] \ge {\rm cap}(A) \,\inf\limits_A g(x,\cdot)$)}. \nonumber
\end{align}

\n
At the end of Section 4 we will use an identity generalizing (\ref{1.5}) to the case of the simple random walk killed outside $U \subseteq \IZ^d$ and $A$ finite subset of $U$. Then, one has $h_{A,U} = G_U \,e_{A,U}$, where $h_{A,U} (x) = P_x [H_A < T_U]$, for $x \in \IZ^d$, and $e_{A,U} (x) = P_x[T_U \le \wt{H}_A] \,1_A(x)$, for $x \in \IZ^d$.

\medskip
In the case of a box $B = [0,L)^d$ one knows (see for instance \cite{Lawl91}, p.~31) that
\begin{equation}\label{1.6}
c\,L^2 \le G \,1_B(x) \le c' \,L^2, \;\mbox{for $x \in B$ and $L \ge 1$},
\end{equation}
as well as (see (2.16) on p.~53 of \cite{Lawl91})
\begin{equation}\label{1.7}
c\,L^{d-2} \le {\rm cap} (B) \le c' \,L^{d-2}, \; \mbox{for $L \ge 1$}.
\end{equation}
We will also need the constant
\begin{equation}\label{1.8}
\wh{c} = \sup\limits_{L \ge 1} \; {\rm cap}([0,L)^d) / L^{d-2}.
\end{equation}

\n
We now introduce some length-scales. Apart from the macroscopic scale $N$ that governs the size of the box $D_N$ in (\ref{0.5}), two other length-scales play an important role:
\begin{align}
L_0 & = [N^{\frac{1}{d-1}}], \;\mbox{and} \label{1.9}
\\
L_1 & = k_N \, L_0, \; \mbox{where $k_N$ is the integer such that $k_N\, L_0 \le N^{\frac{2}{d}} (\log N)^{\frac{1}{d}} < (k_N + 1) \,L_0$}. \label{1.10}
\end{align}

\n
One has $\frac{1}{d-1} < \frac{2}{d}$ so that $k_N \r \infty$, and $L_1/ L_0 \underset{N}{\longrightarrow} \infty$, with $L_0 \sim N^{\frac{1}{d-1}}$ and $L_1 \sim N^{\frac{2}{d}} (\log N)^{\frac{1}{d}}$, as $N \r \infty$.

\medskip
We call $B_0$-box (or sometimes $L_0$-box) any box of the form
\begin{equation}\label{1.11}
B_{0,z} = z + [0,L_0)^d, \; \mbox{where $z \in \IL_0 \stackrel{\rm def}{=} L_0 \, \IZ^d$}.
\end{equation}
We often write $B_0$ to refer to a generic box $B_{0,z}$, $z \in \IL_0$ (so $z = x_{B_0}$ will be the base point of $B_0$). Likewise we call $B_1$-box (or $L_1$-box) any box of the form
\begin{equation}\label{1.12}
\mbox{$B_{1,z} = z + [0,L_1)^d$, where $z \in \IL_1 \stackrel{\rm def}{=} L_1 \, \IZ^d$ ($\subseteq \IL_0$ by (\ref{1.9}), (\ref{1.10}))},
\end{equation}
and denote by $B_1$ a generic box $B_{0,z}, z \in \IL_1$ (so that $z = x_{B_1}$ is the base point of $B_1$).

\medskip
We recall two lemmas from \cite{Szni21b}. They play an important role in the proof of the main Theorem \ref{theo4.1}, where the coarse grained random set $C_\omega$ is constructed. The combination of isoperimetric controls and the first lemma will typically enable us to extract well-spaced (good) $B_0$-boxes in the boundary of the random set $\cU_1$ (see below (\ref{1.23})), which carry substantial capacity, see (\ref{4.79}), (\ref{4.81}), (\ref{4.84}). In the statement of the next lemma, and throughout the article, the terminology {\it coordinate projection} refers to any of the $d$ canonical projections on the respective hyperplanes of points with vanishing $i$-th coordinate.

\begin{lemma}\label{lem1.1}
Given $K \ge 100 d$, $a \in (0,1)$, then for large $N$, for any box $B$ of size $L \ge L_1$ and any set $A$ union of $B_0$-boxes contained in $B$ such that for a coordinate projection $\pi$ one has
\begin{equation}\label{1.13}
|\pi (A)| \ge a  \,|B|^{\frac{d-1}{d}} ,
\end{equation}

\medskip\n
one can find a subset $\wt{A}$ of $A$, which is a union of $B_0$-boxes having base points with respective $\pi$-projections at mutual $| \cdot |_\infty$-distance at least $\ov{K} \,L_0$ (with $\ov{K} = 2 K + 3)$, and such that
\begin{equation}\label{1.14}
{\rm cap}(\wt{A}) \ge c(a) \,|B|^{\frac{d-2}{d}} \; \mbox{and} \;\; | \pi (\wt{A}) | \ge \ov{K}\,^{-(d-1)} | \pi (A) |\, .
\end{equation}
\end{lemma}

The second lemma corresponds to Lemma 1.2 of \cite{Szni21b}, and is a direct application of Theorem 1.4 of \cite{AsseScha20}. It will be used at the end of the proof of Proposition \ref{prop4.2} to extract small collections of $B_0$-boxes with substantial capacity, see (\ref{4.97}).

\begin{lemma}\label{lem1.2} $(\ov{K} = 2K + 3)$

\medskip\n
For $K, N \ge c_1$, when $\wt{A}$ is a union of $B_0$-boxes with base points that are at mutual $| \cdot |_\infty$-distance at least $\ov{K} L_0$, then there exists a union of $B_0$-boxes $A' \subseteq \wt{A}$ such that
\begin{equation}\label{1.15}
{\rm cap}(A') \ge c \;{\rm cap}(\wt{A}) \;\; \mbox{and} \;\; \dis\frac{|A'|}{|B_0|} \le c' \dis\frac{{\rm cap}(\wt{A})}{{\rm cap} (B_0)} \;.
\end{equation}
\end{lemma}

We will now collect some notation and facts concerning random interlacements. We refer to \cite{CernTeix12}, \cite{DrewRathSapo14c}, and the end of Section 1 of \cite{Szni17} for more details. The random interlacements $\cI^u$, $u \ge 0$ and the corresponding vacant sets $\cV^u = \IZ^d \backslash \cI^u$, $u \ge 0$ are defined on a probability space denoted by $(\Omega, \cA, \IP)$. In essence, $\cI^u$ corresponds to the trace left on $\IZ^d$ by a certain Poisson point process of doubly infinite trajectories modulo time-shift that tend to infinity at positive and negative infinite times, with intensity proportional to $u$. As $u$ grows $\cV^u$ becomes thinner and there is a critical value $u_* \in (0,\infty)$ such that for all $u < u_*$, $\IP$-a.s. $\cV^u$ has a unique infinite component $\cC^u_\infty$, and for $u > u_*$ all components of $\cV^u$ are finite, see \cite{Szni10}, \cite{SidoSzni09}, \cite{Teix09a}, as well as the monographs \cite{CernTeix12}, \cite{DrewRathSapo14c}.

\medskip
In this work we are mainly interested in the strongly percolative regime of $\cV^u$ that corresponds to
\begin{equation}\label{1.16}
u < \ov{u},
\end{equation}

\n
where we refer to (2.3) of \cite{Szni17} or (1.26) of \cite{Szni21b} for the precise definition of $\ov{u}$. Informally, (\ref{1.16}) corresponds to a regime of local presence and uniqueness of the infinite cluster $\cC^u_\infty$ in $\cV^u$. One knows by \cite{DrewRathSapo14a} that $\ov{u} > 0$  and that $\ov{u} \le u_*$, see (2.4), (2.6) in \cite{Szni17}. The equality $\ov{u} = u_*$ is expected but presently open. In the closely related model of the level-set percolation for the Gaussian free field, the corresponding equality has been shown in the recent work \cite{DumiGoswRodrSeve20}.

\medskip
We now introduce some additional boxes related to the length scale $L_0$, which take part in the definition of the important random set $\cU_1$ (see (1.40) of \cite{Szni21b}, and also (4.27) of \cite{NitzSzni20}, as well as (\ref{1.22}) below). Throughout $K$ implicitly satisfies
\begin{equation}\label{1.17}
K \ge 100 .
\end{equation}
We consider the boxes
\begin{equation}\label{1.18}
B_{0,z} = z + [0,L_0)^d \subseteq D_{0,z} = z + [-3 L_0, 4 L_0)^d \subseteq U_{0,z} = z + [- K L_0 + 1, K L_0 - 1)^d,
\end{equation}
with $z \in \IL_0$ ($= L_0 \IZ^d$, see (\ref{1.11})).

\medskip
Given a box $B_0$ as above and the corresponding $D_0$, we consider the successive excursions in the interlacements that go from $D_0$ to $\partial U_0$ (see (1.41) of \cite{Szni17}) and write (see (1.42) and (2.14) of \cite{Szni17}):
\begin{align}
N_v (D_0) = & \;\mbox{the number of excursions from $D_0$ to $\partial U_0$ in the interlacements} \label{1.19}
\\
& \; \mbox{trajectories with level at most $v$, for $v \ge 0$}. \nonumber
\end{align}
Given
\begin{equation}\label{1.20}
\mbox{$\alpha > \beta > \gamma$ in $(0,\ov{u})$},
\end{equation}

\medskip\n
the notion of an {\it $(\alpha,\beta,\gamma)$-good} box $B_{0,z}$ plays an important role in the definition of $\cU_1$. We refer to (2.11) - (2.13) of \cite{Szni17}, see also (1.38) of \cite{Szni21b} for the precise definition. The details of the definition will not be important here. In essence, one looks at the (naturally ordered) excursions from $D_{0,z}$ to $\partial U_{0,z}$ in the trajectories of the interlacements. For an $(\alpha,\beta,\gamma)$-good box $B_{0,z}$ the complement of the first $\alpha\, {\rm cap}(D_{0,z})$ excursions leaves in $B_{0,z}$ at least one connected set with sup-norm diameter at least $L_0/10$, which is connected to any similar components in neighboring boxes of $B_{0,z}$ via paths in $D_{0,z}$ avoiding the first $\beta \,{\rm cap}(D_{0,z})$ excursions from $D_{0,z}$ to $\partial U_{0,z}$. In addition, the first $\beta\, {\rm cap}(D_{0,z})$ excursions spend a substantial ``local time'' on the inner boundary of $D_{0,z}$ of an amount at least  $\gamma\, {\rm cap}(D_{0,z})$. We refer to $B_0$-boxes that are not $(\alpha,\beta,\gamma)$-good as {\it $(\alpha,\beta,\gamma)$-bad} boxes.

\medskip
We now fix a level $u$ as in (\ref{0.1}), that is
\begin{equation}\label{1.21}
0 < u < \ov{u} .
\end{equation}
Following (4.27) of \cite{NitzSzni20} or (1.40) of \cite{Szni21b}, we introduce the important random set
\begin{align}
\cU_1 = & \;\mbox{the union of $L_0$-boxes $B_0$ that are either contained in $([-3N,3N]^d)^c$ or} \label{1.22}
\\
&\; \mbox{linked to an $L_0$-box contained in  $([-3N,3N]^d)^c$ by a path of $L_0$-boxes } \nonumber
\\
&\; \mbox{$B_{0,z_i}$, $0 \le i \le n$, which are all, except maybe for the last one, $(\alpha, \beta, \gamma)$-good} \nonumber
\\
&\;\mbox{and such that $N_u (D_{0,z_i}) < \beta \,{\rm cap} (D_{0,z_i})$.}\nonumber
\end{align}
We then define
\begin{equation}\label{1.23}
\cU_0 = \cU_1^c \; (\subseteq [-3N - L_0, 3 N + L_0]^d).
\end{equation}

\n
We use the notation $\partial_{B_0} \,\cU_1$ to refer to the (random) collection of $B_0$-boxes that are not contained in $\cU_1$ but are neighbor of a $B_0$-box in $\cU_1$. Note that when $B_0$ is $(\alpha,\beta,\gamma)$-good and belongs to $\partial_{B_0} \,\cU_1$, then necessarily $N_u (D_0) \ge \beta \,{\rm cap} (D_0)$ (otherwise $B_0$ would belong to $\cU_1$). Although we will not need this fact here, let us mention that the random set $\cU_1$ provides paths in $\cV^u$ going from any $B_0$ in $\cU_1 \cap D_N$ to $([-3N + L_0, 3 N - L_0]^d)^c$. Such paths necessarily meet the inner boundary $S_{2N}$ of $B(0,2N)$, see below (1.40) of \cite{Szni21b}.

\medskip
We also recall a statement (see Lemma 1.4 of \cite{Szni21b}) concerning the rarity of $(\alpha,\beta,\gamma)$-bad boxes.
\begin{lemma}\label{lem1.3}
Given $K \ge c_2(\alpha, \beta,\gamma)$ there exists a non-negative function $\rho(L)$ depending on $\alpha,\beta,\gamma,K$ satisfying $\lim_L \rho(L) = 0$, such that
\begin{equation}\label{1.24}
\begin{array}{l}
\mbox{$\lim\limits_N \; \mbox{\f $\dis\frac{1}{N^{d-2}}$} \; \log \IP[\cB_N] = - \infty$, where $\cB_N$ stands for the event}
\\[1.5ex]
\mbox{$\cB_N = \{$there are more than $\rho(L_0) \,N^{d-2}$ many $(\alpha,\beta,\gamma)$-bad $B_0$-boxes}
\\
\qquad \;\;\;  \mbox{intersecting $[-3N,3N]^d\}$}.
\end{array}
\end{equation}
\end{lemma}

We then proceed with the definition of the bubble set. Given an $L_1$-box $B_1$, we denote by ${\rm Deep} \,B_1$ the set (see (1.46) of \cite{Szni21b})
\begin{equation}\label{1.25}
{\rm Deep} \,B_1 = \bigcup\limits_{z \in \IL_0, D_{0,z} \subseteq B_1} B_{0,z},
\end{equation}

\n
which in essence is obtained by ``peeling off'' a shell of depth $3L_0$ from the surface of $B_1$, thus only keeping $B_0$-boxes such that the corresponding $D_0$ is contained in $B_1$.

\medskip
One then defines the {\it bubble set}
\begin{equation}\label{1.26}
Bub  = \bigcup\limits_{B_1 \subseteq D_N, \,\cU_1 \cap {\rm Deep} \,B_1 = \phi} B_1
\end{equation}

\n
that is the union of the $B_1$-boxes in $D_N$ such that $\cU_1$ does not reach ${\rm Deep} \,B_1$, i.e.~such that ${\rm Deep} \,B_1 \subseteq \cU_0$, see Figure 2.

\section{Volume estimates for rarefied boxes}
\setcounter{equation}{0}

In this section we bring into play a notion of {\it rarefied boxes}. With the help of capacity and volume estimates originally developed in the context of the {\it method of enlargement of obstacles}, see Chapter 4 \S 3 of \cite{Szni98a}, or Theorem 2.1 of \cite{Szni97a}, we derive volume controls in Proposition \ref{prop2.1}. They play an important role in the proof of Theorem \ref{theo4.1}, when constructing the coarse grained random set $C_\omega$ accounting for the bubble set. The volume estimates of Proposition \ref{prop2.1} are applied in Section 4 to boxes of so-called Types $\Ib$ and $\IB$, see (\ref{4.13}), (\ref{4.16}). Informally, they correspond to certain nearly macroscopic boxes ``at the boundary of $\cU_0$'' that intersect the bubble set. The estimates in the present section enable us to discard the so-called {\it rarefied boxes} of Types $\Ib$ and $\IB$ in Section 4, see (\ref{4.37}), (\ref{4.49}). On the other hand, the {\it substantial} (i.e.~non-rarefied) {\it boxes} of Types $\Ib$ and $\IB$ fulfill a kind of Wiener criterion, which ensures that $(\alpha,\beta,\gamma)$-good boxes of $\partial_{B_0} \,\cU_1$ are ``present on many scales'', see (\ref{4.36}), (\ref{4.48}). As an aside, the notion of rarefied boxes that we consider in this section is substantially more refined than that which was used in Section 3 of \cite{Szni21b}.

\medskip
We first introduce an ``$M$-adic decomposition of $\IZ^d$'' where $L_1$, see (\ref{1.10}), corresponds to the smallest scale, and $N$ roughly to the largest scale. More precisely, we consider a dyadic integer $M > 4$, solely depending on $d$, such that
\begin{equation}\label{2.1}
M^2 > 3^d + 1 \quad \mbox{(with $M = 2^b$, $b$ integer)}.
\end{equation}

\n
Having $M$ a dyadic number will be convenient in the next section when discussing {\it resonance sets} (see also the Appendix).

\medskip
As mentioned above, the smallest scale under consideration corresponds to $L_1$ and the largest scale corresponds to $M^{\ell_N} L_1$, where
\begin{equation}\label{2.2}
M^{\ell_N} L_1 \le N < M^{(\ell_N + 1)} L_1.
\end{equation}

\n
We view things from the point of view of the top scale and $0 \le \ell \le \ell_N$ labels the depth with respect to the top scale. For such $\ell$ we set (hopefully there should be no confusion with the notation for random interlacements, where the level always appear as a superscript, see above (\ref{0.1}))
\begin{align}
\cI_\ell = &\;\mbox{the collection of $M$-adic boxes of depth $\ell$, i.e.~of boxes of the form} \label{2.3}
\\
& \; \{  M^{\ell_N - \ell} L_1 \, z + [0, M^{\ell_N - \ell} \,L_1)^d \} \cap \IZ^d, \;\mbox{where $z \in \IZ^d$}. \nonumber
\end{align}

\n
The collections $\cI_\ell$, $0 \le \ell \le \ell_N$ are naturally nested; $\cI_{\ell_N}$ corresponds to the collection of $B_1$-boxes and $\cI_0$ to boxes of approximate size $N$. Given $\ell$ as above and $B \in \cI_\ell$, the ``tower above $B$'' stands for the collection of boxes $B' \in \bigcup_{0 \le \ell' \le \ell} \cI_{\ell'}$ such that $B' \supseteq B$. Further,
\begin{equation}\label{2.4}
\begin{array}{l}
\mbox{for $0 \le \ell' \le \ell \le \ell_N$ and $B \in \cI_\ell$, we denote by $B^{(\ell')}$ the unique box of $\cI_{\ell'}$ }
\\
\mbox{that contains $B$}.
\end{array}
\end{equation}
We consider some depth
\begin{equation}\label{2.5}
1 \le q \le \ell_N,
\end{equation}
as well as a set
\begin{equation}\label{2.6}
F \subseteq \IZ^d .
\end{equation}

\n
For the applications given in Section 4, the depth $q$ will correspond to $\wt{p}$ in (\ref{4.9}) and $F$ will correspond to $\bigcup_{j \in \cG} \, B'_j$ in the context of Type $\Ib$ boxes, see (\ref{4.35}), (\ref{4.13}), and to $\bigcup_{1 \le j \le J^{\IB}} \Delta'_j$ in the context of Type $\IB$ boxes, see (\ref{4.47}), (\ref{4.16}).

\medskip
We then introduce the dimension dependent constants
\begin{equation}\label{2.7}
\mbox{$\delta = \mbox{\f $\dis\frac{1}{2}$} \;\delta_0$, where $\delta_0 = (4c_* \,M^{2d-2})^{-1}$, with $c_*$ from (\ref{1.3})}.
\end{equation}
Given a box $B \in \cI_q$, we say that
\begin{equation}\label{2.8}
\mbox{$B$ is {\it rarefied} if $\dis\sum\limits_{B \subseteq \wt{B} \subseteq B^{(1)}} \; \dis\frac{{\rm cap} (\wt{B} \cap F)}{|\wt{B}|^{\frac{d-2}{d}}} < \delta q$},
\end{equation}

\n
where the sum runs over the boxes $\wt{B}$ in the tower above $B$ that are contained in $B^{(1)}$ (see (\ref{2.4}) for notation).

\medskip
The main object of this section is
\begin{proposition}\label{prop2.1}
Assume that $q$ and $F$ are as in (\ref{2.5}), (\ref{2.6}) and that
\begin{equation}\label{2.9}
L_1^{d-2} g(0,0) > 2 c_* \,M^d.
\end{equation}
Then, for any box $\Delta \in \cI_0$, one has
\begin{equation}\label{2.10}
\dis\sum\limits_{B \in \cI_q, {\rm rarefied}, B \subseteq \Delta} |B \cap F| \le c_3 \,|\Delta | \;\Big(\mbox{\f $\dis\frac{M^2}{3^d + 1}$}\Big)^{-q/2}.
\end{equation}
\end{proposition}

\begin{proof}
The next lemma contains the key control that will then be iterated over scales. In the words of Remark 3.3 of p.~172 of \cite{Szni98a}, it reflects growth (with $\frac{M^2}{3^d + 1} > 1$) and saturation (corresponding to the truncation by $\wt{\eta}_0)$ in the evolution of the properly normalized capacities from one scale to the next.

\begin{lemma}\label{lem2.2} (under (\ref{2.9}))

\medskip\n
Consider $0 \le \ell \le \ell + 1 \le q$ and $\ov{B} \in \cI_\ell$. Denote by $\wh{B} \subseteq \ov{B}$ a generic subbox in $\cI_{\ell + 1}$ and set (with $c_*$ as in (\ref{1.3}))
\begin{equation}\label{2.11}
\mbox{$\eta_0 = (c_* \,M^d)^{-1}$ and \; $\wt{\eta}_0 =  \eta_0 / 2$}.
\end{equation}
Then one has 
\begin{equation}\label{2.12}
\dis\frac{{\rm cap} (\ov{B} \cap F)}{|\ov{B}|^{\frac{d-2}{d}}} \ge \dis\frac{M^2}{3^d + 1} \;  \dis\frac{1}{M^d} \; \dsl_{\wh{B} \subseteq \ov{B}}\;
 \Big\{ \dis\frac{{\rm cap} (\wh{B} \cap F)}{|\wh{B}|^{\frac{d-2}{d}}}  \Big\} \wedge \wt{\eta}_0 .
\end{equation}
\end{lemma}

\begin{proof}
Consider $\ov{B} \in \cI_\ell$ as above. For each $\wh{B} \subseteq \ov{B}$, $\wh{B} \in \cI_{\ell + 1}$, setting $L = |\wh{B}|^{\frac{1}{d}}$ as a shorthand notation, we first show the existence of
\begin{equation}\label{2.13}
\begin{array}{l}
\mbox{$\wh{B} ' \subseteq \wh{B} \cap F$ such that}
\\
 {\rm cap} (\wh{B} \cap F) \wedge (\eta_0 \,L^{d-2}) \ge {\rm cap} (\wh{B} ') \ge {\rm cap}(\wh{B} \cap F) \wedge \Big(\eta_0 \,L^{d-2} - \mbox{\f $\dis\frac{1}{g(0,0)}$}\Big).
\end{array}
\end{equation}
Observe that by (\ref{2.9}) and (\ref{2.11})
\begin{equation}\label{2.14}
\eta_0 \,L^{d-2} = (c_* \,M^d)^{-1} \,L^{d-2} \ge (c_* \,M^d)^{-1} L_1^{d-2} \ge 2 / g(0,0).
\end{equation}
To construct $\wh{B} '$ we note that:
\\
either ${\rm cap} (\wh{B} \cap F) \le \eta_0 \, L^{d-2}$, in which case we choose $\wh{B} ' = \wh{B} \cap F$ and (\ref{2.13}) holds, \\
or ${\rm cap} (\wh{B} \cap F) > \eta_0 \, L^{d-2} \ge 2/ g(0,0)$ (by (\ref{2.14})) and we remove from $\wh{B} \cap F$ one point at a time, decreasing at each step the capacity by an amount at most $1/g(0,0)$ (i.e.~the capacity of a point in $\IZ^d$) until the first time when the resulting capacity is smaller or equal to $\eta_0 \,L^{d-2}$. The set thus obtained has now capacity at most $\eta_0 \,L^{d-2} \le {\rm cap} (\wh{B} \cap F)$ and at least $\eta_0 \,L^{d-2} - 1/g(0,0) = {\rm cap}(\wh{B} \cap F) \wedge (\eta_0 \,L^{d-2} - 1/g(0,0)$). So choosing this set as $\wh{B} '$, the claim (\ref{2.13}) holds as well.

\medskip
So, for each $\wh{B} \subseteq \ov{B}$ with $\wh{B} \in \cI_{\ell + 1}$ with pick $\wh{B} ' \subseteq \wh{B} \cap F$ satisfying (\ref{2.13}).  Next, we consider the measure (with hopefully obvious notation)
\begin{equation}\label{2.15}
\mbox{$\nu = \dis\sum\limits_{\wh{B} \subseteq \ov{B}} \;e_{\wh{B} '}$ \;\; (with $e_{\wh{B} '}$ the equilibrium measure of $\wh{B} '$)},
\end{equation}
which is supported by the set
\begin{equation}\label{2.16}
S = \bigcup\limits_{\wh{B} \subseteq \ov{B}} \wh{B} ' \subseteq \ov{B} \cap F.
\end{equation}

\n
Then, for $x \in S$ we denote by $\Sigma_1$ the sum over the boxes $\wh{B}$ containing $x$ or neighbor (in the sup-norm sense) of the box containing $x$, and by $\Sigma_2$ the sum over the remaining boxes $\wh{B} \subseteq \ov{B}$. We then have
\begin{equation}\label{2.17}
\begin{array}{lcl}
\dsl_y g(x,y) \,\nu(y) &\hspace{-4ex}  = &\hspace{-4ex} \dsl_{\wh{B}} \! _{\,_1} \; \dsl_{y \in \wh{B}} \; g(x,y) \,e_{\wh{B}'} (y)
\\[2ex]
&\hspace{-4ex}  + &\hspace{-4ex} \dsl_{\wh{B}} \! _{\,_2} \;  \dsl_{y \in \wh{B}} \; g(x,y) \,e_{\wh{B}'} (y) \;\; \mbox{(in this sum necessarily $|x-y| \ge L$)}
\\[2ex]
&\hspace{-4ex} \stackrel{(\ref{1.5}), (\ref{1.3})}{\le} &\hspace{-2ex} 3^d + \dsl_{\wh{B}} \!\! _{\,_2} \;\; \dis\frac{c_*}{L^{d-2}} \;{\rm cap} (\wh{B}') \stackrel{(\ref{2.13})}{\le} 3^d + c_* \,\eta_0 \, M^d
\\[2ex]
& \hspace{-4ex}\stackrel{(\ref{2.11})}{=} &\hspace{-2ex} 3^d + 1 .
\end{array}
\end{equation}

\n
It follows that $G \nu (x) \le 3^d + 1$ for each $x \in S$. Since $\nu$ is supported by $S$, we find that ${\rm cap}(S) = \langle e_S,1\rangle \ge \langle e_S,  G\nu \rangle  (3^d + 1)^{-1}  = \langle G e_S, \nu \rangle (3^d + 1)^{-1} \stackrel{(\ref{1.5})}{=} \nu (\IZ^d) / (3^d + 1)$, and hence
\begin{equation}\label{2.18}
\begin{array}{lcl}
{\rm cap} (\ov{B} \cap F) & \stackrel{(\ref{2.16})}{\ge} & {\rm cap}(S) \ge \dis\frac{\nu(\IZ^d)}{3^d + 1} \stackrel{(\ref{2.15})}{=} \mbox{\f $\dis\frac{1}{3^d+1}$}\;\dsl_{\wh{B} \subseteq \ov{B}} \; {\rm cap}(\wh{B} ')
\\[2ex]
& \stackrel{(\ref{2.13})}{\ge} & \dis\frac{1}{3^d + 1} \; \dsl_{\wh{B} \subseteq \ov{B}} \; {\rm cap} (\wh{B} \cap F) \wedge \Big(\eta_0 \,L^{d-2} - \dis\frac{1}{g(0,0)}\Big)
\end{array}
\end{equation}

\n
Dividing both sides by $|\ov{B}|^{\frac{d-2}{d}} = M^{d-2} \,|\wh{B}|^{\frac{d-2}{d}} = M^{d-2} \,L^{d-2}$, and using that $\eta_0 - (g(0,0)\,L^{d-2})^{-1} \ge \wt{\eta}_0$ by (\ref{2.14}), (\ref{2.11}), the claim (\ref{2.12}) follows. This proves Lemma \ref{lem2.2}.
\end{proof}

We now proceed with the proof of Proposition \ref{prop2.1}. We define the quantities
\begin{equation}\label{2.19}
\mbox{$Y_{\ov{B}} = \dis\frac{{\rm cap}(\ov{B} \cap F)}{| \ov{B} |^{\frac{d-2}{d}}}$, for $\ov{B} \in \dis\bigcup\limits_{0 \le \ell \le q} \cI_\ell$ such that $\ov{B} \subseteq \Delta ( \in \cI_0)$}.
\end{equation}

\n
We note that the $Y_{\ov{B}}$ satisfy (with hopefully obvious notation)
\begin{equation}\label{2.20}
\left\{\begin{array}{l}
\mbox{$Y_{\ov{B}} \le \wh{c}$\;\;\,(with $\wh{c}$ as in (\ref{1.8}))},
\\[1ex]
\mbox{for $\wh{B} \subseteq \ov{B}$ with $\wh{B} \in \cI_{\ell + 1}, \,\ov{B} \in \cI_\ell, Y_{\wh{B}} \le M^{d-2} \,Y_{\ov{B}}$},
\\[0.5ex]
Y_{\ov{B}} \stackrel{(\ref{2.12})}{\ge} \mbox{\f $\dis\frac{M^2}{3^d + 1}$} \; \mbox{\f $\dis\frac{1}{M^d}$} \; \dsl_{\wh{B} \subseteq \ov{B}} Y_{\wh{B}} \wedge \wt{\eta}_0 .
\end{array}\right.
\end{equation}

\n
We can then apply Lemma 3.4, p.~173 of \cite{Szni98a} and obtain as on p.~175 of the same reference the capacity estimate (where in the notation of \cite{Szni98a}, $c_6 = \wh{c}$, $c_8 = \frac{M^2}{3^d + 1}$, $c_7 = M^{d-2}$, $\delta_1 = \wt{\eta}_0$, $c_9 = (\frac{\delta_1}{c_6 \, c_7}) \wedge 1$, and $\frac{1}{2} \,\delta_1 \, c_7^{-1} = \frac{1}{2} \, \wt{\eta}_0 \, M^{2-d} = (4c_* \,M^{2d-2})^{-1} = \delta_0$ in (\ref{2.7})):
\begin{equation}\label{2.21}
\mbox{\f $\dis\frac{1}{M^{dq}}$} \; \dsl_{B \subseteq \cI_q,{\rm rarefied},B \subseteq \Delta} \; Y_B \le c'\, \Big( \mbox{\f $\dis\frac{M^2}{3^d+1}$} \Big)^{-q/2} .
\end{equation}

\n
(and $c'$ corresponds to $\frac{c_6\, c_8}{c_9}$ in the notation of (3.38) on p.~175 of \cite{Szni98a}). 

\medskip
Further, one knows by (3.36) of \cite{Szni21b} that for $A \subseteq B$ with $B$ a box, one has $|A| / |B| \le c \,{\rm cap}(A) / |B|^{\frac{d-2}{d}}$, so that from (\ref{2.21}) we obtain
\begin{equation}\label{2.22}
\mbox{\f $\dis\frac{1}{M^{dq}}$} \; \dsl_{B \subseteq \cI_q,{\rm rarefied},B \subseteq \Delta} \; \mbox{\f $\dis\frac{|B \cap F|}{|B|}$} \le c\, \Big( \mbox{\f $\dis\frac{M^2}{3^d+1}$} \Big)^{-q/2} .
\end{equation}

\n
Since for all $B \in \cI_q$, $M^{dq} |B| = |\Delta |$, the claim (\ref{2.10}) follows. This completes the proof of Proposition \ref{prop2.1}.
\end{proof}

\begin{remark}\label{rem2.3} \rm Incidentally let us mention that $\wh{c} > \wt{\eta}_0 > \delta$ (in the notation of (\ref{1.8}), (\ref{2.11}), (\ref{2.7})). The second inequality is immediate. As for the first inequality, setting $a_n = {\rm cap} ([0,M^n \,L_1)^d) / (M^n \,L_1)^{d-2}$ for $n \ge 0$, and assuming $N$ large enough so that (\ref{2.9}) holds, Lemma \ref{lem2.2} with $F = \IZ^d$ shows that $a_{n+1} \ge \frac{M^2}{3^d + 1} \;(a_n \wedge \wt{\eta}_0)$ for all $n \ge 0$. Since the sequence $a_n$ is bounded by $\wh{c}$ and $M^2/(3^d + 1) > 1$, this implies that for some $n_0 \ge 0$ the inequality $a_{n_0} \le \wt{\eta}_0$ occurs so that $\wh{c} \ge a_{n_0 + 1} \ge \frac{M^2}{3^d + 1} \;\wt{\eta}_0 > \wt{\eta}_0$, and this yields the first inequality. \hfill $\square$
\end{remark}

We close this section with some notation. We set
\begin{equation}\label{2.23}
R_\ell = M^{\ell_N - \ell} \,L_1, \; \mbox{for $0 \le \ell \le \ell_N$},
\end{equation}
so that $R_\ell$ is the size of the boxes in $\cI_\ell$, for $0 \le \ell \le \ell_N$. Also we denote by
\begin{equation}\label{2.24}
\mbox{$\ov{D}_N$ the union of boxes in $\cI_0$ that intersect $D_N$},
\end{equation}
so that
\begin{equation}\label{2.25}
D_N = [-N,N]^d \subseteq \ov{D}_N \subseteq [-2N, 2N]^d \;\;\mbox{and $|D_N| \le |\ov{D}_N| \le 2^d\, |D_N|$}.
\end{equation}

\section{Resonance sets}
\setcounter{equation}{0}

In this section we introduce certain {\it resonance sets}, where on many well-separated scales a finite set $U_0$ and its complement $U_1$ (which up to translation play the role of $\cU_0$ and $\cU_1$ in (\ref{1.23}), (\ref{1.22})) occupy a non-degenerate fraction of volume. The main Proposition \ref{prop3.1} comes as an adaptation to our set-up of the results of \cite{NitzSzni20} in the case of Brownian motion, and of \cite{ChiaNitz} in the case of random walks among random conductances. It shows that when the simple random walk starts at a point where on many well-separated scales $U_0$ occupies at least half of the relative volume, see (\ref{3.11}), then it enters the resonance set with high probability. Proposition \ref{prop3.1} plays an important role in the proof of Theorem \ref{theo4.1}, specifically for the treatment of boxes of Type $\IA$ (informally, they correspond to certain nearly macroscopic boxes ``well within $\cU_0$'' that intersect the bubble set, see (\ref{4.17})). Proposition \ref{prop3.1} provides us with a quantitative bound on the probability that the random walk starting on such a box avoids a certain set $\wh{Res}_J$, see (\ref{4.62}), (\ref{4.56}). The main steps of the proof of Proposition \ref{prop3.1} are sketched in the Appendix.

\medskip
We now introduce the set-up. We first consider an integer, which will parametrize the strength of the resonance
\begin{equation}\label{3.1}
J \ge 1,
\end{equation}
as well as the integer
\begin{equation}\label{3.2}
L(J) = \min\{L \ge 5; d2^{d-1} \, 2^{-L} \le (200 J)^{-1}\},
\end{equation}
which will control the separation between dyadic scales. We keep here similar notation as in \cite{NitzSzni20} and \cite{ChiaNitz}. In this section and in the Appendix, $L$ does not refer to a spatial scale but to a separation between dyadic scales (hopefully no confusion should occur). Further, we have
\begin{equation}\label{3.3}
\mbox{$U_0$ a finite non-empty subset of $\IZ^d$, and $U_1 = \IZ^d \backslash U_0$},
\end{equation}
a separation between (dyadic) scales
\begin{equation}\label{3.4}
L \ge L(J),
\end{equation}
and a finite set $\cA$ such that for some non-negative integer $h$
\begin{equation}\label{3.5}
\cA \subseteq (J+1) \,L\, \IN_* + h
\end{equation}

\n
(we will assume that $U_0$ occupies at least half of the volume of the balls $B(0,2^m L_1)$ for each $m \in \cA)$. We also consider the enlarged set $\cA^*$ where intermediate scales have been added:
\begin{equation}\label{3.6}
\mbox{$\cA_* =     \{m \in \IN$; for some $0 \le j \le J$, $m + j \,L \in \cA\} \subseteq L\,\IN_* + h$},
\end{equation}
and the {\it $(\cA_*,J)$-resonance set}
\begin{equation}\label{3.7}
Res_{\cA_*,J} = \Big\{x \in \IZ^d; \; \dsl_{m \in \cA_*} 1\Big\{\mbox{\f $\dis\frac{|B(x,2^{m+2} L_1) \cap U_1|}{|B(x,2^{m+2} L_1)|}$} \in [\wt{\alpha}, 1 - \wt{\alpha}]\Big\} \ge J \Big\}, \; \mbox{with}\;\; \wt{\alpha} = \mbox{\f $\dis\frac{1}{3}$} \;4^{-d}.
\end{equation}

\n
The next Proposition \ref{prop3.1} states a uniform control showing that it is hard for the simple random walk to avoid the resonance set in (\ref{3.7}), when $\cA$ is sufficiently large, and $U_0$ occupies more than half the volume of the sup-norm balls centered at its starting point with radius $2^m L_1$, $m \in \cA$. More precisely, we have
\begin{proposition}\label{prop3.1}
There exists a doubly indexed non-negative sequence $\gamma_{I,J}$, $I,J \ge 1$, such that 
\begin{equation}\label{3.8}
\mbox{for each $J \ge 1$, $\lim\limits_{I \r \infty} \; \gamma_{I,J} = 0$  (see also (\ref{3.14}))},
\end{equation}
and a sequence of positive integers
\begin{equation}\label{3.9}
r_{\min} (J) \ge 6, \; J \ge 1.
\end{equation}
They have the property that for any $I, J \ge 1$, when
\begin{equation}\label{3.10}
L_1 \ge r_{\min} (J),
\end{equation}
for any $U_0, U_1$ as in (\ref{3.3}), $L \ge L(J)$ as in (\ref{3.4}), and $\cA$ satisfying (\ref{3.5}) for which
\begin{align}
&|B(0,2^m L_1) \cap U_1 | \le \mbox{\f $\dis\frac{1}{2}$} \;|B(0,2^m L_1)|, \; \mbox{for all $m \in \cA$ and} \label{3.11}
\\
&|\cA | = I, \label{3.12}
\end{align}
one has the bound (see the beginning of Section 1 for notation)
\begin{equation}\label{3.13}
P_0 [H_{Res_{\cA_*,J}} = \infty] \le \gamma_{I,J} .
\end{equation}
\end{proposition}

The proof of Proposition \ref{prop3.1} is similar to the proofs in \cite{NitzSzni20} and \cite{ChiaNitz}. The main steps are sketched in the Appendix. One actually has a more quantitative statement than (\ref{3.8}), see (\ref{A.43}):
\begin{equation}\label{3.14}
\mbox{for each $J \ge 1$, $\lim\limits_{I \r \infty} I^{-1/2^{J-1}} \log \gamma_{I,J} \le \log \big(1 - \check{c}_2 (J)\big) < 0$},
\end{equation}
with $\check{c}_2(J) \in (0,1)$ from (\ref{A.19}).

\medskip
In the proof of Theorem \ref{theo4.1} in the next section, we will apply Proposition \ref{3.1} to the random sets $\cU_0, \cU_1$ shifted at a point where on many well-separated scales $2^m \,L_1$, $\cU_0$ occupies at least half of the relative volume of the sup-norm ball of radius $2^m L_1$ centered at this point.

\section{Coarse graining of the the bubble set}
\setcounter{equation}{0}

This section contains Theorem \ref{theo4.1}, which is the central element of this article. It constructs a coarse grained object, namely a random set $C_\omega$ of low complexity whose equilibrium potential is close to $1$ on most of the bubble set $B u b$. This random set is made of $(\alpha, \beta, \gamma)$-good $B_0$-boxes for which the corresponding $D_0$-boxes (see (\ref{1.18})) have a ``local level'' $N_u(D_0) / {\rm cap} (D_0)$ at least $\beta$. Its purpose is to quantify the cost induced by the bubble set (about its specific use we refer to the proof of Proposition 4.1 in \cite{Szni21b}). The challenge in the construction of $C_\omega$ lies in the fact that the bubble set may be very irregular with little depth apart from its constitutive grains of size $L_1$. Theorem \ref{theo4.1} constitutes an important improvement on Theorem 3.1 of \cite{Szni21b}: loosely speaking, it shows that the $c_0$ of \cite{Szni21b} can be chosen arbitrarily close to $1$.

\medskip
We now specify the set-up. We assume that (see (\ref{1.16}))
\begin{align}
& 0 < u < \ov{u}, \label{4.1} 
\\[1ex]
& \mbox{$\alpha  > \beta > \gamma$ belong to $(u,\ov{u})$}, \label{4.2}
\\[1ex]
& 0 < \varepsilon < 10^{-3}. \label{4.3}
\end{align}

\n
Further, with $c_1$ as in Lemma \ref{lem1.2} and $c_2(\alpha,\beta,\gamma)$ as in Lemma \ref{lem1.3}, we assume that
\begin{equation}\label{4.4}
K \ge c_1 \vee c_2 (\alpha,\beta,\gamma).
\end{equation}

\n
We recall that the asymptotically negligible bad event $\cB_N$ is defined in (\ref{1.24}) and that $\ov{K} = 2 K + 3$. Here is the main result.

\begin{theorem}\label{theo4.1}
For any $a \in (0,1)$ and $u, \alpha, \beta, \gamma, \varepsilon, K$ as in (\ref{4.1}) - (\ref{4.4}), for large $N$ on $\cB^c_N$, one can construct a random subset $C_\omega \subseteq [-4N, 4N]^d$, which is a union of $B_0$-boxes and satisfies the following properties:
\begin{equation}\label{4.5}
\left\{ \begin{array}{rl}
{\rm i)} & \mbox{for all $B_0 \subseteq C_\omega$, $B_0$ is $(\alpha,\beta,\gamma)$-good and such that $N_u (D_0) \ge \beta \,{\rm cap} (D_0)$},
\\[1ex]
{\rm ii)} & \mbox{the base points of the $B_0 \subseteq C_\omega$ have mutual distance at least $\ov{K} \,L_0$},
\\[1ex]
{\rm iii)} & \mbox{the set $\cS_N$ of possible values of $C_\omega$ is such that  $| \cS_N| = \exp \{o(N^{d-2})\}$},
\\[1ex]
{\rm iv)} & \mbox{the $2 \ov{K}L_1$-neighborhood of $C_\omega$ has volume at most $\varepsilon \, |D_N|$},
\\[1ex]
{\rm v)} & \mbox{if $h_{C_\omega}$ stands for the equilibrium potential of $C_\omega$, see (\ref{1.5}), one has}
\\
& \mbox{$|\{x \in B u b$;  $h_{C_\omega} (x) < a\}| \le \varepsilon \,|D_N|$, with $B u b$ as in (\ref{1.26})}.
\end{array}\right.
\end{equation}
\end{theorem}

We refer to the discussion below Theorem 3.1 of \cite{Szni21b} for an informal description of the use of the conditions in (\ref{4.5}). Here, the crucial novelty is that $a$ in v) can be chosen arbitrarily close to $1$. In Section 5 this leads to the improved asymptotic upper bound on the probability $\IP[\cA_N]$ of an excess of disconnected points in $D_N$, see Theorem \ref{theo5.1}. If $\ov{u}$ and $u_*$ coincide (as expected), this asymptotic upper bound matches in principal order the asymptotic lower bound in (\ref{0.10}), and thus yields the exponential rate of decay of $\IP[\cA_N]$, see (\ref{0.14}).

\medskip
At this point it is perhaps helpful to provide an informal description of the proof of Theorem \ref{theo4.1}. Loosely speaking, one considers nearly macroscopic boxes $B$ at depth $p$ (i.e.~in $\cI_p$, see (\ref{2.3})) that meet the bubble set $B u b$, where $p$ as chosen in (\ref{4.11}), eventually only depends on $d,a$ and $\ve$, see (\ref{4.99}). One classifies these boxes into Types $\Ib, \IB$, and $\IA$, see (\ref{4.13}), (\ref{4.16}), and (\ref{4.17}). When $N$ is large, all $B_1$-boxes contained in $Bub$ are almost contained in $\cU_0$, see (\ref{1.26}), (\ref{1.23}). The much bigger (nearly macroscopic) boxes $B$ of Types $\Ib$ or $\IB$ loosely speaking correspond to ``boundary boxes of $\cU_0$'', where in the case of Type $\Ib$, $\cU_1$ occupies more than half the volume of $B$, see (\ref{4.13}), and in the case of Type $\IB$ something similar takes place at an intermediate scale not much bigger than that of $B$, see (\ref{4.16}). The boxes $B$ of Type $\IA$ correspond to ``inner boxes of $\cU_0$'', where on quite a few well-separated intermediate scales above that of $B$, the random set $\cU_0$ occupies more than half of the volume.

\medskip
The boxes of Types $\Ib$ and $\IB$ are then classified as either {\it rarefied} or {\it substantial}, see (\ref{4.35}), (\ref{4.36}) for Type $\Ib$, and (\ref{4.47}), (\ref{4.48}) for Type $\IB$. The volume estimates from Section 2 (stemming from the {\it method of enlargement of obstacles}) allow to neglect the rarefied boxes, see (\ref{4.39}), (\ref{4.40}) for Type $\Ib$, and (\ref{4.51}), (\ref{4.52}) for Type $\IB$. On the other hand, for substantial boxes $B$ of Types $\Ib$ and $\IB$ a type of Wiener criterion holds, see (\ref{4.36}), (\ref{4.48}). Using isoperimetric controls going back to \cite{DeusPisz96} one can infer the substantial ``surface-like presence'' of $(\alpha,\beta,\gamma)$-good $B_0$-boxes of $\partial_{B_0} \,\cU_1$ in many scales above $B$. As observed below (\ref{1.23}), the inequality $N_u(D_0) \ge \beta \, {\rm cap}(D_0)$ necessarily holds for these $B_0$-boxes. In the case of boxes $B$ of Type $\IA$, the methods of Section 3 (concerning {\it resonance sets}) apply instead. They show that when starting in $B$ the simple random walk enters with a ``high probability'' (specifically at least $1 - (1-a)/10$) a certain resonance set $\wh{Res}_J$, see (\ref{4.56}), (\ref{4.62}). On this resonance set again a similar procedure can be performed to extract on many scales a ``surface-like presence'' of $(\alpha,\beta,\gamma)$-good $B_0$-boxes with $N_u(D_0) \ge \beta \, {\rm cap}(D_0)$. Then, in Proposition \ref{prop4.2}, we extract a collection of $(\alpha,\beta,\gamma)$-good $B_0$-boxes with $N_u(D_0) \ge \beta \, {\rm cap}(D_0)$, having base points at mutual distance at least $\ov{K} L_0$, such that for each box $B$ of Types $\Ib$ and $\IB$, or contained in $\wh{Res}_J$, on a large number of scales above $B$, the capacity of the $B_0$-boxes contained within the box of that scale remains comparable to the capacity of that box. The Lemmas \ref{lem1.1} and \ref{lem1.2} enter the proof of Proposition \ref{prop4.2}. Once Proposition \ref{prop4.2} is established, Theorem \ref{theo4.1} quickly follows.

\bigskip
\n
{\it Proof of Theorem \ref{theo4.1}:} We recall the assumptions (\ref{4.1}) - (\ref{4.4}) on the parameters. We assume that we are on the event
\begin{equation}\label{4.6}
\begin{split}
\Omega_{\ve,N} =  &\;\mbox{$\{|B u b| \ge \ve\, |D_N|\} \backslash \,\cB_N$ (on the complementary event $\{|Bub| < \ve\, |D_N|\} \backslash \,\cB_N$}
\\
&\;\;\;   \mbox{in $\cB^c_N$ we will simply set $C_\omega = \phi)$}.
\end{split}
\end{equation}

\n
We recall the dyadic integer $M$ (solely depending on $d$) from (\ref{2.1}), as well as the naturally nested collections of boxes $\cI_\ell$ (at depth $\ell$), for $0 \le \ell \le \ell_N$, see (\ref{2.3}). The size of a box in $\cI_\ell$ is $R_\ell = M^{\ell_N - \ell} \, L_1$, see (\ref{2.23}). We also recall that when $B \in \cI_\ell$, then for $0 \le \ell' \le \ell$ the notation $B^{(\ell ')}$ refers to the unique box of $\cI_{\ell '}$ that contains $B$, see (\ref{2.4}).

\medskip
We now pick some integer
\begin{equation}\label{4.7}
J \ge 1.
\end{equation}

\medskip\n
We will later choose $J$ as a function of $d \ge 3$ and $a \in (0,1)$, see (\ref{4.99}). Having in mind the estimates on the resonance set from Proposition \ref{prop3.1}, we pick in the notation of (\ref{3.8})
\begin{equation}\label{4.8}
\mbox{$I(a,J) \ge 1$ such that $\gamma_{I,J} \le (1 - a) / 10$}
\end{equation}

\n
(and view for the time being $I$ as a function of $d,a$, and $J$). Next, with the volume estimates for rarefied boxes from Proposition \ref{prop2.1} in mind, we choose an integer, to be later interpreted as a depth:
\begin{equation}\label{4.9}
\mbox{$\wt{p} (a, \ve, J) \ge 4$ such that} \left\{ \begin{array}{rl}
{\rm i)} & c_3\, \Big(\mbox{\f $\dis\frac{M^2}{3^d+ 1}$}\Big)^{- \wt{p} /2} \le \ve / (10 \, 2^{d+1}), \mbox{with $c_3$ from (\ref{2.10})},
\\[2ex]
{\rm ii)} & \wt{p} \ge 2 \, \wh{c} \, J/\delta, \; \mbox{with $\wh{c}, \delta$ as in (\ref{1.8}), (\ref{2.7})}.
\end{array}\right.
\end{equation}
Since $\wt{p} \ge 4$ and $M^2 > 3^d + 1$ (see (\ref{2.1})), one readily checks that
\begin{equation}\label{4.10}
M^{\wt{p}} \ge M^4 > 9^d > 3d \,4^d \; \mbox{(using induction on $d \ge 3$ for the last inequality)}.
\end{equation}
With $L(J)$ as in (\ref{3.2}) we also define the depth
\begin{equation}\label{4.11}
p(a,\ve,J) = 5 + I(J+1) \,L(J) + \wt{p}.
\end{equation}

\n
From now on we implicitly assume $N$ large enough so that (in the notation of (\ref{3.9}), (\ref{2.9}), (\ref{2.2}), and (\ref{1.25})):
\begin{equation}\label{4.12}
\mbox{$L_1 > r_{\min} (J), \,L_1^{d-2} g(0,0) > 2 c_*\,M^d, \ell_N > p$, and $|{\rm Deep} \,B_1| > \fr \;|B_1|$ for any $L_1$-box}.
\end{equation}

\n
We will now classify the boxes $B \in \cI_p$ intersecting $B u b \;(\subseteq D_N)$, see (\ref{1.26}), into three types (namely $\Ib$, $\IB$, and $\IA$). Informally, the boxes of Types $\Ib$ and $\IB$ correspond to ``boundary boxes in $\cU_0$'', whereas boxes of Type $\IA$ correspond to ``inner boxes in $\cU_0$''. More precisely, we say that
\begin{equation}\label{4.13}
\mbox{$B$ is of Type $\Ib$ if $B \in \cI_p$, $B \cap B u b \not= \phi$, and $|B \cap \cU_0| < \fr \; |B|$}.
\end{equation}
We now describe Types $\IB$ and $\IA$. To this end, for $B \in \cI_p$ and $0 \le \ell \le p$ we define
\begin{equation}\label{4.14}
\Delta_{B,\ell} = x_B + [-2R_\ell, 2R_\ell)^d,
\end{equation}

\n
where we recall that $x_B$ stands for the base point of $B$ (see the beginning of Section 1) and $R_\ell$ is defined in (\ref{2.23}). We note that with $B$ and $\ell$ as above
\begin{equation}\label{4.15}
\left\{ \begin{array}{rl}
{\rm i)} & x_B \in B^{(\ell)} = x_{B^{(\ell)}} + [0,R_\ell)^d, 
\\[1ex]
{\rm ii)} & B = x_B + [0,R_p)^d \subseteq B^{(\ell)} \subseteq x_B + [-2R_\ell, 2 R_\ell)^d = \Delta_{B,\ell} \subseteq B(x_B, 2R_\ell).
\end{array}\right.
\end{equation}
We then say that
\begin{equation}\label{4.16}
\left\{ \begin{array}{l}
\mbox{$B$ is of Type $\IB$ if $B \in \cI_p$, $B \cap B u b \not= \phi$, $|B \cap \cU_0| \ge \fr \; |B|$}
\\
\mbox{(so $B$ is not of Type $\Ib$) and for some}
\\
\mbox{$\ell \in \cS \stackrel{\rm def}{=} \{p-4 - (J + 1) \,L(J), p-4-2 (J+1) \,L(J), \dots, \wt{p} + 1\}$, one has}
\\
\mbox{$|\Delta_{B,\ell} \cap \cU_0 | < \mbox{\f $\dis\frac{3}{4}$} \; | \Delta_{B,\ell}|$}
\end{array}\right.
\end{equation}
(we recall that $\wt{p} + 1 \stackrel{(\ref{4.11})}{=} p-4 - I(J+1) \,L(J))$.

\medskip
Finally, we say that
\begin{equation}\label{4.17}
\left\{ \begin{array}{l}
\mbox{$B$ is of Type $\IA$, if $B \in \cI_p$, $B \cap B u b \not= \phi$, $|B \cap \cU_0| \ge \fr \; |B|$ and}
\\
\mbox{for all $\ell \in \cS$, $|\Delta_{B,\ell} \cap \cU_0| \ge \mbox{\f $\dis\frac{3}{4}$} \,|\Delta_{B,\ell}|$}.
\end{array}\right.
\end{equation}

\n
Note that the three types are mutually exclusive. Also for $1 \le \ell \le p$, one has $4R_\ell \le 4 N/ M < N$, so that by (\ref{4.15}) ii) (and $Bub \subseteq D_N = [-N,N]^d)$
\begin{equation}\label{4.18}
\mbox{for any $B \in \cI_p$ such that $B \cap Bub \not= \phi$ and $1 \le \ell \le p$, $\Delta_{B,\ell} \subseteq [-2N,2N]^d$}.
\end{equation}
\begin{samepage}
Moreover, one has the inclusion
\begin{equation}\label{4.19}
Bub \subseteq \big(\bigcup\limits_{{\rm Type} \, \Ib} B\big) \cup \big(\bigcup\limits_{{\rm Type} \, \IB} B\big) \cup \big(\bigcup\limits_{{\rm Type} \, \IA} B\big)
\end{equation}
and the above three sets in parentheses are pairwise disjoint.
\end{samepage}

\medskip
Each type will require a different treatment. We start with the case of boxes of Type~$\Ib$. Our next goal is to establish (\ref{4.39}) (\ref{4.40}).

\medskip
We note that when $B_1$ is an $L_1$-box and $B \in \cI_p$ is of type $\Ib$ with $B_1 \subseteq B \cap B ub$, then by (\ref{4.12}) we have $|B_1 \cap \, \cU_0| \ge \frac{1}{2} \;|B_1|$ but at the same time $|B \cap \cU_0| < \frac{1}{2} \,|B|$ since $B$ is of Type $\Ib$. We then define
\begin{equation}\label{4.20}
\begin{array}{l}
\mbox{$\wt{B}(B_1)$ the maximal box in the tower above $B_1$ that is contained in $B$ and}
\\
\mbox{such that $|\wt{B}(B_1) \cap \cU_0| \ge \fr\; |\wt{B}(B_1)|$},
\end{array}
\end{equation}
as well as
\begin{equation}\label{4.21}
\mbox{$B'(B_1)\;(\subseteq B)$ the box immediately above $\wt{B}(B_1)$ in that tower}.
\end{equation}
Thus, for $B_1 \subseteq B u b \cap B$ with $B$ of Type $\Ib$, one has
\begin{equation}\label{4.22}
\begin{array}{l}
\mbox{$B_1 \subseteq B'(B_1) \subseteq B \subseteq \ov{D}_N \stackrel{(\ref{2.25})}{\subseteq} [-2N, 2N]^d$, and}
\\
\mbox{$\mbox{\f $\dis\frac{1}{2M^d}$} \; |B'(B_1)| = \fr \;|\wt{B}(B_1)| \le | \wt{B}(B_1) \cap \cU_0| \le |B'(B_1) \cap \cU_0| < \fr \;|B' (B_1)|$}.
\end{array}
\end{equation}
Moreover, we note that
\begin{equation}\label{4.23}
\begin{array}{l}
\mbox{when $B$ is of Type $\Ib$, then as $B_1 \subseteq B u b \cap B$ varies, the corresponding $B'(B_1)$}
\\
\mbox{are pairwise disjoint or comparable for the inclusion relation}.
\end{array}
\end{equation}

\n
As an aside the boxes $B'(B_1)$ can be ``very small'' with a size close to $L_1$. In particular, they can be much smaller than $B \;(\in \cI_p)$. We then introduce an arbitrary enumeration
\begin{equation}\label{4.24}
\begin{array}{l}
\mbox{$B'_j$, $1 \le j \le J^{\Ib}$, of the maximal elements for the inclusion relation of the}
\\
\mbox{$B'(B_1)$, with $B_1 \subseteq Bub \cap (\bigcup_{{\rm Type}\,\Ib} B)$},
\end{array}
\end{equation}
so that the boxes
\begin{equation}\label{4.25}
\begin{array}{l}
\mbox{$B'_j, 1 \le j \le J^{\Ib}$, are pairwise disjoint, contained in $\bigcup_{{\rm Type}\,\Ib} B \subseteq \ov{D}_N \subseteq [-2N,2N]^d$},
\\
\mbox{and cover $Bub \cap (\bigcup_{{\rm Type}\,\Ib} B)$}.
\end{array}
\end{equation}
By (\ref{4.22}) we know that $\frac{1}{2M^d}\, |B'_j| \le |B'_j \cap\, \cU_0| < \frac{1}{2} \;|B'_j |$ for each $1 \le j \le J^{\Ib}$, and when $N$ is large, we can find a nearly concentric box $\wh{B} '_j$ in $B'_j$, union of $B_0$-boxes, such that $|B'_j \backslash \wh{B} '_j| \le \frac{1}{4M^d} |B'_j|$, and the distance of $\wh{B} '_j$ to the inner boundary of $B'_j$ is at least $c(M) \,L_1$. As a result we have $|\wh{B} '_j| \ge (1- (4M^d)^{-1}) \,|B'_j| \ge \frac{2}{3} \,|B'_j|$, so that
\begin{equation}\label{4.26}
\mbox{\f $\dis\frac{1}{4M^d}$} \;|\wh{B} '_j| \le |\wh{B} '_j \cap \cU_0| < \fr \;|B'_j| \le \mbox{\f $\dis\frac{3}{4}$} \;|\wh{B}'_j |.
\end{equation}

\n
\begin{samepage}
So, by the isoperimetric controls (A.3) - (A.6), p.~480-481 of \cite{DeusPisz96}, we see that when $N$ is large on the event $\Omega_{\ve,N}$ in (\ref{4.6})
\begin{equation}\label{4.27}
\left\{\begin{array}{l}
\mbox{for each $1 \le j \le J^{\Ib}$ one can find a coordinate projection $\pi_j '\,^{\!\!,\,\Ib}$ and at least}
\\
\mbox{$c_4 (|B'_j|\, /\, |B_0|)^{\frac{d-1}{d}} B_0$-boxes in $\partial_{B_0}\, \cU_1$ (see below (\ref{1.23}) for notation) in distinct}
\\[0.5ex]
\mbox{$\pi_j '\,^{\!\!,\,\Ib}$-columns, with base points at distance at least $c_5\,L_1$ from $\partial B'_j$}.
\end{array}\right.
\end{equation}

\n
(We recall that $M$ in (\ref{2.1}) is viewed as a dimension dependent constant).
\end{samepage}

\medskip
We then introduce the set of good indices:
\begin{equation}\label{4.28}
\begin{array}{l}
\cG^{\Ib} = \{j \in \{1,\dots,J^{\Ib}\}; 
\\
\mbox{there are at most $\frac{1}{2} \;c_4(|B'_j|\, /\, |B_0|)^{\frac{d-1}{d}} (\alpha,\beta,\gamma)$-bad $B_0$-boxes in $B'_j\}$}.
\end{array}
\end{equation}

\n
As remarked below (\ref{1.23}), when $B_0$ is an $(\alpha,\beta,\gamma)$-good box in $\partial_{B_0} \,\cU_1$, then $N_u(D_0) \ge \beta \, {\rm cap}(D_0)$ holds as well. We thus find that
\begin{equation}\label{4.29}
\left\{ \begin{array}{l}
\mbox{for each $j \in \cG^{\Ib}$ one has a coordinate projection $\pi_j '\,^{\!\!,\,\Ib}$ and at least}
\\
\mbox{$\frac{1}{2} \; c_4(|B'_j| \,/ \, |B_0|)^{\frac{d-1}{d}} B_0$-boxes in distinct $\pi_j '\,^{\!\!,\,\Ib}$-columns, which are $(\alpha,\beta,\gamma)$-good,}
\\[0.5ex]
\mbox{with $N_u(D_0) \ge \beta\, {\rm cap}(D_0)$, and contained in $B'_j$, with base points at distance}
\\[0.5ex]
\mbox{at least $c_5 \,L_1$ from $\partial B'_j$}.
\end{array}\right.
\end{equation}
We then define for $1 \le j \le J^{\Ib}$,
\begin{equation}\label{4.30}
a'_j = |B'_j| \,/\, \big(\dis\sum\limits_{1 \le k \le J^{\Ib}} |B'_k|\big).
\end{equation}
Then, in the same spirit as below (3.26) of \cite{Szni21b}, one finds that
\begin{equation}\label{4.31}
\begin{array}{l}
\frac{1}{2} \;c_4 \,\big|Bub \cap \big(\bigcup_{{\rm Type}\,\Ib} B\big)\big|^{\frac{d-1}{d}} \dsl_{j \notin \cG^{\Ib}} a'_j \,^{\!\!\frac{d-1}{d}} \stackrel{(\ref{4.25}),(\ref{4.30})}{\le}
\\
\frac{1}{2}\; c_4 \; \dsl_{j \notin \cG^{\Ib}}  |B'_j|^{\frac{d-1}{d}} \stackrel{(\ref{4.28})}{\le} \dsl_{j \notin \cG^{\Ib}}  \; \dsl_{B_0 \subseteq B'_j} |B_0|^{\frac{d-1}{d}} \;1 
\{B_0\; \mbox{is}\;  (\alpha,\beta,\gamma)\mbox{-bad}\} \stackrel{(\ref{4.25})}{\le}
\\
L_0^{d-1} \;\dsl_{B_0 \subseteq \ov{D}_N} \;1\{B_0 \;\mbox{is} \; (\alpha,\beta,\gamma)\mbox{-bad}\} \stackrel{(\ref{4.6}), (\ref{1.24})}{\le} L_0^{d-1} \,\rho(L_0)\, N^{d-2},
\end{array}
\end{equation}

\n
where in the last step we have used that $\Omega_{\ve, N} \subseteq \cB^c_N$ (with $\cB_N$ defined in (\ref{1.24})). 

\medskip
We now introduce the event
\begin{equation}\label{4.32}
\wt{\Omega}_{\ve, N} = \big\{\big| Bub \cap \big(\bigcup\limits_{{\rm Type}\,\Ib} B\big)\big| \ge \mbox{\f $\dis\frac{\ve}{20}$} \; |D_N|\big\}.
\end{equation}

\n
In essence, on the complement of $\wt{\Omega}_{\ve, N}$ in $\Omega_{\ve,N}$, see (\ref{4.6}), we will simply ``discard'' the whole set $B u b \cap (\bigcup_{{\rm Type}\,\Ib} B)$ to achieve (\ref{4.39}), (\ref{4.40}). The main work pertains to the event $\Omega_{\ve,N} \cap \wt{\Omega}_{\ve, N}$. On this event we find by (\ref{4.31}) that
\begin{equation}\label{4.33}
\Big(\dsl_{j \notin \cG^{\Ib}} a'_j\Big)^{\frac{d-1}{d}} \le \dsl_{j \notin \cG^{\Ib}} a' _j \,^{\!\!\frac{d-1}{d}} \stackrel{(\ref{4.31}),(\ref{4.32})}{\le} \mbox{\f $\dis\frac{2}{c_4}$} \;L_0^{d-1} \,\rho(L_0) \,N^{d-2} \, / \, \Big(\mbox{\f $\dis\frac{\ve}{20}$} \; |D_N|\Big)^{\frac{d-1}{d}}.
\end{equation}
Taking the $\frac{d}{d-1}$-th power of the above inequality, we find in view of (\ref{4.30}) and the definition of $L_0$ in (\ref{1.9}) that for large $N$ on $\Omega_{\ve,N} \cap \wt{\Omega}_{\ve,N}$
\begin{equation}\label{4.34}
\dsl_{j \notin \cG^{\Ib}} |B'_j| \le \mbox{\f $\dis\frac{c'}{\ve}$} \;\rho (L_0)^{\frac{d}{d-1}} \;\dsl_{1 \le j \le J^{\Ib}} |B'_j| \stackrel{(\ref{4.25})}{\le} \mbox{\f $\dis\frac{c}{\ve}$} \;\rho(L_0)^{\frac{d}{d-1}} \,|D_N| \le \mbox{\f $\dis\frac{\ve}{20}$} \;|D_N|.
\end{equation}

\n
We now classify the boxes of Type $\Ib$ as either rarefied or substantial. We recall the definition of $\delta$ in (\ref{2.7}), as well as that of $\wt{p}$ in (\ref{4.9}). We work on the event $\Omega_{\ve,N}$ from (\ref{4.6}). We say that a box $B$ of Type $\Ib$ is {\it rarefied} if
\begin{equation}\label{4.35}
\dsl_{B^{(\wt{p})} \subseteq \wt{B} \subseteq B^{(1)}} \; \dis\frac{{\rm cap}\big(\wt{B} \cap (\bigcup_{j \in \cG^{\Ib}} B'_j)\big)}{|\wt{B}|^{\frac{d-2}{d}}} < \delta \, \wt{p} ,
\end{equation}

\n
where the above sum runs over the boxes $\wt{B}$ in the tower above $B^{(\wt{p})}$ (see (\ref{2.4}) for notation) that are contained in $B^{(1)}$. Further, we say that $B$ of Type $\Ib$ is {\it substantial} if
\begin{equation}\label{4.36}
\dsl_{B^{(\wt{p})} \subseteq \wt{B} \subseteq B^{(1)}} \; \dis\frac{{\rm cap}\big(\wt{B} \cap (\bigcup_{j \in \cG^{\Ib}} B'_j)\big)}{|\wt{B}|^{\frac{d-2}{d}}} \ge \delta \, \wt{p} .
\end{equation}
Thus, $B$ of Type $\Ib$ is rarefied when $B^{(\wt{p})}$ is rarefied in the sense of (\ref{2.8}) with the choices $q = \wt{p}$ and $F = \bigcup_{j \in \cG^{\Ib}} B'_j$. By (\ref{4.12}) the condition (\ref{2.9}) is satisfied and the controls from Proposition \ref{prop2.1} now yield that
\begin{equation}\label{4.37}
\dsl_{B: {\rm Type}\, \Ib, {\rm rarefied}} \big| B \cap \big(\dis\bigcup\limits_{j \in \cG^{\Ib}} B'_j\big)\big| \le c_3 \, |\ov{D}_N| \; \Big(\mbox{\f $\dis\frac{M^2}{3^d + 1}$}\Big)^{-\wt{p}/2} \stackrel{(\ref{4.9}),(\ref{2.25})}{\le} \mbox{\f $\dis\frac{\ve}{20}$} \;|D_N| .
\end{equation}
We then set
\begin{equation}\label{4.38}
\Gamma^{\Ib} = \Big\{\Big(\bigcup_{B: {\rm Type}\, \Ib, {\rm rarefied}} B \cap \big(\bigcup\limits_{j \in \cG^{\Ib}} B'_j\big)\Big) \cup \big(\bigcup\limits_{j \notin \cG^{\Ib}} B'_j\big)\Big\} \cap Bub \; (\subseteq D_N).
\end{equation}

\n
We thus find by (\ref{4.25}), (\ref{4.38}) that for large $N$ on the event $\Omega_{\ve,N}$ in (\ref{4.6})
\begin{equation}\label{4.39}
Bub \cap \big(\bigcup\limits_{{\rm Type}\, \Ib} B\big) \subseteq \big(\bigcup_{B: {\rm Type}\, \Ib, {\rm substantial}} B\big)  \cup \Gamma^{\Ib} \subseteq \ov{D}_N
\end{equation}
and observe that on $\Omega_{\ve,N} \backslash \wt{\Omega}_{\ve,N}$
\begin{equation*}
 | \Gamma^{\Ib}| \stackrel{(\ref{4.25})}{\le} \big| Bub \cap \big(\bigcup_{{\rm Type}\,\Ib} B\big)\big| \stackrel{(\ref{4.32})}{\le} \mbox{\f $\dis\frac{\ve}{20}$}\;|D_N|, 
\end{equation*}
whereas on $\Omega_{\ve,N} \cap \wt{\Omega}_{\ve,N}$ one has
\begin{equation*}
 | \Gamma^{\Ib}| \stackrel{(\ref{4.37}),(\ref{4.34})}{\le}  \mbox{\f $\dis\frac{\ve}{10}$}\;|D_N|, 
 \end{equation*}
so that (for large $N$ on $\Omega_{\ve,N}$)
\begin{equation}\label{4.40}
| \Gamma^{\Ib}| \le \mbox{\f $\dis\frac{\ve}{10}$}\;|D_N|. 
\end{equation}

\n
We now turn to the treatment of boxes of Type $\IB$, see (\ref{4.16}). Our next goal is to establish in (\ref{4.51}), (\ref{4.52}) an analogue of (\ref{4.39}), (\ref{4.40}).

\medskip
We thus consider a box $B$ of Type $\IB$. So $|B \cap \cU_0| \ge \frac{1}{2} \;|B|$ and there is a largest element $\ell_0$ in $\cS$ (see (\ref{4.16})) such that $|\Delta_{B,\ell_0} \cap \cU_0| < \frac{3}{4} \; |\Delta_{B, \ell_0}|$. Note that by (\ref{4.14}), (\ref{4.15}), $B \subseteq \Delta_{B,p}$ and $|\Delta_{B,p}| = 4^d \,|B|$, so that
\begin{equation}\label{4.41}
\mbox{\f $\dis\frac{1}{2 \cdot 4^d}$} \; |\Delta_{B,p}| = \fr \;|B| \stackrel{(\ref{4.16})}{\le} |\Delta_{B,p} \cap \cU_0 |, \; |\Delta_{B,\ell_0} \cap \cU_0| < \mbox{\f $\dis\frac{3}{4}$}\;| \Delta_{B,\ell_0}|, \;\mbox{and $\wt{p} < \ell_0 < p$}.
\end{equation}
We then consider the set $\{\ell \in [\ell_0,p]$; $|\Delta_{B,\ell} \cap \cU_0| \ge  \frac{3}{4} \; | \Delta_{B,\ell} |\}$.

\medskip
Either this set is empty, in which case one has
\begin{equation}\label{4.42}
\mbox{\f $\dis\frac{1}{2 \cdot 4^d}$} \; |\Delta_{B,p}| \le  |\Delta_{B,p} \cap \cU_0 |  < \mbox{\f $\dis\frac{3}{4}$}\;  |\Delta_{B,p}|,
\end{equation}

\n
or this set is not empty, and we consider $\ell'_0$ its smallest element which (by definition of $\ell_0$) is bigger than $\ell_0$. We set $\ell_1 = \ell'_0 -1 \ge \ell_0$, so that $\ell_1$ does not belong to the set above (\ref{4.42}) and hence
\begin{equation}\label{4.43}
\mbox{\f $\dis\frac{3}{4}$}\;  |\Delta_{B,\ell_1}| >  |\Delta_{B,\ell_1} \cap \cU_0| \stackrel{\Delta_{B,\ell_1} \supseteq \Delta_{B,\ell'_0}}{\ge}  |\Delta_{B,\ell'_0} \cap \cU_0| \ge \mbox{\f $\dis\frac{3}{4}$}\;  |\Delta_{B,\ell'_0}| \stackrel{(\ref{4.14})}{=} \mbox{\f $\dis\frac{3}{4 M^d}$}\;  |\Delta_{B,\ell_1}|
\end{equation}

\n
(in the second and third inequality we respectively used that $\Delta_{B,\ell_1} \supseteq \Delta_{B,\ell'_0}$ and  that $\ell'_0$ belongs to the set above (\ref{4.42})).

\medskip
In either case (whether the set above (\ref{4.42}) is empty or not) we have a
\begin{equation}\label{4.44}
\begin{array}{l}
\mbox{$\Delta '(B)$ of the form $\Delta_{B,\ell}$ where $\ell$ is  the smallest element of $[\wt{p} + 1, p]$ such that}
\\
\mbox{\f $\dis\frac{3}{4}$} \; |\Delta'(B)| \ge |\Delta '(B) \cap \cU_0| \ge \mbox{\f $\dis\frac{3}{4M^d}$} \;|\Delta ' (B)|
\end{array}
\end{equation}
(we recall that $M \ge 8$, see (\ref{2.1}), so $2 \cdot 4^d \le \frac{4}{3} \;M^d$).

\medskip
Next, we note that $\Delta '(B)$ has size at least $4 R_p$ (and at most $4 R_{\wt{p} + 1})$, and the number of columns of $B_0$-boxes in any given coordinate direction that are contained in $\Delta '(B)$ is at least $(4 R_p/L_0)^{d-1} \stackrel{(\ref{2.23})}{\ge} c\,M^{-(p+1)(d-1)} (N/L_0)^{d-1} \stackrel{(\ref{1.9})}{\ge} c\, M^{-(p+1)(d-1)} N^{d-2} \gg \rho(L_0) \, N^{d-2}$. As a result, for large $N$ on $\Omega_{\ve,N} \subseteq \cB^c_N$, see (\ref{4.6}), for any box $B$ of Type $\IB$, in any coordinate direction only a vanishing fraction of columns of $B_0$-boxes in $\Delta '(B)$ contains an $(\alpha,\beta,\gamma)$-bad $B_0$-box.

\medskip
So we consider an enumeration
\begin{equation}\label{4.45}
\begin{array}{l}
\mbox{$\Delta '_j$, $1 \le j \le J^\IB$ of the $\Delta '(B)$, where $B$ runs over the collection of}
\\
\mbox{Type $\IB$ boxes, which thus satisfies $\bigcup_{{\rm Type}\,\IB} B \subseteq \bigcup_{1 \le j \le J^{\IB}} \Delta'_j$}.
\end{array}
\end{equation}

\n
We can apply isoperimetric considerations (as in (\ref{4.27}), (\ref{4.29})), where we now leverage the above mentioned rarity of $(\alpha,\beta,\gamma)$-bad boxes, to obtain that
\begin{equation}\label{4.46}
\left\{ \begin{array}{l}
\mbox{for large $N$, on the event $\Omega_{\ve,N}$ in (\ref{4.6}), for any $1 \le j \le J^{\IB}$, one can find}
\\
\mbox{a coordinate projection $\pi_j '\,^{\!\!,\,\IB}$ and at least $c_6 \, (|\Delta'_j| \,/ \, |B_0|)^{\frac{d-1}{d}}$ $B_0$-boxes}
\\
\mbox{of $\partial_{B_0} \,\cU_1$ in distinct $\pi_j '\,^{\!\!,\,\IB}$-columns, contained in $\Delta'_j$, all $(\alpha,\beta,\gamma)$-good,}
\\
\mbox{with $N_u(D_0) \ge \beta \,{\rm cap}(D_0)$, and with base points at distance at least $c_7\, L_1$}
\\[0.3ex]
\mbox{from $\partial \Delta'_j$}.
\end{array}\right.
\end{equation}

\n
Recall $\delta$ and $\wt{p}$ from (\ref{2.7}) and (\ref{4.9}). We then say that a box $B$ of Type $\IB$ is {\it rarefied} when (with analogous notation as in (\ref{4.35}), (\ref{4.36})):
\begin{equation}\label{4.47}
\dsl_{B^{(\wt{p})} \subseteq \wt{B} \subseteq B^{(1)}} \; \dis\frac{{\rm cap}\big(\wt{B} \cap (\bigcup_{1 \le j \le J^{\IB}} \Delta'_j)\big)}{|\wt{B}|^{\frac{d-2}{d}}} < \delta \, \wt{p} ,
\end{equation}
and that it is {\it substantial} when
\begin{equation}\label{4.48}
\dsl_{B^{(\wt{p})} \subseteq \wt{B} \subseteq B^{(1)}} \; \dis\frac{{\rm cap}\big(\wt{B} \cap (\bigcup_{1 \le j \le J^{\IB}} \Delta'_j)\big)}{|\wt{B}|^{\frac{d-2}{d}}} \ge \delta \, \wt{p} .
\end{equation}

\n
The volume controls of Proposition \ref{prop2.1} (with $q = \wt{p}$ and $F = \bigcup_{1 \le j \le J^{\IB}} \Delta'_j$), noting that (\ref{2.9}) holds due to (\ref{4.12}), now yield
\begin{equation}\label{4.49}
\dsl_{B: {\rm Type}\, \IB, {\rm rarefied}} \big| B \cap \big(\dis\bigcup\limits_{1\le j \le J^{\IB}} \Delta'_j\big)\big| \le c_3 \, |\ov{D}_N| \; \Big(\mbox{\f $\dis\frac{M^2}{3^d + 1}$}\Big)^{-\wt{p}/2} \stackrel{(\ref{4.9}),(\ref{2.25})}{\le} \mbox{\f $\dis\frac{\ve}{20}$} \;|D_N| .
\end{equation}
We thus set
\begin{equation}\label{4.50}
\Gamma^{\IB} = \bigcup_{B: {\rm Type}\, \IB, {\rm rarefied}} B \cap \big(\bigcup\limits_{1 \le j \le  J^{\IB}} \Delta'_j\big) \subseteq \ov{D}_N,
\end{equation}
and find that
\begin{equation}\label{4.51}
Bub \cap \big(\bigcup\limits_{{\rm Type}\, \IB} B\big) \subseteq \big(\bigcup_{B: {\rm Type}\, \IB, {\rm substantial}} B\big)  \cup \Gamma^{\IB} \subseteq \ov{D}_N
\end{equation}
with
\begin{equation}\label{4.52}
|\Gamma^{\IB} | \le \mbox{\f $\dis\frac{\ve}{20}$} \;|D_N| .
\end{equation}

\n
We finally turn to the treatment of boxes of Type $\IA$. Our goal is to establish (\ref{4.62}). To this end we have in mind to apply the results of Section 3 concerning resonance sets. This first requires some preparation. We introduce the notation (with $\cU_1$ as in (\ref{1.22}))
\begin{equation}\label{4.53}
s_r(x) = \dis\frac{|B(x,r) \cap \cU_1|}{|B(x,r)|},\; \mbox{for $x \in \IZ^d$ and $r \ge 0$ integer},
\end{equation}
as well as the {\it $J$-resonance set} (recall $\wt{p}$, $p$ from (\ref{4.9}), (\ref{4.11}), and $R_\ell$ from (\ref{2.23})),
\begin{equation}\label{4.54}
\cR es_J = \Big\{x \in \IZ^d; \dsl_{\wt{p} < \ell \le p-4} 1\{s_{8R_\ell}(x) \in [\wt{\alpha}, 1 - \wt{\alpha}]\} \ge J\Big\}, \; \mbox{with $\wt{\alpha} = \mbox{\f $\dis\frac{1}{3}$} \;4^{-d}$ as in (\ref{3.7})}.
\end{equation}

\n
Note that when $B \in \cI_p$ and $x,y \in B$, then $|x-y|_1 \le d\,R_p$ and for each $0 \le \ell \le p-4$, one has
\begin{equation}\label{4.55}
| s_{8R_\ell}(x) - s_{8R_\ell}(y)| \stackrel{(\ref{A.8})}{\le} \mbox{\f $\dis\frac{|x-y|_1}{8R_\ell}$} \le \mbox{\f $\dis\frac{d}{8}$}\; \mbox{
\f $\dis\frac{R_p}{R_{p-4}}$} = \mbox{\f $\dis\frac{d}{8M^4}$} \stackrel{(\ref{4.10})}{<}  \mbox{\f $\dis\frac{\wt{\alpha}}{8}$}.
\end{equation}

\n
Thus, when $B \subseteq \cI_p$ intersects ${\cR} es_J$, we can apply (\ref{4.55}) to $y = x_B$ (the base point of $B$) and $x$ in $B \cap {\cR} es_J$ to find that
\begin{equation}\label{4.56}
\begin{array}{ll}
\cR es_J \subseteq \wh{{\cR} es}_J \stackrel{\rm def}{=} & \!\!\! \mbox{the union of boxes $B \in \cI_p$ such that}
\\
& \dsl_{\wt{p} < \ell \le p-4} \;1\{s_{8R_\ell} (x_B) \in [ \wh{\alpha}, 1-\wh{\alpha}]\} \ge J, \; \mbox{where $\wh{\alpha} = \mbox{\f $\dis\frac{1}{6}$} \;4^{-d}$} .
\end{array}
\end{equation}

\n
As we now explain, the results of Section 3 show that when $N$ is large, the simple random walk starting in a box $B$ of Type $\IA$ enters $\cR es_J$ (and therefore $\wh{\cR es}_J$ as well) with ``high probability''.

\medskip
To this end we recall that $M = 2^{b}$, see (\ref{2.1}), and we set (see (\ref{4.16}))
\begin{equation}\label{4.57}
\begin{array}{l}
\cA = \big\{m = 1 + b\, \big(\ell_N - p + 4 + i \,(J + 1) \,L(J)\big); \; 1 \le i \le I\big\}
\\[1ex]
\mbox{(so that $\{2^m L_1, m \in \cA\} = \{2 R_\ell$; $\ell \in \cS\}$ in the notation of (\ref{4.16})),} 
\end{array}
\end{equation}

\n
and as in (\ref{3.6}) with $L = b \,L(J)$, we set
\begin{equation}\label{4.58}
\begin{array}{l}
\cA_* =   \{m \in \IN; \; \mbox{for some $0 \le j \le J$, $m + j \, L \in \cA\}$}
\\
\quad \;\;=  \big\{m = 1 + b\,\big(\ell_N - p+ 4 + k \,L(J)\big); 1 \le k \le I(J+1)\big\}
\\[1ex]
\mbox{(so that $\{2^m\,L_1$, $m \in \cA^*\} \subseteq \{2 R_\ell$; $p-4 \ge \ell \ge \wt{p} + 1\}$, since}
\\
\mbox{$\wt{p} + 1 \stackrel{(\ref{4.11})}{=} p-4 - I (J + 1) \,L(J)$)}.
\end{array}
\end{equation}

\n
Consider now $B \in \cI_p$, a box of Type $\IA$, so that, see (\ref{4.17}), $|\Delta_{B,\ell} \cap \cU_0| \ge \frac{3}{4}| \, \Delta_{B,\ell}|$ for all $\ell \in \cS$. When $N$ is large, this implies that (see (\ref{4.14}))
\begin{equation}\label{4.59}
|B(x_B, 2 R_\ell) \cap \cU_0| \ge \mbox{\f $\dis\frac{5}{8}$} \;|B(x_B, 2R_\ell)|, \; \mbox{i.e. $s_{2R_\ell}(x_B) \le \mbox{\f $\dis\frac{3}{8}$}$, for each $\ell \in \cS$},
\end{equation}

\n
and by a similar bound as in (\ref{4.55}) it follows that
\begin{equation}\label{4.60}
\mbox{for each $x \in B$ and $\ell \in \cS$, $s_{2R_\ell}(x) \le \mbox{\f $\dis\frac{3}{8}$} + \mbox{\f $\dis\frac{d}{2M^4}$} \stackrel{(\ref{4.10})}{<} \fr$}.
\end{equation}

\n
Thus, for $x \in B$, setting $U_0 = \cU_0 - x$ (and $U_1 = \cU_1 - x)$ in (\ref{3.3}), and noting that $|B(0,2^m \,L_1) \cap U_1 | < \frac{1}{2} \;|B(0,2^m \, L_1)|$ for all $m \in \cA$ by (\ref{4.60}) and the second line of (\ref{4.57}), we find by Proposition \ref{prop3.1} (note that (\ref{3.10}) holds by (\ref{4.12})):
\begin{equation}\label{4.61}
P_x [H_{\cR es_J} = \infty] \stackrel{(\ref{4.58})}{\le} P_0 [H_{R es_{\cA_*,J}} = \infty] \stackrel{(\ref{3.13})}{\le} \gamma_{I,J} \stackrel{(\ref{4.8})}{\le} (1-a) /10
\end{equation}
(see (\ref{3.7}) for the definition of $R es_{\cA_*,J})$.

\medskip
Taking (\ref{4.56}) into account, this shows that for large $N$
\begin{equation}\label{4.62}
\mbox{for any $x \in \bigcup_{{\rm Type}\,\IA} B, \; P_x [H_{\wh{\cR es}_J} = \infty] \le \gamma_{I,J} \le (1-a)/ 10$}.
\end{equation}

\n
In addition, when $B \in \cI_p$ is contained in $\wh{\cR es}_J$, then for some $\ell \in (\wt{p}, p-4], \,B(x_B, 8R_\ell)$ intersects $\cU_0$ and $\cU_1$ and hence $B(0,3 N + L_0)$, see (\ref{1.23}), so that
\begin{equation*}
\mbox{$B \subseteq B(x_B, 8R_\ell) \subseteq B(0, 3N + L_0 + 16 R_{\wt{p} + 1})$ with $16R_{\wt{p} + 1} \stackrel{(\ref{4.9})}{\le} 16M^{-5}  \,N \stackrel{(\ref{2.1})}{<} N / 20$}.
\end{equation*}
As a result, for large $N$,
\begin{equation}\label{4.63}
\wh{\cR es}_J \subseteq \Big[- \big(3 + \mbox{\f $\dis\frac{1}{10}$}\big) \, N, \; \big( 3 +\mbox{\f $\dis\frac{1}{10}$}\big) \,N\Big]^d \cap \IZ^d \stackrel{\rm def}{=} \wh{D}_N.
\end{equation}

\n
We then define the following set, which is a union of boxes of $\cI_p$:
\begin{equation}\label{4.64}
A = \wh{\cR es}_J  \cup \big(\bigcup\limits_{B:{\rm Type}\, \IB,{\rm substantial}} B\big)  \cup \big(\bigcup\limits_{B:{\rm Type}\, \Ib,{\rm substantial}} B\big) \subseteq \wh{D}_N
\end{equation}

\n
(where for the last inclusion we used (\ref{4.63}), (\ref{4.51}), (\ref{4.39}) and $\ov{D}_N \subseteq [-2N,2N]^d \subseteq \wh{D}_N$, see (\ref{2.25})).

\medskip
We can now collect (\ref{4.39}), (\ref{4.40}) for Type $\Ib$, (\ref{4.51}), (\ref{4.52}) for Type $\IB$, and (\ref{4.62}) for Type $\IA$, to find for large $N$ on the event $\Omega_{\ve,N}$ in (\ref{4.6})
\begin{equation}\label{4.65}
\mbox{for all $x \in Bub \, \backslash \, (\Gamma^{\Ib} \cup \Gamma^{\IB})$, $P_x [H_A = \infty] \le (1-a)/10$,  and $|\Gamma^{\Ib} \cup \Gamma^{\IB}| \le \mbox{\f $\dis\frac{\ve}{5}$} \;|D_N|$}.
\end{equation}

\n
The next step on the way to the proof of Theorem \ref{4.1} is (we recall (\ref{4.7}), (\ref{4.9}), (\ref{4.11}) for notation).

\begin{proposition}\label{prop4.2}
For large $N$ on the event $\Omega_{\ve,N}$ in (\ref{4.6}), for each $B \in \cI_p$ included in $A$, see (\ref{4.64}), we can find a subset $\cL_B \subseteq \{1,\dots,p\}$ with $|\cL_B| = J$, so that for each $B$ as above and $\ell \in \cL_B$, there is a collection $\cC_{B,\ell}$ of $(\alpha,\beta,\gamma)$-good $B_0$-boxes with $N_u(D_0) \ge \beta \,{\rm cap} (D_0)$ intersecting $[-3N, 3N]^d$ such that
\begin{equation}\label{4.66}
\begin{array}{l}
\mbox{the $B_0$-boxes in $\cC \stackrel{\rm def}{=} \bigcup_{B \in \cI_p, B \subseteq A, \ell \in \cL_B} \cC_{B,\ell}$ have base points}
\\
\mbox{at mutual $| \cdot |_\infty$-distance at least $\ov{K} L_0$},
\end{array}
\end{equation}
and for each $B \in \cI_p$, $B \subseteq A$, $\ell \in \cL_B$,
\begin{align}
& \mbox{each $B_0 \in \cC_{B,\ell}$ is contained in $B(x_B, 10 R_\ell)$ (see (\ref{2.23}) for notation)}, \label{4.67}
\\[1ex]
& \mbox{$\cC_{B,\ell}$ consists of at most $c_8 (R_\ell / L_0)^{d-2}\, B_0$-boxes}, \label{4.68}
\\[1ex]
&\mbox{for each $x \in B(x_B, 10 R_\ell), \;P_x[H_{\bigcup_{\cC_{B,\ell}} B_0} < T_{B(x_B,10 R_{\ell - 1})}] \ge c_9 \;(> 0)$}. \label{4.69}
\end{align}
\end{proposition}

Once we show Proposition \ref{prop4.2} it will be a quick step to complete the proof of Theorem \ref{theo4.1}. It may be helpful at this stage to provide a brief outline of the proof of Proposition~\ref{prop4.2}.

\medskip
In a first step we will define the set $\cL_B$ for each $B \in \cI_p$ contained in $A$, leveraging the fact that such a $B$ is substantial when it is of Type $\Ib$ or $\IB$, and that otherwise it is included in $\wh{Res}_J$. The selected levels will be such that for each $\ell \in \cL_B$ there is a collection of disjoint sub-boxes within $B(x_B, 8R_\ell)$  such that their union has a non-degenerate capacity in $B(x_B,8R_\ell)$, and within each such sub-box there is a ``surface-like'' presence of $(\alpha,\beta,\gamma)$-good $B_0$-boxes from $\partial_{B_0} \,\cU_1$ (hence such that $N_u (D_0) \ge \beta\, {\rm cap}(D_0)$, see below (\ref{1.23})), having disjoint projection in some coordinate direction.

\medskip
In a second step, with the help of Lemma \ref{lem1.1}, we will extract for each $B$ and $\ell \in \cL_B$ as above, a collection $\wt{\cC}_{B,\ell}$ of $(\alpha,\beta,\gamma)$-good $B_0$-boxes with $N_u(D_0) \ge \beta\, {\rm cap} (D_0)$, contained in $B(x_B, 8R_\ell)$, with base points at mutual distance at least $H \,\ov{K}\,L_0$ (with $H$ ``large'' and solely depending on $d,\alpha, \ve, J$, see (\ref{4.78})), so that the union of the $B_0$-boxes from $\wt{\cC}_{B,\ell}$ has a non-degenerate capacity in $B(x_B, 8R_\ell)$, see (\ref{4.85}), (\ref{4.86}). We will then consider the union $\wt{\cC}$ of these collections $\wt{\cC}_{B,\ell}$ of $B_0$-boxes. At this stage the mutual distance between $B_0$-boxes in $\wt{\cC}$ might be smaller than $\ov{K} L_0$ (and $\wt{\cC}$ need not satisfy (\ref{4.66})).

\medskip
In the third and last step we will introduce an equivalence relation within $\wt{\cC}$, for which two $B_0$-boxes of $\wt{\cC}$ lie in the same equivalence class if they can be joined by a path of boxes in $\wt{\cC}$ with steps of $| \cdot |_\infty$-size at most $\ov{K}\,L_0$. We will show that the equivalence classes have a ``small size'', see (\ref{4.90}). Then, we will select a representative in each equivalence class, and for each $B \in \cI_p$, $B \subseteq A$, and $\ell \in \cL_B$ consider the collection of the representatives of the boxes in $\wt{\cC}_{B,\ell}$. With the help of Lemma \ref{lem1.2} (which in essence goes back to \cite{AsseScha20}), we will extract the desired collections $\cC_{B,\ell}$, so that (\ref{4.66}) - (\ref{4.69}) hold.

\bigskip\n
{\it Proof of Proposition \ref{prop4.2}:} Our first step is to define $\cL_B$, for $B \in \cI_p$ contained in $A$, see (\ref{4.64}). We start with the case where $B$ is of Type $\Ib$, substantial, then we proceed with the case where $B$ is of Type $\IB$, substantial, and finally we handle the case where $B$ is contained in
\begin{equation}\label{4.70}
\wh{A} = \wh{\cR es}_J \,  \backslash \, \Big\{\big(\bigcup\limits_{B:{\rm Type}\, \Ib,{\rm substantial}} B\big)  \cup \big(\bigcup\limits_{B:{\rm Type}\, \IB,{\rm substantial}} B\big) \Big\}.
\end{equation}
Thus, we first consider $B$ of Type $\Ib$, substantial. By (\ref{4.36}) we know that $\mu \stackrel{\rm def}{=} |\{ 1 \le \ell \le \wt{p}$; ${\rm cap} (B^{(\ell)} \cap (\bigcup_{j \in \cG^{\Ib}} B'_j)) \ge \frac{\delta}{2} \;|B^{(\ell)}|^{\frac{d-2}{d}}\}|$ satisfies the inequality $\wh{c} \,\mu + \frac{\delta}{2} \,\wt{p} \ge \wh{c} \,\mu + \frac{\delta}{2} \,(\wt{p} - \mu) \ge \delta \,\wt{p}$ (with $\wh{c}$ from (\ref{1.8})), so that $\mu \,\wh{c} \ge \frac{\delta}{2} \,\wt{p}$ and hence
\begin{equation}\label{4.71}
\Big|\Big\{1 \le \ell \le \wt{p}; \, {\rm cap}\,\Big(B^{(\ell)} \cap \big(\bigcup_{j \in \cG^{\Ib}} B'_j\big)\Big) \ge \mbox{\f $\dis\frac{\delta}{2}$} \; |B^{(\ell)} |^{\frac{d-2}{d}}\Big\}\Big| \ge  \mbox{\f $\dis\frac{\delta}{2 \wh{c}}$} \;\wt{p} \stackrel{(\ref{4.9}) \,{\rm ii)}}{\ge} J
\end{equation}
(let us incidentally point out that $\frac{\delta}{2\wh{c}} < 1$, see Remark \ref{rem2.3}).

\medskip
So for $B$ of Type $\Ib$, substantial, we define
\begin{equation}\label{4.72}
\begin{array}{l}
\mbox{$\cL_B$ as the collection of the $J$ largest integers $\ell$ in $[1,\wt{p}]$ such that}\\
{\rm cap} (B^{(\ell)} \cap \big(\bigcup_{j \in \cG^{\Ib}} B'_j)\big) \ge \frac{\delta}{2} \;|B^{(\ell)}|^{\frac{d-2}{d}}.
\end{array}
\end{equation}

\n
We then turn to the case of a box $B$ of Type $\IB$, substantial. Using (\ref{4.48}) in place of (\ref{4.36}), a similar argument as above shows that 
\begin{equation}\label{4.73}
\Big| \Big\{1 \le \ell \le \wt{p}; \; {\rm cap}\,\Big(B^{(\ell)} \cap \big(\bigcup_{1 \le j \le J^{\IB}} \Delta'_j\big)\Big)\ge \mbox{\f $\dis\frac{\delta}{2}$} \;|B^{(\ell)}|^{\frac{d-2}{d}}\Big\}\Big| \ge \mbox{\f $\dis\frac{\delta}{2 \wh{c}}$}\;\wt{p}  \ge J,
\end{equation}

\n
and for $B$ of Type $\IB$, substantial, we define
\begin{equation}\label{4.74}
\begin{array}{l}
\mbox{$\cL_B$ as the collection of the $J$ largest integers $\ell$ in $[1,\wt{p}]$ such that}\\
{\rm cap} (B^{(\ell)} \cap \big(\bigcup_{1 \le j \le J^{\IB}} \Delta'_j)\big) \ge \frac{\delta}{2} \;|B^{(\ell)}|^{\frac{d-2}{d}}.
\end{array}
\end{equation}

\n
In addition, as we now explain, for $B$ and $\ell \in \cL_B$ as in (\ref{4.74}), we can extract a sub-collection $\cJ_{B,\ell} \subseteq \{1, \dots, J^{\IB}\}$ such that
\begin{equation}\label{4.75}
\left\{ \begin{array}{l}
\mbox{for each $j \in\cJ_{B,\ell}$, $\Delta'_j \cap B^{(\ell)} \not= \phi$, the $\Delta'_j$, $j \in \cJ_{B,\ell}$ are pairwise disjoint,}
\\
\mbox{and any $\Delta'_k$, $k \in \{1, \dots, J^{\IB}\}$ intersecting $B^{(\ell)}$ is contained in some}
\\
\mbox{$\wt{\Delta}_j$, $j \in \cJ_{B,\ell}$, where $\wt{\Delta}'_j$ denotes the closed ball in supremum-distance}
\\
\mbox{with triple radius and same center as $\Delta'_j$}
\end{array}\right.
\end{equation}

\n
(we refer to the unique $x\in R_p \, \IZ^d$ and $r \ge 1$ such that $\Delta'_j = x + [-r,r)^d$ as  the ``center'' and the ``radius'' of $\Delta'_j$, and recall that the size of $\Delta'_j$, i.e.~$2r$, is at most $4 R_{\wt{p} + 1}$, see (\ref{4.44})).

\medskip
To prove (\ref{4.75}) we use a routine cover argument. We list the $\Delta'_j$, $1\le j \le J^{\IB}$ intersecting $B^{(\ell)}$ by decreasing size. We first consider the first such $\Delta'_j$\,, thus of largest size, and delete from the list all the other $\Delta'_k$ that intersect $\Delta'_j$. They are all contained in $\wt{\Delta}'_j$. If the remaining list is empty, we are done. Otherwise, we proceed with the next $\Delta'_k$ in the list, which by construction does not intersect the first chosen $\Delta'_j$, and proceed similarly until coming to an empty list.

\medskip
We thus find that for $B$ of Type $\IB$, substantial, and $\ell \in \cL_B$,
\begin{equation}\label{4.76}
B^{(\ell)} \cap \big( \bigcup_{1\le j \le J^{\IB}} \Delta'_j\big) \subseteq B^{(\ell)} \cap \big(\bigcup_{j \in \cJ_{B,\ell}} \wt{\Delta}'_j\big),
\end{equation}

\n
where the $\Delta'_j, j \in \cJ_{B,\ell}$ are pairwise disjoint, intersect $B^{(\ell)}$ and have size at most $4R_{\wt{p} + 1}$ (and by (\ref{4.51}) are thus contained in $B(0,2N + 4 R_{\wt{p} + 1}) \subseteq B(0,2N + R_{\wt{p}}) \subseteq [-3N, 3N]^d \subseteq \wh{D}_N$, (see (\ref{4.63})).

\medskip
Finally, we turn to the case of $B \in \cI_p$ contained in $\wh{A}$ (see (\ref{4.70})). Then, $B \subseteq \wh{\cR es}_J$, see (\ref{4.56}), and we define
\begin{equation}\label{4.77}
\begin{array}{l}
\mbox{$\cL_B$ as the collection of the $J$ largest integers in $[\wt{p} + 1, p-4]$ such that}\\
s_{8R_\ell} (x_B) \; \Big(= \dis\frac{|B(x_B,8R_\ell) \cap \cU_1|}{|B(x_B,8R_\ell)|} \Big) \in [\wh{\alpha}, 1 - \wh{\alpha}] \quad \mbox{(recall  $\wh{\alpha} = \frac{1}{6} \; 4^{-d}$)} .
\end{array}
\end{equation}

\n
We have now defined $\cL_B$ for any $B \in \cI_p$ contained in $A$ through (\ref{4.72}), (\ref{4.74}), (\ref{4.77}) and thus completed the first step of the proof of Proposition \ref{prop4.2}.

\medskip
In the second step we are going to introduce for each $B$ and $\ell \in \cL_B$ as above a collection $\wt{\cC}_{B,\ell}$ of $B_0$-boxes, which are $(\alpha,\beta,\gamma)$-good with $N_u(D_0) \ge \beta \, {\rm cap}(D_0)$, contained in $B(x_B,8R_\ell)$, with base points at mutual $| \,\cdot \,|_\infty$-distance at least $H \, \ov{K}\,L_0$, where $H$ is defined in (\ref{4.78}) below, and such that $\bigcup_{\wt{\cC}_{B,\ell}} \,B_0$ has a non-degenerate capacity in $B(x_B, 8R_\ell)$, see (\ref{4.85}), (\ref{4.86}).

\medskip
With this in mind, we introduce (we recall that $p$ in (\ref{4.11}) depends on $d,a,\ve,J)$:
\begin{equation}\label{4.78}
H(a,\ve,J) = 10^{d+1}  \;J M^{(p+1)d}.
\end{equation}

\n
We begin with the case of $B$ of Type $\Ib$, substantial. We recall that by (\ref{4.29}), for any $j \in \cG^{\Ib}$, there is a coordinate projection $\pi_j '\,^{\!\!,\,\Ib}$ and at least $\frac{1}{2} \; c_4 (|B'_j | \, / \, |B_0|)^{\frac{d-1}{d}} B_0$-boxes in distinct $\pi_j '\,^{\!\!,\,\Ib}$ -columns, $(\alpha,\beta,\gamma$)-good and such that $N_u(D_0) \ge \beta\,{\rm cap}(D_0)$, contained in $B'_j$, with base points at distance at least $c_5 \,L_1$ from $\partial  B'_j$. We can now apply Lemma \ref{lem1.1}, so that for large $N$ on the event $\Omega_{\ve,N}$ in (\ref{4.6}):
\begin{equation}\label{4.79}
\left\{\begin{array}{l}
\mbox{for each $j \in \cG^{\Ib}$ one can find a collection $\cC^{\,',\Ib}_j$ of $(\alpha,\beta,\gamma)$-good $B_0$-boxes}
\\
\mbox{contained in $B'_j \;(\subseteq \,\ov{D}_N$, see (\ref{4.25})), such that $N_u(D_0) \ge \beta\, {\rm cap}(D_0)$,}
\\
\mbox{at distance at least $c_5\,L_1$ from $\partial B'_j$, with base points having}
\\
\mbox{$\pi_j '\,^{\!\!,\,\Ib}$-projection at mutual $|\cdot |_\infty$-distance at least $H \,\ov{K} L_0$, and so that}
\\
\mbox{${\rm cap} \,\big(\bigcup_{\cC^{',\Ib}_{j}} B_0\big) \ge c_{10} \,|B'_j|^{\frac{d-2}{d}}$}.
\end{array}\right.
\end{equation}

\n
We then define for each $B$ of Type $\Ib$, substantial, and $\ell \in \cL_B$ the collection of $B_0$-boxes:
\begin{equation}\label{4.80}
\wt{\cC}_{B,\ell} = \bigcup\limits_{j \in \cG^{\Ib}, B'_j \,{\rm intersects}\, B^{(\ell)}}  \cC^{\,',\Ib}_j
\end{equation}

\n
(we recall that when $B'_j$ intersects $B^{(\ell)}$,  then $B'_j \subseteq B^{(\ell)}$, and the $B'_j, j \in \cG^{\Ib}$ are pairwise disjoint, see (\ref{4.21}), (\ref{4.25})).

\medskip
We then turn to the definition of $\wt{\cC}_{B,\ell}$ for $B$ of Type $\IB$, substantial, and $\ell \in \cL_B$. We combine (\ref{4.46}) for $j\in \cJ_{B,\ell}$ and Lemma \ref{lem1.1} to find that for large $N$ on the event $\Omega_{\ve,N}$ in (\ref{4.6}):
\begin{equation}\label{4.81}
\left\{\begin{array}{l}
\mbox{for any $B$ of Type $\IB$, substantial, $\ell \in \cL_B$, and $j \in \cJ_{B,\ell}$, one can find a}
\\
\mbox{collection $\cC^{\,',\IB}_{j,B,\ell}$ of $(\alpha,\beta,\gamma)$-good $B_0$-boxes contained in $\Delta'_j \; (\subseteq [-3N,3N]^d$,}
\\
\mbox{see below (\ref{4.76})), with $N_u(D_0) \ge \beta\,{\rm cap}(D_0)$, at distance at least $c_7\,L_1$ from}
\\
\mbox{$\partial \Delta'_j$, with base points having $\pi_j '\,^{\!\!,\,\IB}$-projections at mutual $|\cdot |_\infty$-distance}
\\
\mbox{at least $H \,\ov{K}L_0$, such that ${\rm cap} \,\big(\bigcup_{\cC^{',\IB}_{j,B,\ell}} B_0\big) \ge c_{11} \,|\Delta'_j|^{\frac{d-2}{d}}$}.
\end{array}\right.
\end{equation}

\n
We then define for $B$ of Type $\IB$, substantial, and $\cC \in \cL_B$ the collection of $B_0$-boxes
\begin{equation}\label{4.82}
\wt{\cC}_{B,\ell} = \bigcup\limits_{j \in \cJ_{B,\ell}} \cC^{\,',\IB}_{j,B,\ell}
\end{equation}

\n
(and, see below (\ref{4.76}), all $\Delta'_j$, $j \in \cJ_{B,\ell}$ are contained in $B(x_B, R_\ell + 4 R_{\wt{p} + 1}) \stackrel{\ell \le \wt{p}}{\subseteq} B(x_B,2 R_\ell)$ and all $B_0$-boxes in $\wt{\cC}_{B,\ell}$ are contained in $[-3N,3N]^d$).

\medskip
Finally, we turn to the definition of $\wt{\cC}_{B,\ell}$ for $B \in \cI_p$, $B \subseteq \wh{A}$, see (\ref{4.70}), and $\ell \in \cL_B$, see (\ref{4.77}). Recall that $\cL_B \subseteq [\wt{p} + 1, p-4]$ and $s_{8R_\ell} (x_B) \in [\wh{\alpha}, 1 - \wh{\alpha}]$ for each $\ell \in \cL_B$. Using isoperimetry (see (A.3) - (A.6), p.~480-481 of \cite{DeusPisz96}), we see that for large $N$ on the event $\Omega_{\ve, N}$ in (\ref{4.6}):
\begin{equation}\label{4.83}
\begin{array}{l}
\mbox{for each $B \in \cI_p$ contained in $\wh{A}$ and $\ell \in \cL_B$, there is a coordinate projection}
\\
\mbox{$\pi_{B,\ell}$ and a collection of at least $c_{12}\,(|B (x_{B}, 8R_\ell)| \, / \, |B_0|)^{\frac{d-1}{d}}$ $B_0$-boxes}
\\
\mbox{from $\partial_{B_0} \,\cU_1$ (see below (\ref{1.23})), in $B(x_B, 8R_\ell)$, with distinct $\pi_{B,\ell}$-projections}.
\end{array}
\end{equation}

\medskip\n
Note that any $B_0 \in \partial_{B_0} \, \cU_1$ intersects $[-3N,3N]^d$, see (\ref{1.22}), and on $\cB^c_N \;(\supset \Omega_{\ve,N})$ there are at most $\rho(L_0) \,N^{d-2}$ $(\alpha,\beta,\gamma)$-bad $B_0$-boxes intersecting $[-3N, 3N]^d$, see (\ref{1.24}). As $N \r \infty$, $\rho (L_0) \, N^{d-2} \ll (R_p / L_0)^{d-1} \le (R_\ell/L_0)^{d-1}$, and we see that for large $N$ on the event $\Omega_{\ve,N}$, for each $B \in \cI_p$, $B \subseteq \wh{A}$, and $\ell \in \cL_B$, we can find a collection of at least $\frac{1}{2} \;c_{12} (| B(x_B, 8R_\ell)| \, / \, |B_0|)^{\frac{d-1}{d}}$ $(\alpha,\beta,\gamma)$-good $B_0$-boxes in $B(x_B,8R_\ell)$ with $N_u(D_0) \ge \beta \, {\rm cap}(D_0)$, having distinct $\pi_{B,\ell}$-projections, and intersecting $[-3N,3N]^d$.

\medskip
Once again we apply Lemma \ref{lem1.1} and find that for large $N$ on the event $\Omega_{\ve,N}$ in (\ref{4.6}):
\begin{equation}\label{4.84}
\left\{\begin{array}{l}
\mbox{for any $B \in \cI_p$ contained in $\wh{A}$, and $\ell \in \cL_B$, there is a collection $\wt{\cC}_{B,\ell}$}
 \\
\mbox{of $(\alpha,\beta,\gamma)$-good $B_0$-boxes with $N_u(D_0) \ge \beta \, {\rm cap}(D_0)$, contained in}
\\
\mbox{$B(x_B,8R_\ell)$, intersecting $[-3N,3N]^d$, with base points having}
\\
\mbox{$\pi_{B,\ell}$-projections at mutual $|\cdot |_\infty$-distance at least $H \, \ov{K}L_0$ and such}
\\
\mbox{that ${\rm cap} \,\big(\bigcup_{\wt{\cC}_{B,\ell}} B_0) \ge c_{13}\, |B(x_B,8R_\ell)|^{\frac{d-2}{d}}$}.
\end{array}\right.
\end{equation}

\n
Collecting (\ref{4.80}), (\ref{4.82}), (\ref{4.84}), for large $N$ on the event $\Omega_{\ve, N}$ in (\ref{4.6}), we have defined $\wt{\cC}_{B,\ell}$ for each $B \in \cI_p$ contained in $A$ and $\ell \in \cL_B$. In particular for such $B$ and $\ell$
\begin{equation}\label{4.85}
\begin{array}{l}
\mbox{the boxes $B_0$ in $\wt{\cC}_{B,\ell}$ are $(\alpha,\beta,\gamma)$-good with $N_u(D_0) \ge \beta \, {\rm cap}(D_0)$, are}
\\
\mbox{contained in $B(x_B,8R_\ell)$, intersect $[-3N,3N]^d$, and have base points}
\\
\mbox{at mutual $| \cdot |_\infty$-distance at least $H\,\ov{K}L_0$}.
\end{array}
\end{equation}

\n
In addition, as we now explain
\begin{equation}\label{4.86}
\mbox{for any $B, \ell$ as above, ${\rm cap} \,\big(\dis\bigcup\limits_{\wt{\cC}_{B,\ell}} B_0\big) \ge c\,R_\ell^{d-2}$}.
\end{equation}

\n
Indeed, when $B$ is of Type $\Ib$, substantial, and $\ell \in \cL_B$, we know by (\ref{4.72}) that ${\rm cap}(B^{(\ell)} \cap (\bigcup_{j \in \cG^{\Ib}} B'_j)) \ge \frac{\delta}{2} \;| B^{(\ell)}|^{\frac{d-2}{d}}$ and by (\ref{4.79}) that for each $j \in \cG^{\Ib}$, ${\rm cap}\,(\bigcup_{\cC^{\,',\Ib}_j}B_0) \ge c_{10} \, |B'_j|^{\frac{d-2}{d}}$. Since $\wt{\cC}_{B,\ell}$ is the union of the $\cC_j^{\,',\Ib}$ for the $j \in \cG^{\Ib}$ such that $B^{(\ell)} \cap B'_j \not= \phi$ (and hence such that $B'_j \subseteq B^{(\ell)}$), see (\ref{4.80}), we find by the strong Markov property and the repeated application of the second line of (\ref{1.5}) that the simple random walk starting in $B^{(\ell)}$ enters $\bigcup_{\wt{\cC}_{B,\ell}} B_0$ ($\subseteq B^{(\ell)}$) with a probability bounded below by a constant. In this lower bound, integrating the starting point of the walk with the equilibrium measure $e_{B^{(\ell)}}$, and using (\ref{1.5}) as well as $h_{B^{(\ell)}} = 1$ on $B^{(\ell)} \supseteq \bigcup_{\wt{\cC}_{B,\ell}} B_0$, we find that ${\rm cap} \, (\bigcup_{\wt{\cC}_{B,\ell}} B_0) \ge c\,|B^{(\ell)} |^{\frac{d-2}{d}} = c\,R_\ell^{d-2}$.

\bigskip\n
When $B$ is of Type $\IB$, substantial, and $\ell \in \cL_B$, we know by (\ref{4.73}), (\ref{4.75})  that ${\rm cap} \,(B^{(\ell)} \cap (\bigcup_{j \in \cJ_{B,\ell}} \wt{\Delta} '_j)) \ge {\rm cap} \,(B^{(\ell)} \cap (\bigcup_{1 \le j \le J^{\IB}} \Delta '_j)) \ge \frac{\delta}{2} \;|B^{(\ell)}|^{\frac{d-2}{d}}$. In addition, for any $j \in \cJ_{B,\ell}$, ${\rm cap}\,(\bigcup_{ \cC^{\,',\IB}_{j,B,\ell}} B_0) \ge c_{11} \,|\Delta'_j|^{\frac{d-2}{d}}$ by (\ref{4.81}). Thus, by a similar argument as in the previous paragraph, the simple random walk starting in $B^{(\ell)}$ enters $\bigcup_{j \in \cJ_{B,\ell}} \wt{\Delta}'_j$, and hence $\bigcup_{j \in \cJ_{B,\ell}} \Delta'_j$, and therefore the $\bigcup_{j \in \cJ_{B,\ell}} \bigcup_{ \cC^{\,',\IB}_{j,B,\ell}} B_0 = \bigcup_{\wt{\cC}_{B,\ell}} B_0$ with a probability bounded below by a constant. A similar lower bound naturally holds true when the walk starts in $B(x_B, 8 R_\ell) \supseteq B^{(\ell)}$. This bigger set contains $\bigcup_{\wt{\cC}_{B,\ell}} B_0$, see (\ref{4.85}), and as in the previous paragraph it follows that ${\rm cap} \, (\bigcup_{\wt{\cC}_{B,\ell}} B_0) \ge c\,|B^{(\ell)}|^{\frac{d-2}{d}} = c\,R_\ell^{d-2}$.

\bigskip\n
Finally, when $B \in \cI_p$ is contained in $\wh{A}$ (see (\ref{4.70})) and $\ell \in \cL_B$, we know by (\ref{4.84}) that ${\rm cap} \,(\bigcup_{\wt{\cC}_{B,\ell}} B_0) \ge c_{13} \,|B(x_B,8R_\ell)|^{\frac{d-2}{d}} \ge c\, R^{d-2}_\ell$. This completes the proof of (\ref{4.86}).

\bigskip
We are now ready to start the third (and last) step of the proof of Proposition \ref{prop4.2}. In view of (\ref{4.85}), (\ref{4.86}) we introduce the collection of $B_0$-boxes
\begin{equation}\label{4.87}
\wt{\cC} = \bigcup\limits_{B, \ell \in \cL_B} \wt{\cC}_{B,\ell} \; \mbox{(the union runs over all $B \in \cI_p$, $B \subseteq A$, and $\ell \in \cL_B$)}.
\end{equation}

\n
The collection $\wt{\cC}$ may contain $B_0$-boxes with base points at mutual $| \cdot |_\infty$-distance smaller than $\ov{K} L_0$, and (\ref{4.66}) need not hold for $\wt{\cC}$. We will eventually extract sub-collections $\cC_{B,\ell}$ from $\wt{\cC}$ for each $B \subseteq \cI_p$ contained in $A$ and $\ell \in \cL_B$, so as to fulfill the requirements of Proposition \ref{prop4.2}, see (\ref{4.97}), (\ref{4.98}).

\medskip
With this in mind, we first observe that for each $B_0 \in \wt{\cC}$ and each $\wt{\cC}_{B,\ell}$ there is at most one box of $\wt{\cC}_{B,\ell}$ with base point at $| \cdot |_\infty$-distance  smaller than $H \ov{K} L_0/2$ from the base point of $B_0$. As a result, we find that for large $N$ on $\Omega_{\ve,N}$: 
\begin{equation}\label{4.88}
\begin{array}{l}
\mbox{for each $B_0 \in \wt{\cC}$ the number of boxes in $\wt{\cC}$ having base points at   ~~~~ ~ }
\\
\mbox{$| \cdot |_\infty$-distance smaller than $H \ov{K} L_0 / 2$ from $x_{B_0}$ is at most}
\\
|\{ (B,\ell); \, B \in \cI_p,   \, B \subseteq A, \, \ell \in \cL_B\}| \stackrel{(\ref{4.64})}{\le} J\,10^d \, M^{(p+1) d}  \stackrel{(\ref{4.78})}{=} H/10.
\end{array}
\end{equation}
Then, for each $B_0 \in \wt{\cC}$, we denote by
\begin{equation}\label{4.89}
\begin{array}{l}
\mbox{$\wt{\cC}_{B_0}$ ($\subseteq \wt{\cC}$) the collection of $B_0$-boxes in $\wt{\cC}$ having base points that can}
\\
\mbox{be joined to $x_{B_0}$ by a path of base points of boxes in $\wt{\cC}$ with steps of}
\\
\mbox{$| \cdot |_\infty$-distance at most $\ov{K} L_0$}.
\end{array}
\end{equation}
As a consequence of (\ref{4.88}), we see that for any $B_0 \in \wt{\cC}$
\begin{equation}\label{4.90}
\begin{array}{l}
\mbox{the base points of the $L_0$-boxes in $\wt{\cC}_{B_0}$ lie at $| \cdot |_\infty$-distance at most ~~~~~~~~~  }
\\
\mbox{$J \, 10^d \, M^{(p+1) d} \,\ov{K} L_0 = H \ov{K} L_0 / 10$ from $x_{B_0}$}.
\end{array}
\end{equation}

\n
We then select in each $\wt{\cC}_{B_0}$ a unique box $\wh{B}_0$ in this collection (i.e.~$\wh{B}_0$ is a uniquely chosen representative in $\wt{\cC}_{B_0}$), and we write
\begin{equation}\label{4.91}
\mbox{$\wh{\cC}$ ($\subseteq \wt{\cC}$) for the collection $\{\wh{B}_0$; $B_0 \in \wt{\cC}\}$ (of chosen representatives)}.
\end{equation}

\medskip
\psfrag{B}{$B_0$}
\psfrag{whB0}{$\wh{B}_0$}
\psfrag{whB}{$\wh{B}'_0$}
\psfrag{B0}{$B'_0$}
\begin{center}
\includegraphics[width=10cm]{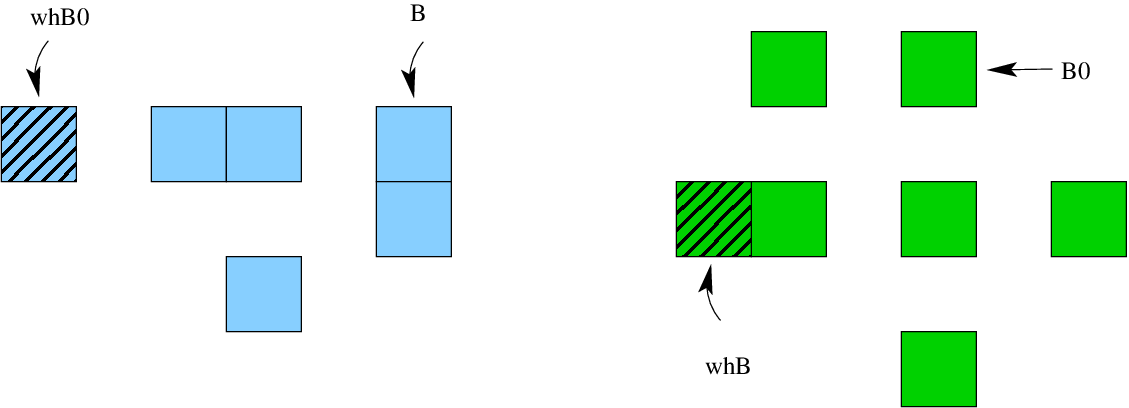}
\end{center}

\medskip
\begin{center}
\begin{tabular}{ll}
Fig.~3: & Two distinct components $\wt{\cC}_{B_0}$, $\wt{\cC}_{B'_0}$ depicted in blue and green colors, \\
& with respective representatives $\wh{B}_0$ and $\wh{B}'_0$ in $\wh{\cC}$ corresponding to the\\
& hatched boxes.
\end{tabular}
\end{center}

\n
We then find that $\wt{\cC} = \bigcup_{B_0 \in \wh{\cC}} \,\wt{\cC}_{B_0}$, that for each $B_0 \in \wt{\cC}$, $\wt{\cC}_{B_0} \cap \wh{\cC} = \{\wh{B}_0\}$, and that
\begin{equation}\label{4.92}
\left\{ \begin{array}{rl}
{\rm i)} & \mbox{for each $B_0 \in \wt{\cC}$, $|x_{B_0} - x_{\wh{B}_0}|_\infty \le H \ov{K} L_0 / 10$,}
\\[1ex]
{\rm ii)} & \mbox{the base points of distinct boxes of $\wh{\cC}$ have mutual $| \cdot |_\infty$-distance}
\\
&\mbox{bigger than $\ov{K} L_0$},
\\[1ex]
{\rm iii)} & \mbox{for any $B \in \cI_p$, $B \subseteq A$, and $\ell \in \cL_B$, the $x_{B_0}$ for $B_0 \in \wt{\cC}_{B,\ell}$ have}
\\
&\mbox{mutual $| \cdot |_\infty$-distance at least $H \ov{K} L_0$ (see (\ref{4.85}))}.
\end{array}\right.
\end{equation}

\n
We then introduce the dimension dependent constant
\begin{equation}\label{4.93}
\mbox{$c_{**} = \sup \{g (0, \wh{y}) / g(0, y)$; for $y, \wh{y} \in \IZ^d$ such that $|\wh{y}|_\infty \ge \fr \;|y|_\infty\} \in [1,\infty)$}.
\end{equation}

\n
The following lemma states a version of the informal principle ``a lower bound on the contraction of a map induces a lower bound on the capacity of the image of a set under this map''. More precisely, one has
\begin{lemma}\label{lem4.3}
If $F$ and $\wh{F}$ are two finite subsets of $\IZ^d$ for which there is a bijection $\wh{f}$: $F \r \wh{F}$ such that $|\wh{f} (x) - \wh{f}(y)|_\infty \ge \frac{1}{2} \;|x-y|_\infty$, for all $x,y$ in $F$, then
\begin{equation}\label{4.94}
{\rm cap} (\wh{F}) \ge {\rm cap} (F) / c_{**}.
\end{equation}
\end{lemma}

\begin{proof}
We consider the equilibrium measure $e_F$ of $F$ and its image $\wh{e}$ under $\wh{f}$. So, $\wh{e}$ is supported by $\wh{F}$ and for each $x \in F$, $\wh{e} \,(\wh{f}(x)) = e_F(x)$. Then, for each $\wh{x} = \wh{f} (x)$ in $\wh{F}$, we have (see (1.2) for notation)
\begin{equation*}
G \, \wh{e} (\wh{x}) = \dsl_{\wh{y} \in \wh{F}} \,g(\wh{x}, \wh{y}) \;\wh{e} (\wh{y}) = \dsl_{y \in F} g\big(\wh{f}(x), \, \wh{f}(y)\big) \,e_F(y) \stackrel{(\ref{4.93})}{\le} c_{**} \dsl_{y \in F} g(x,y) \, e_F(y) = c_{**},
\end{equation*}
since $G e_F = 1$ on $F$. So $G \wh{e} \le c_{**}$ on $\wh{F}$ and $\wh{e}$ is supported by $\wh{F}$, hence we have
\begin{equation*}
{\rm cap}(F) = e_F(F) = \wh{e} (\wh{F}) \stackrel{(\ref{1.5})}{=} \langle \wh{e}, G \,e_{\wh{F}}\rangle = \langle G \,\wh{e}, e_{\wh{F}}\rangle \le c_{**} \,e_{\wh{F}} (\wh{F}) = c_{**} \,{\rm cap}(\wh{F}).
\end{equation*}
The claim (\ref{4.94}) now follows.
\end{proof}

Given $B \in \cI_p$ contained in $A$ and $\ell \in \cL_B$, we can apply Lemma \ref{lem4.3} to $F = \bigcup_{\wt{\cC}_{B,\ell}} B_0$ and $\wh{F} = \bigcup_{\wt{\cC}_{B,\ell}} \wh{B}_0$, taking $\wh{f}$ to be the map from $F$ onto $\wh{F}$ that sends each $B_0 \subseteq F$ to $\wh{B}_0 \subseteq \wh{F}$ through translation by the vector $x_{\wh{B}_0} - x_{B_0}$. So, $|\wh{f}(x) - \wh{f}(y)|_\infty = |x-y|_\infty$ if $x,y$ belong to the same $B_0 \in \wt{\cC}_{B,\ell}$, and by (\ref{4.92}) i) and iii) when $x,y$ belong to distinct $L_0$-boxes in $\wt{\cC}_{B,\ell}$, then $|\wh{f}(x) - \wh{f}(y)|_\infty \stackrel{{\rm i)}}{\ge} |x-y|_\infty - 2 H \ov{K} L_0 / 10 \stackrel{{\rm iii)}}{\ge} \frac{1}{2} \;|x-y|_\infty$.

\medskip
Hence, $\wh{f}$ satisfies the assumptions of Lemma \ref{lem4.3} so that
\begin{equation}\label{4.95}
\begin{array}{l}
\mbox{for any $B \in \cI_p$ contained in $A$ and $\ell \in \cL_B$},
\\
{\rm cap} \,\big(\bigcup_{\wt{\cC}_{B,\ell}} \wh{B}_0\big) \ge c^{-1}_{**} \;{\rm cap} \,\big(\bigcup\limits_{\wt{\cC}_{B,\ell}} B_0\big) \stackrel{(\ref{4.86})}{\ge} c\, R^{d-2}_{\ell}.
\end{array}
\end{equation}

\n
In addition, by (\ref{4.92}) i) and the fact that the $L_0$-boxes in $\wt{\cC}_{B,\ell}$ are contained in $B(x_B, 8R_\ell)$, see (\ref{4.85}), we find that when $N$ is large on $\Omega_{\ve,N}$ (see (\ref{4.6})):
\begin{equation}\label{4.96}
\mbox{for any $B \in \cI_p$ contained in $A$ and $\ell \in \cL_B$, $\bigcup_{\wt{\cC}_{B,\ell}} \wh{B}_0 \subseteq B(x_B, 10 R_\ell)$}.
\end{equation}

\n
In view of (\ref{4.95}) and (\ref{4.92}) ii) we can now apply Lemma \ref{lem1.2}, so that for large $N$ on $\Omega_{\ve,N}$ in (\ref{4.6}),
\begin{equation}\label{4.97}
\left\{ \begin{array}{l}
\mbox{for each $B \in \cI_p$ contained in $A$ and $\ell \in \cL_B$, there is a sub-collection $\cC_{B,\ell}$ of}
\\
\mbox{$\{\wh{B}_0$; $B_0 \in \wt{\cC}_{B,\ell}\}$ of at most $c_8 (R_\ell /L_0)^{d-2}$-boxes, all contained in $B(x_B, 10 R_\ell)$},
\\
\mbox{and such that ${\rm cap} \big(\bigcup_{\cC_{B,\ell}} B_0\big) \ge c\, R_\ell^{d-2}$}.
\end{array}\right.
\end{equation}

\n
As we now explain, for $B, \ell$ as above, the simple random walk starting in $B(x_B,10R_\ell)$ enters $\bigcup_{\cC_{B,\ell}}B_0$ before exiting $B(x_B, 10 R_{\ell -1})$ with a non-degenerate probability. For this purpose we write $g_{B,\ell}(\cdot, \cdot)$ for the Green function of the walk killed outside $B(x_B,10 R_{\ell - 1})$. We know by (\ref{1.4}) that $g_{B,\ell}(x,y) \ge c_\Delta \, g(x,y)$ when $x,y \in B(x_B, 10 R_\ell)$. In addition, setting $F = \bigcup_{\cC_{B,\ell}} B_0$ ($\subseteq B(x_B,10R_\ell)$), the equilibrium measure of $F$ for the walk killed outside $B(x_B, 10 R_{\ell - 1})$ is supported by $F$ and dominates $e_F$, see below (\ref{1.5}). It now follows that
\begin{equation}\label{4.98}
\begin{split}
\mbox{for $x \in B(x_B, 10R_\ell), \; P_x\big[H_{\bigcup_{\cC_{B,\ell}}B_0} < T_{B(x_B,10 R_{\ell - 1})}\big]$} & \ge c_\Delta \;P_x\big[H_{\bigcup_{\cC_{B,\ell}}B_0} < \infty\big]
\\
&\!\!\!\!\! \stackrel{(\ref{1.5}),(\ref{4.97})}{\ge} c_9 \; (> 0).
\end{split}
\end{equation}

\n
Since all $\cC_{B,\ell}$ are contained in $\wh{\cC}$, we find by (\ref{4.92}) ii) that the base points of the $B_0$-boxes in $\cC \stackrel{\rm def}{=} \bigcup_{B \in \cI_p, B \subseteq A, \ell \in \cL_B} \cC_{B,\ell}$ have mutual $|\, \cdot \,|_\infty$-distance at least $\ov{K} L_0$. Combined with (\ref{4.97}) - (\ref{4.98}), we see that the collections $\cC_{B,\ell}$ of $(\alpha,\beta,\gamma)$-good $B_0$-boxes with $N_u(D_0) \ge \beta\,{\rm cap}(D_0)$ and intersecting $[-3N,3N]^d$ fulfill (\ref{4.66}) - (\ref{4.69}), and Proposition \ref{prop4.2} is proved. \hfill $\square$

\medskip
We proceed with the proof of Theorem \ref{theo4.1}. We will now combine (\ref{4.65}) and Proposition \ref{prop4.2}. Thus, for large $N$ on the event $\Omega_{\ve,N}$ in (\ref{4.6}), we see that except maybe on a set of at most $\frac{\ve}{5} \, |D_N|$ points in the bubble set $Bub$, the simple random walk starting at $x \in Bub$ enters the set $A$ in (\ref{4.64}) with probability at least $1 - (1-a)/10$ and once in $A$ enters $\bigcup_{\cC} B_0$ with probability at least $1 - (1- c_9)^J$ (using the strong Markov property at the successive times of exit of $B(x_B, 10 R_{\ell - 1})$, $1 \le \ell \le p$, and (\ref{4.69}), if $B$ stands for the box of $\cI_p$ contained in $A$ where the walk first enters $A$).

\medskip
We can now choose $J$ as a function of $a \in (0,1)$ via
\begin{equation}\label{4.99}
\mbox{$J(a) \ge 1$ is such that $(1-a) / 10 + (1 - c_9)^J < 1-a$}.
\end{equation}
We have now obtained that for large $N$ on the event $\Omega_{\ve,N}$ in (\ref{4.6})
\begin{equation}\label{4.100}
\left\{ \begin{array}{l}
\mbox{for each $B \in \cI_p$, $B \subseteq [-4N,4N]^d$ one can select a subset $\cL_B$ of $\{1,\dots , p\}$}
\\
\mbox{with $0$ or $J$ elements, and for each $B$ and $\ell \in \cL_B$ a collection $\cC_{B,\ell}$}
\\
\mbox{(possibly empty) of at most $c_8(R_\ell / L_0)^{d-2} (\alpha,\beta,\gamma)$-good $B_0$-boxes with}
\\
\mbox{$N_u(D_0) \ge \beta \, {\rm cap} (D_0)$, contained in $B(x_B, 10R_\ell) \cap (-4N,4N)^d$, so that}
\\
\mbox{the base points of all these boxes as $B$ and $\ell \in \cL_B$ vary, keep a mutual}
\\
\mbox{distance at least $\ov{K} L_0$, and so that denoting  their union by}
\\
\mbox{$C = \bigcup_{B, \ell \in \cL_B} \bigcup_{\cC_{B,\ell}} B_0$, one has $|\{x\in Bub$; $P_x [H_C < \infty] < a\}| \le \frac{\ve}{5} \;|D_N|$}.
\end{array}\right.
\end{equation}

\medskip\n
Thus, for large $N$ we set $C_\omega = \phi$, on $\cB^c_N \backslash \Omega_{\ve,N}$, and on $\Omega_{\ve,N}$ construct a measurable choice $\cL_{\cB,\omega}$ of $\cL_B$ for $B \in \cI_p$, $B \subseteq [-4N,4N]^d$ and $\cC_{B,\ell,\omega}$ of $\cC_{B,\ell}$ when $\ell \in \cL_{B,\omega}$, and set $C_\omega = \bigcup_{B,\ell \in \cL_{B,\omega}} \bigcup_{\cC_{B,\ell,\omega}} B_0$. The conditions (\ref{4.5}) i), ii), v) are fulfilled. Concerning (\ref{4.5}) iv), note that the $2 \ov{K} L_1$-neighborhood of $C_\omega$ has volume at most $c\, \ov{K}\,\!^d L_1^d (N | L_0)^{d-2} \times J \times c\, M^{(p+1) d} = o(N^d)$, by (\ref{1.9}), (\ref{1.10}), so that (\ref{4.5}) iv) holds as well.

\medskip
There remains to check (\ref{4.5}) iii). For each $B \in \cI_p$ included in $ [-4N,4N]^d$, there are at most $(p^J + 1)$ possibilities for $\cL_{B,\omega}$, and for each $\cL_{B,\omega}$, which is not empty, and $\ell \in \cL_{B,\omega}$, one has at most $\exp\{c\, (R_\ell /L_0)^{d-2} \log (R_\ell / L_0)\} \le \exp\{c \,(N /L_0)^{d-2} \log (N /L_0)\}$ possible choices for the collection $\cC_{B,\ell,\omega}$ of the at most $c_8 (R_\ell / L_0)^{d-2}$ boxes $B_0$ in $B(x_B, 10R_\ell)$, and hence in total at most $(p^J + 1)^{c M^{(p+1)d}} \exp\{ c\, J \,M^{(p+1)d} (N/L_0)^{d-2} \log (N/L_0)\} = \exp\{o(N^{d-2})\}$ possibilities. This shows that for large $N$, $C_\omega$ varies in a set $\cS_N$ with $\exp\{o(N^{d-2})\}$ elements and (\ref{4.5}) iii) holds as well. This completes the proof of Theorem \ref{theo4.1}. \hfill $\square$

\section{The asymptotic upper bound}
\setcounter{equation}{0}

In this section we derive the main asymptotic upper bound on the exponential rate of decay of the probability of an excess of disconnected points corresponding to the event $\cA_N = \{|D_N \, \backslash \, \cC^u_{2N}| \ge \nu\, |D_N|\}$ from (\ref{0.9}). With Theorem \ref{theo4.1} now available, most of the remaining task has already been carried out in Section 4 of \cite{Szni21b}. Importantly, if as expected the identity $\ov{u} = u_*$ holds, the asymptotic upper bound of Theorem \ref{theo5.1} below actually governs the exponential decay of $\IP[\cA_N]$, see Remark \ref{rem5.2} 1). An application of Theorem \ref{theo5.1} to the simple random walk is given in Corollary \ref{cor5.3}.

\medskip
We assume that (see (\ref{1.16})):
\begin{equation}\label{5.1}
0 < u < \ov{u},
\end{equation}
and we define the function (see (\ref{0.2}) for notation):
\begin{equation}\label{5.2}
\wh{\theta} (v) = \left\{ \begin{array}{l}
\theta_0(v), \; \mbox{for $0 \le v < \ov{u}$},
\\[0.5ex]
1, \; \mbox{for $v \ge \ov{u}$},
\end{array}\right.
\end{equation}
and for $\nu \in [\theta_0(u), 1)$, we set
\begin{equation}\label{5.3}
\wh{J}_{u,\nu} = \min \Big\{ \mbox{\f $\dis\frac{1}{2d}$} \;\dis\int_{\IR^d} | \nabla \varphi|^2 dz; \; \varphi \ge 0, \varphi \in D^1(\IR^d), \;\strokedint_D \wh{\theta} \big((\sqrt{u} + \varphi)^2\big) \,dz \ge \nu\Big\}.
\end{equation}

\n
The existence of a minimizer for (\ref{5.3}) is established by the same argument as in the case of $\ov{J}_{u,\nu}$, see Theorem 2 of \cite{Szni21}, which corresponds to the function $\ov{\theta}_0$ ($\le \wh{\theta}$) in place of $\wh{\theta}$, see (\ref{0.11}). In addition for $u$ and $\nu$ as above, since $\ov{\theta}_0 \le \wh{\theta}$, one has
\begin{equation}\label{5.4}
\wh{J}_{u,\nu} \le \ov{J}_{u,\nu} \quad \mbox{(and these quantities coincide if $\ov{u} = u_*$)}.
\end{equation}

\begin{theorem}\label{theo5.1}
Consider $u$ as in (\ref{5.1}) and $\nu \in [\theta_0(u),1)$, then
\begin{equation}\label{5.5}
\limsup\limits_N \; \mbox{\f $\dis\frac{1}{N^{d-2}}$} \; \log \IP[\cA_N] \le - \wh{J}_{u,\nu} \quad \mbox{(see (\ref{0.9}) and (\ref{5.3}) for notation)}.
\end{equation}
\end{theorem}

\begin{proof}
Consider $a \in (0,1)$ and set $\theta^{(a)}$ and $J^{(a)}_{u,\nu}$ as in (\ref{5.2}), (\ref{5.3}) where $\ov{u}$ is replaced by $(\sqrt{u} + a(\sqrt{\ov{u}} - \sqrt{u}))^2 < \ov{u}$ in the definition of $\theta^{(a)}$, and $\wh{\theta}$ is replaced by $\theta^{(a)}$ ($\ge \wh{\theta}$) in the definition of $J^{(a)}_{u,\nu}$. Then, combining Theorem \ref{theo4.1} of the previous section and Theorem 4.3 of \cite{Szni21b}, we find that for any $a \in (0,1)$,
\begin{equation}\label{5.6}
\limsup\limits_N \; \mbox{\f $\dis\frac{1}{N^{d-2}}$} \; \log \IP[\cA_N] \le - J^{(a)}_{u,\nu} .
\end{equation}

\n
The functions $\theta^{(a)}$ decrease as $a$ increases and $\lim_{a \r 1} \theta^{(a)} (v) = \wh{\theta}(v)$ for all $v \ge 0$. It follows that $J^{(a)}_{u,\nu} \le \wh{J}_{u,\nu}$ and $J^{(a)}_{u,\nu}$ is a non-decreasing function of $a$. As we now explain, $\lim_{a \r 1} J^{(a)}_{u,\nu} = \wh{J}_{u,\nu}$. The argument is similar to (4.43), (4.44) of \cite{Szni21b}. Namely, we consider an increasing sequence $a_n$ in $(0,1)$ tending to $1$, and for each $n \ge 1$ a minimizer $\varphi_n$ for $J^{(a_n)}_{u,\nu}$. By Theorem 8.6, p.~208 and Corollary 8.7, p.~212 of \cite{LiebLoss01}, we can extract a subsequence $\varphi_{n_\ell}$, $\ell \ge 1$, converging in $L^2_{\rm loc}(\IR^d)$ and a.e.~to $\varphi$ belonging to $D^1(\IR^d)$ such that
\begin{equation}\label{5.7}
 \mbox{\f $\dis\frac{1}{2d}$} \dis\int_{\IR^d} | \nabla \varphi |^2 dz \le \liminf\limits_\ell  \;  \mbox{\f $\dis\frac{1}{2d}$}  \dis\int_{\IR^d} \; |\nabla \varphi_{n_\ell}|^2 dz = \lim\limits_n  J^{(a_n)}_{u,\nu} .
\end{equation}
Moreover, we have $\wh{\theta} ((\sqrt{u} + \varphi)^2) \ge \limsup_\ell \theta^{(a_{n_\ell})} ((\sqrt{u} + \varphi_{n_\ell})^2)$ a.e, so that
\begin{equation}\label{5.8}
\begin{array}{l}
\dis\strokedint_D \wh{\theta} \big((\sqrt{u} + \varphi)^2\big) \,dz \ge \dis\strokedint_D \limsup\limits_\ell \theta^{(a_{n_\ell})} ((\sqrt{u} + \varphi_{n_\ell})^2) dz
\\[1ex]
\hspace{2.8cm}\, \stackrel{\rm reverse\,Fatou}{\ge} \limsup\limits_\ell  \dis\strokedint_D \theta^{(a_{n_\ell})}  ((\sqrt{u} + \varphi_{n_\ell})^2) dz \ge \nu.
\end{array}
\end{equation}
This shows that $\wh{J}_{u,\nu} \le \lim_n J^{(a_n)}_{u,\nu} = \lim_{a \r 1} J^{(a)}_{u,\nu} \le \wh{J}_{u,\nu}$, so that
\begin{equation}\label{5.9}
\lim\limits_{a \r 1} \; J^{(a)}_{u,\nu} = \wh{J}_{u,\nu}.
\end{equation}
Letting $a$ tend to $1$ in (\ref{5.6}) thus yields (\ref{5.5}) and Theorem \ref{theo5.1} is proved.
\end{proof}

\begin{remark}\label{rem5.2} \rm 1) If the identity $\ov{u} = u_*$ holds (this is the object of active research, and the corresponding identity in the closely related model of the level-set percolation of the Gaussian free field has been established in \cite{DumiGoswRodrSeve20}), then Theorem \ref{theo5.1} and the lower bound (\ref{0.10}) from \cite{Szni21a} and \cite{Szni21} show that for $0 < u < u_*$ and $\nu \in  [ \theta_0 (u),1)$ one has in the notation of (\ref{0.9})
\begin{equation}\label{5.10}
\lim\limits_N \; \mbox{\f $\dis\frac{1}{N^{d-2}}$} \; \log \IP[\cA_N] = \lim\limits_N \; \mbox{\f $\dis\frac{1}{N^{d-2}}$} \; \log \IP[\cA^0_N]  = - \ov{J}_{u,\nu}.
\end{equation}

\n
2) The event $\cA_N$ in Theorem \ref{theo5.1} pertains to an excessive fraction of points in $D_N$ disconnected by $\cI^u$ from $S_{2N}$. One can replace $2$ by an arbitrary integer $m \ge  1$ and instead consider the events $\{|D_N \,\backslash \, \cC^u_{mN}| \ge \nu\,|D_N|\}$, which correspond to an excessive fraction of points in  $D_N$ that are disconnected by $\cI^u$ from $S_{mN}$ (these events are non-decreasing in~$m$). The proof of Theorem \ref{theo5.1} (based on Section 4 and on \cite{Szni21b}) can straightforwardly be adapted to show that for arbitrary $m \ge 1$, $0 < u < \ov{u}$, and $\nu \in [\theta_0(u) ,1)$ one has
\begin{equation}\label{5.11new}
\limsup\limits_N \; \mbox{\f $\dis\frac{1}{N^{d-2}}$}\; \log \IP [|D_N \, \backslash \, \cC^u_{mN}| \ge \nu\, |D_N| ] \le - \wh{J}_{u,\nu}.
\end{equation}

\n
In particular, if as expected $\ov{u}$ and $u_*$ coincide, this upper bound combined with the lower bound (\ref{0.10}), proves that for any $m \ge 1$, $0 < u < u_*$ and $\nu \in [\theta_0(u),1)$
\begin{equation}\label{5.12new}
\lim\limits_N \; \mbox{\f $\dis\frac{1}{N^{d-2}}$}\; \log \IP [|D_N \, \backslash \, \cC^u_{mN}| \ge \nu\, |D_N| ]= - \ov{J}_{u,\nu}.
\end{equation}
(this extends (\ref{5.10})).

\medskip
However, it remains open whether for the larger events $\{|D_N \, \backslash \, \cC^u_\infty| \ge \nu\, |D_N|\}$ corresponding to an excessive fraction of points in $D_N$ disconnected by $\cI^u$ from infinity, one has a similar asymptotics. Namely, is it the case that for all $0 < u < u_*$ and $\nu \in [\theta_0(u),1)$
\begin{equation}\label{5.13new}
\lim\limits_N \; \mbox{\f $\dis\frac{1}{N^{d-2}}$}\; \log \IP [|D_N \, \backslash \, \cC^u_{\infty}| \ge \nu\, |D_N| ]= - \ov{J}_{u,\nu}\,?
\end{equation}

\n
In the context of the Wulff droplet for super-critical Bernoulli percolation, we refer to Theorem 2.12 of \cite{Cerf00} for a corresponding result (the leading rate of decay of the asymptotics is in that case $N^{d-1}$, i.e. ``surface like'', and not $N^{d-2}$, i.e.~``capacity like'', as here).

\bigskip\n
3)  It is a natural question whether for $\nu$ close to $1$ the minimizers $\varphi$ for $\ov{J}_{u,\nu}$ in (\ref{0.11}) do reach the maximum possible value $\sqrt{u}_* - \sqrt{u}$ on a set of positive Lebesgue measure. As explained below (\ref{0.14}) this occurrence could reflect the presence of droplets secluded by the interlacements and contributing to the excess fraction of disconnected points when $\cA_N$ happens. If $\theta_0$ is discontinuous at $u_*$ (a not very plausible possibility), the minimizers for $\ov{J}_{u,\nu}$ are easily seen to reach the value $\sqrt{u}_* - \sqrt{u}$ on a set of positive measure when $\nu$ is close to $1$, see Remark 2 of \cite{Szni21}. But otherwise the situation is unclear, for the behavior of $\theta_0$ close to $u_*$ is very poorly understood. The same is true in the case of the percolation function $\theta^G_0$ for the level-set percolation of the Gaussian free field on $\IZ^d$, see above (\ref{0.15a}). Interestingly, in the case of the cable graph on $\IZ^d$ the corresponding function $\wt{\theta}^G_0$ is explicit. The critical level is $0$ and $\wt{\theta}^G_0(h) = 2 \Phi (h \wedge  0)$, for $h \in \IR$, where $\Phi$ denotes the distribution function of a centered Gaussian variable with variance $g(0,0)$, see Corollary 2.1 of \cite{DrewPrevRodr21}. However, in the case of $\IZ^d$, $d = 3$, the simulations in Figure 4 of \cite{MariLebo06} suggest a behavior of $\theta^G_0$ close to the critical level $h_*$ different from that of $\wt{\theta}^G_0$ near the critical level $0$. 

\medskip
Coming back to the original question whether for $\nu$ close to $1$ the minimizers for $\ov{J}_{u,\nu}$ reach the maximal value $\sqrt{u}_* - \sqrt{u}$ on a set of positive measure, let us mention that the question has a similar flavor to the problem concerning the existence of {\it dead core solutions} for semilinear equations $\frac{1}{2d} \,\Delta v = f(v)$, i.e.~non-negative solutions in a bounded domain $U$ of $\IR^d$, which vanish on a relatively compact open subset $V$ in $U$ and are positive in $U \backslash \ov{V}$, see \cite{PuccSerr04}. Quite informally, assuming $\theta_0$ to be $C^1$ to simplify the discussion, the link goes via the consideration of $v = \sqrt{u}_* - \sqrt{u} - \varphi$, where $\varphi$ minimizer for $\ov{J}_{u.\nu}$ satisfies an Euler-Lagrange equation $-\frac{1}{2} \, \Delta \varphi =  \lambda \eta ' (\varphi) \,1_D$, with $\lambda > 0$ a Lagrange multiplier and $\eta(b) = \theta_0 ((\sqrt{u} + b)^2)$, see for instance Lemma 5 of \cite{Szni21}. Under suitable assumptions on $f$ the reference \cite{PuccSerr04} provides an integral criterion, which characterizes the existence of dead core solutions, see Theorems 1.1, 7.2 and 7.3 of  \cite{PuccSerr04}. Brought into our context, these results raise the question: does the convergence of $\int^\delta_0 \{1 - \eta(\sqrt{u}_* - \sqrt{u} - a)\}^{-1/2} \,da$ for small $\delta > 0$ ensure that for $\nu$ close to $1$ the minimizers for $\ov{J}_{u,\nu}$ take the value $\sqrt{u}_* - \sqrt{u}$ on a set of positive Lebesgue measure? \hfill $\square$
\end{remark}

\medskip
The above Theorem \ref{theo5.1} also has an immediate application to a similar upper bound, where the simple random walk replaces the random interlacements. Informally, this corresponds to taking the singular limit $u \r 0$ in (\ref{5.5}). Specifically, we denote by $\cI$ the set of points in $\IZ^d$ visited by the simple random walk, and by $\cC_{2N}$ the connected component of $S_{2N}$ in $(\IZ^d \backslash \cI) \cup S_{2N}$. One has
\begin{corollary}\label{cor5.3}
For $\nu \in [0,1)$ and $x \in \IZ^d$,
\begin{align}
& \limsup\limits_N \; \mbox{\f $\dis\frac{1}{N^{d-2}}$} \; \log P_x [|D_N \,\backslash\,\cC_{2N}| \ge \nu\,|D_N|] \le - \wh{J}_\nu, \; \mbox{where} \label{5.11}
\\[1ex]
& \wh{J}_\nu = \min \Big\{\mbox{\f $\dis\frac{1}{2d}$} \; \dis\int_{\IR^d} |\nabla \varphi |^2\, dz; \; \varphi \ge 0, \varphi \in D^1(\IR^d), \; \dis\strokedint_D \wh{\theta} (\varphi^2) \,dz \ge \nu\Big\}. \label{5.12}
\end{align}
\end{corollary}

The existence of a minimizer for (\ref{5.12}) is established by a similar argument as in the case of $\ov{J}_{u,\nu}$ in Theorem 2 of \cite{Szni21}. Also, if as expected, the identity $\ov{u} = u_*$ holds, then $\wh{\theta}$ coincides with $\ov{\theta}_0$.

\begin{proof}
Without loss of generality, we can assume that $\nu > 0$ and choose $u \in (0,\ov{u})$ such that $\theta_0(u) < \nu$. As in Corollary 7.3 of \cite{Szni15} or Corollary 6.4 of \cite{Szni17}, one has a coupling $\ov{P}$ of $\cI^u$ under $\IP[\cdot \,|x \in \cI^u]$ and of $\cI$ under $P_x$, so that $\ov{P}$-a.s., $\cI \subseteq \cI^u$. Then $\ov{P}$-a.s., the points in $D_N$ disconnected by $\cI$ from $S_{2N}$ are also points disconnected by $\cI^u$ from $S_{2N}$, so that $\ov{P}$-a.s., $|D_N \,\backslash\, \cC^u_{2N}| \ge |D_N \,\backslash\, \cC_{2N}|$. Hence, we have
\begin{equation}\label{5.13}
\begin{split}
P_x [|D_N \,\backslash\, \cC_{2N}| \ge \nu \,| D_N|] & \le \IP [|D_N\,\backslash \, \cC^u_{2N}| \ge \nu\,|D_N|] \, / \, \IP[x \in \cI^u]
\\
& = \IP[\cA_N] \, / \, (1 - e^{-u/g(0,0)}).
\end{split}
\end{equation}

\n
As a result of Theorem \ref{theo5.1} it follows that
\begin{equation}\label{5.14}
\limsup\limits_N \; \mbox{\f $\dis\frac{1}{N^{d-2}}$} \; \log P_x [|D_N\, \backslash \,\cC_{2N}| \ge \nu\,|D_N|] \le - \wh{J}_{u,\nu}.
\end{equation}

\n
By direct inspection $\wh{J}_{u,\nu}$ is non-increasing in $u$ and $\wh{J}_{u,\nu} \le \wh{J}_{\nu}$. As we now explain
\begin{equation}\label{5.15}
\lim\limits_{u \r 0} \;\wh{J}_{u,\nu} = \wh{J}_\nu .
\end{equation}

\n
The argument is similar to the proof of (\ref{5.9}). We consider a sequence $u_n$ in $(0,\ov{u})$ decreasing to $0$ with $\theta_0(u_n) < \nu$ for all $n$, and let $\varphi_n$ be a minimizer for $\wh{J}_{u_n,\nu}$. Once again, by Theorem 8.6 and Corollary 8.7 of \cite{LiebLoss01}, we can extract a subsequence $\varphi_{n_\ell}$, $\ell \ge 1$, converging in $L^2_{{\rm loc}}(\IR^d)$ and a.e. to $\varphi$ element of $D^1(\IR^d)$ such that
\begin{equation*}
\begin{array}{l}
\hspace{-3.7ex} \mbox{\f $\dis\frac{1}{2d}$} \; \dis\int_{\IR^d} \;|\nabla \varphi |^2\, dz \le \liminf\limits_\ell \; \mbox{\f $\dis\frac{1}{2d}$} \; \dis\int_{\IR^d} |\nabla \varphi_{n_\ell}|^2 \,dz = \lim\limits_n \;\wh{J}_{u_n,\nu}, \;\mbox{and}
\\[2ex]
\dis\strokedint_D \wh{\theta} (\varphi^2) \,dz \ge \dis\strokedint_D \limsup\limits_\ell \; \wh{\theta} \,\big((\sqrt{u}_{n_\ell} + \varphi_{n_\ell})^2\big) \,dz
\\[2ex]
\qquad \quad \;\;\stackrel{\rm reverse\,Fatou}{\ge} \limsup\limits_\ell \; \dis\strokedint_D \wh{\theta} \,  \big((\sqrt{u}_{n_\ell} + \varphi_{n_\ell})^2\big) \,dz \ge \nu.
\end{array}
\end{equation*}

\n
This shows that $\wh{J}_\nu \le \lim_n \; \wh{J}_{u_n,\nu}$ and hence completes the proof of (\ref{5.15}). The claim of Corollary \ref{cor5.3} now follows by letting $u$ tend to $0$ in (\ref{5.14}).
\end{proof}

\begin{remark}\label{rem5.4} \rm It remains an open question whether a matching asymptotic lower bound for (\ref{5.11}) holds as well. The consideration of {\it tilted random walks} as in \cite{Li17} (which provides the ``right'' asymptotic lower bound for the disconnection of $D_N$ by $\cI$), and similar ideas as in Section 4 and Remark 6.6 of \cite{Szni21a} might be helpful. \hfill $\square$
\end{remark}
\appendix
\section{Appendix: Resonance sets and $I$-families}
\setcounter{equation}{0}

In this appendix we recall some results concerning resonance sets and $I$-families developed in \cite{NitzSzni20} and \cite{ChiaNitz}, and we sketch the proof of Proposition \ref{prop3.1}. 

\medskip
First some notation. For $x \in \IZ^d$ and $r \ge 0$ integer we write
\begin{equation}\label{A.1}
\begin{array}{l}
\mbox{$m_{x,r}$ for the normalized counting measure on $B(x,r) \cap \IZ^d$, and $\langle f \rangle_{B(x,r)}$}
\\
\mbox{for the integral of $f$ with respect to $m_{x,r}$}.
\end{array}
\end{equation}

\n
We consider an integer
\begin{equation}\label{A.2}
J \ge 1,
\end{equation}

\n
and (as in (\ref{4.20}) of  \cite{ChiaNitz}, with $\delta = 200 J$) a length scale $r_{\min}(J)$ such that
\begin{equation}\label{A.3}
r_{\min} (J) \ge 8 \times 200 J
\end{equation}

\n
and such that for suitable constants $\check{c}_0 (J)$, $\check{c}_1(J)$, $c'(J) \in (0,1)$, for all integers $r \ge r_{\min}$
\begin{equation}\label{A.4}
\left\{ \begin{array}{rl}
{\rm i)} & m_{x,r} (\{y; |y - z|_\infty \ge (1 - \check{c}_0 (J)) \,r\}) \le (4 \times 200 J)^{-1},
\\[2ex]
{\rm ii)} & \mbox{for all $y, y' \in B(x, (1 - \check{c}_0(J))\,r), \; q_{t,B(x,r)} (y,y') \ge c'(J) \, t^{-d/2}$}
\\[0.5ex]
&\hspace{6cm}\! \mbox{for all $t \in [\check{c}_1(J) \,r^2, r^2]$},
\end{array}\right.
\end{equation}

\n
where $q_{t,B(x,r)} (\cdot,\cdot)$ stands for the transition kernel of the simple random walk with unit jump rate killed outside $B(x,r)$.

\medskip
We assume that $N$ is large so that in the notation of (\ref{1.10})
\begin{equation}\label{A.5}
L_1 \ge r_{\min}(J) \vee 6.
\end{equation}
We then consider
\begin{equation}\label{A.6}
\mbox{$U_0$ a finite non-empty subset of $\IZ^d$, $U_1 = \IZ^d \backslash U_0$},
\end{equation}
and we define the density functions for $x \in \IZ^d$ and $m \ge 0$ integer,
\begin{equation}\label{A.7}
\left\{ \begin{array}{l}
\sigma_m(x) = m_{x,2^m L_1} (U) = \mbox{\f $\dis\frac{|B(x,2^m L_1) \cap U_1|}{|B(x,2^m L_1)|}$}\;, \; \mbox{and}
\\[1ex]
\wt{\sigma}_m(x)  = \sigma_{m+2} (x).
\end{array}\right.
\end{equation}

\n
The next two lemmas are straightforward adaptations of Lemmas 1.1 and 1.2 of \cite{NitzSzni20} in the $\IR^d$-case and of Lemmas 4.3 and 4.4 of \cite{ChiaNitz} in the discrete case. We have:
\begin{lemma}\label{lemA.1}
\begin{align}
& \mbox{For $m \ge 0$, $x,y \in \IZ^d$, $|\sigma_m(x+y) - \sigma_m(x)| \le \mbox{\f $\dis\frac{1}{2^m L_1}$} \;|y|_1$}. \label{A.8}
\\[1ex]
& \mbox{For $0 \le m' < m$, $x \in \IZ^d$, $|\sigma_m(x) - \langle \sigma_{m'}\rangle_{B(x,2^m L_1)}| \le d\, 2^{d-1} \; 2^{m' - m}$}.\label{A.9}
\end{align}
\end{lemma}

Further, we have (and the $\delta$ below is unrelated to (\ref{2.7}) but follows the notation of \cite{NitzSzni20}, \cite{ChiaNitz}):
\begin{lemma}\label{lemA.2}
For $x \in \IZ^d$, $0 \le m' < m$, setting $\beta' = \langle\sigma_{m'}\rangle_{B(x,2^m L_1)}$, then for all $0 \le \delta \le \beta' \wedge (1- \beta')$ at least one of i) or ii) below holds true:
\begin{equation}\label{A.10}
\left\{ \begin{array}{rl}
{\rm i)} & m_{x,2^m L_1} \;(\sigma_{m'} > \beta' + \delta) \ge \mbox{\f $\dis\frac{\delta}{2}$} \; \mbox{and} \;\; m_{x,2^m L_1} \;(\sigma_{m'} < \beta' - \delta) \ge  \mbox{\f $\dis\frac{\delta}{2}$} ,
 \\[1ex]
{\rm ii)} & m_{x,2^m L_1}, \;(\beta' - \delta \le \sigma_{m'} \le \beta' + \delta) \ge  \mbox{\f $\dis\frac{1}{4}$} -  \mbox{\f $\dis\frac{\delta}{2}$} .
\end{array}\right.
\end{equation}
\end{lemma}

The next proposition corresponds to Proposition 4.5 of \cite{ChiaNitz} in our set-up (with the choice $\delta = 1/(200J)$), see also Proposition 1.3 of \cite{NitzSzni20}.
\begin{proposition}\label{propA.3}
Recall (\ref{A.5}), (\ref{A.6}). For $x \in \IZ^d$, $0 \le m' < m$, setting $\beta' = \langle \sigma_{m'} \rangle_{B(x,2^m L_1)}$, assume that $(200J)^{-1} \le \beta' \wedge (1- \beta') \wedge \frac{1}{4}$, then
\begin{equation}\label{A.11}
P_x \big[H_{\{\sigma_{m'} \in [\beta' - \frac{1}{200J}, \beta' + \frac{1}{200J}]\}} < \tau_{2^m L_1}\big] \ge c(J),
\end{equation}
where for $r \ge 0$,
\begin{equation}\label{A.12}
\tau_r = \inf \{ s \ge 0; \;|X_s - X_0|_\infty \ge r\}.
\end{equation}
\end{proposition}

We then define (as in (1.27) of \cite{NitzSzni20} or (4.26) of \cite{ChiaNitz})
\begin{equation}\label{A.13}
L(J) = \min \{L \ge 5; \,d \,2^{d-1} \, 2^{-L} \le (200 J)^{-1} \}.
\end{equation}

\n
We look at well-separated spatial scales $2^{m_0} L_1 >  2^{m_1} L_1 > \dots > 2^{m_J} L_1$, where
\begin{equation}\label{A.14}
m_j \ge L(J) + m_{j+1}, \; \mbox{for $0 \le j < J$},
\end{equation}

\n
and consider the increasing sequence of intervals (see (4.28) of  \cite{ChiaNitz}):
\begin{equation}\label{A.15}
I_j = \Big[ \mbox{\f $\dis\frac{1}{2}$}  - \mbox{\f $\dis\frac{j+1}{100J}$} , \;  \mbox{\f $\dis\frac{1}{2}$}  +  \mbox{\f $\dis\frac{j+1}{100J}$} \Big], \; 0 \le j \le J,
\end{equation}

\n
together with the non-decreasing sequence of stopping times
\begin{equation}\label{A.16}
\mbox{$\gamma_0 = H_{\{\sigma_{m_0} \in I_0\}}$,  and $\gamma_{j+1} = \gamma_j  + H_{\{\sigma_{m_{j+1}} \in I_{j+1}\}} \circ \theta_{\gamma_j}$,  for $0 \le j < J$}
\end{equation}

\n
(we refer to the beginning of Section 1 for notation).

\medskip
The next proposition corresponds to Proposition 4.6 of  \cite{ChiaNitz} in the set-up of random walks among random conductances, and to Proposition 1.4 of  \cite{NitzSzni20} in the Brownian case.

\begin{proposition}\label{propA.4}
Recall (\ref{A.5}), (\ref{A.6}). Assume that $J \ge 1$ and that the non-negative integers $m_j$, $0 \le j \le J$, satisfy (\ref{A.14}). Let $\cE$ in the notation of (\ref{A.16}), (\ref{A.12}) denote the event
\begin{equation}\label{A.17}
\cE = \{\gamma_0 = 0\} \cap \bigcap\limits_{0 \le j < J} \{\gamma_{j+1} < \gamma_j + \tau_{2^{m_j} L_1} \circ \theta_{\gamma_j}\}.
\end{equation}

\n
Then, for any $x \in \IZ^d$ such that
\begin{equation}\label{A.18}
\sigma_{m_0} (x) \in \Big[ \mbox{\f $\dis\frac{1}{2}$} -  \mbox{\f $\dis\frac{1}{L_1}$}, \;  \mbox{\f $\dis\frac{1}{2}$} +  \mbox{\f $\dis\frac{1}{L_1}$}\Big],
\end{equation}
one has
\begin{equation}\label{A.19}
P_x[\cE] \ge \check{c}_2(J).
\end{equation}
In addition, on the event $\cE$, we have (see (\ref{A.7}) for notation)
\begin{align}
&\sup\{|X_s - X_{\gamma_j}|_\infty; \; \gamma_j \le s \le \gamma_J\} \le \mbox{\f $\dis\frac{3}{2}$} \;2^{m_j} L_1, \; \mbox{for $0 \le j < J$, and} \label{A.20}
\\[1ex]
&\wt{\sigma}_{m_j} (X_{\gamma_J}) \in [\wt{\alpha}, 1 - \wt{\alpha}], \; \mbox{for all $0 \le j \le J$, where $\wt{\alpha} =  \mbox{\f $\dis\frac{1}{3}$} \;4^{-d}$}. \label{A.21}
\end{align}
\end{proposition}

Let us mention that the constraint on $x$ stated in (\ref{A.18}) is more restrictive (i.e.~less general) than the constraint $\sigma_{m_0} (x) \in [\frac{1}{2} - r_{\min}(J)^{-1}$, $\frac{1}{2} + r_{\min} (J)^{-1}]$ corresponding to Proposition 4.6 of  \cite{ChiaNitz}, but more convenient in the present context.

\medskip
We then turn to the crucial notion of {\it $I$-family} that will provide the tool which permits to bound the probability that the walk avoids the resonance set, when starting at a point around which the local density of $\cU_0$ is at least $\frac{1}{2}$ on many well-separated scales $2^m L_1$.

\medskip
Similar to (\ref{2.11}) of \cite{NitzSzni20} and (A.2) of \cite{ChiaNitz}, we consider
\begin{equation}\label{A.22}
\mbox{$J \ge 1$, $I \ge 1$, as well as $L \ge L(J)$ and an integer $h \ge 0$}.
\end{equation}

\n
An $I$-family consists of stopping times $S_i$, $0 \le i \le I$, of a random finite subset $\cL$ of $(J+1) \,L \, \IN_* + h$ (with $\IN_* = \{1,2, \dots\})$, and of integer-valued random variables $\wh{m}_i$, $1 \le i \le I$, so that (see the beginning of Section 1 for notation)
\begin{equation}\label{A.23}
\left\{ \begin{array}{rl}
{\rm i)} & \mbox{$0 \le S_0 \le S_1 \le \dots \le S_I$ are $P_0$-a.s. finite $(\cF_t)$-stopping times},
\\[1ex]
{\rm ii)} & \mbox{$\cL$ is an $\cF_{S_0}$-measurable finite subset of $(J+1) \,L \,\IN_* + h$ with $| \cL | \ge I$},
\\[1ex]
{\rm iii)} & \mbox{$\wh{m}_i$, $1 \le i \le I$, are $\cF_{S_i}$-measurable, pairwise distinct, and $\cL$-valued},
\\[1ex]
{\rm iv)} & \mbox{$P_0$-a.s., $\sigma_{\wh{m}_i} (X_{S_i}) \in \Big[\mbox{\f $\dis\frac{1}{2}$}  - \mbox{\f $\dis\frac{1}{L_1}$} , \; \mbox{\f $\dis\frac{1}{2}$}  + \mbox{\f $\dis\frac{1}{L_1}$} \Big]$, for $ 1 \le i \le I$}.
\end{array}\right.
\end{equation}

\n
Together with (\ref{A.22}), we assume that the local density of $\cU_0$ at the origin is at least $\frac{1}{2}$ for a set of scales of the form $2^m L_1$, where $m$ runs over $\cA$ subset of $(J+1) \,L\,\IN_* + h$, with $|\cA | = I$. More precisely,
\begin{equation}\label{A.24}
\begin{array}{l}
\mbox{there is a set $\cA \subseteq (J+1) \,L\,\IN_* + h$ such that $|\cA | = I$, and}
\\
\mbox{$\sigma_m(0) \le \fr$ for all $m \in \cA$}.
\end{array}
\end{equation}

\n
Note that at each jump of the walk the functions $\sigma_m(\cdot)$ vary at most by an amount $(2^m L_1)^{-1}$, see (\ref{A.8}). Since the walk is transient and $U_0$ finite, $P_0$-a.s., $\sigma_m (X_s)$ is equal to $1$ for large $s$, and hence $P_0$-a.s. all the functions $\sigma_m(X_s)$, $s \ge 0$, $m \in \cA$, visit the interval $[\frac{1}{2} - \frac{1}{L_1}$, $\frac{1}{2} + \frac{1}{L_1}]$ at some point. As we now record, under (\ref{A.22}), (\ref{A.24}), there is at least one $I$-family in the sense of (\ref{A.23}), corresponding to the choice (see also (2.12) of \cite{NitzSzni20}):
\begin{equation}\label{A.25}
\left\{ \begin{split}
\cL & = \cA,
\\
S_0 & = 0,
\\
S_1 & = \;\inf\Big\{s \ge 0; \; \sigma_m(X_s) \in \Big[\mbox{\f $\dis\frac{1}{2}$}  - \mbox{\f $\dis\frac{1}{L_1}$} , \; \mbox{\f $\dis\frac{1}{2}$}  + \mbox{\f $\dis\frac{1}{L_1}$} \Big], \; \mbox{for some $m \in \cL\Big\}$},
\\
\wh{m}_1 & = \left\{ \begin{array}{l} 
\min \Big\{ m \in \cL; \; \sigma_m (X_{S_1}) \in \Big[\mbox{\f $\dis\frac{1}{2}$}  - \mbox{\f $\dis\frac{1}{L_1}$} , \; \mbox{\f $\dis\frac{1}{2}$}  + \mbox{\f $\dis\frac{1}{L_1}$} \Big]\Big\}, \; \mbox{if $S_1 < \infty$},
\\
\min \{ m \in \cL\}, \; \mbox{if $S_1 = \infty$},
\end{array}\right.
\\[1ex]
S_2 & = \;\inf \Big\{s \ge S_1; \; \sigma_m(X_s) \in \Big[\mbox{\f $\dis\frac{1}{2}$}  - \mbox{\f $\dis\frac{1}{L_1}$} , \; \mbox{\f $\dis\frac{1}{2}$}  + \mbox{\f $\dis\frac{1}{L_1}$} \Big], \; \mbox{for some $m \in \cL \,\backslash \,\{\wh{m}_1\}\Big\}$},
\\
\wh{m}_2 & = \left\{ \begin{array}{l} 
\min \Big\{ m \in \cL\,\backslash \,\{\wh{m}_1\};\; \sigma_m (X_{S_2}) \in \Big[\mbox{\f $\dis\frac{1}{2}$}  - \mbox{\f $\dis\frac{1}{L_1}$} , \; \mbox{\f $\dis\frac{1}{2}$}  + \mbox{\f $\dis\frac{1}{L_1}$} \Big]\Big\}, \; \mbox{if $S_2 < \infty$},
\\
\min \big\{ m \in \cL \backslash \{\wh{m}_1\}\big\}, \; \mbox{if $S_2 = \infty$},
\end{array}\right.
\\
\vdots
\\
S_I& =  \;\inf\Big\{s \ge S_{I - 1}; \; \sigma_m(X_s) \in \Big[\mbox{\f $\dis\frac{1}{2}$}  - \mbox{\f $\dis\frac{1}{L_1}$} , \; \mbox{\f $\dis\frac{1}{2}$}  + \mbox{\f $\dis\frac{1}{L_1}$} \Big], \;\mbox{for some}
\\
&\quad \;  m \in \cL  \, \backslash\,\{\wh{m}_1, \dots, \wh{m}_{I-1}\}  \Big\},
\\
\wh{m}_I & = \;\mbox{the unique element of $\cL\, \backslash \,\{\wh{m}_1, \wh{m}_2, \dots , \wh{m}_{I-1}\}$ (also when $S_I$ is finite)}.
\end{split}\right.
\end{equation}

\n
(The formulas for $\wh{m}_i$ when $S_i = \infty$ are merely there for completeness: they pertain to $P_0$-negligible events, as explained above (\ref{A.25}).)

\medskip
Given an $I$-family as in (\ref{A.23}), we also define the $P_0$-a.s. finite stopping times
\begin{equation}\label{A.26}
T_i = \inf\{s \ge S_i; \;|X_s - X_{S_i}|_\infty \ge 2\,2^{\wh{m}_i} L_1\}, \; 1 \le i \le I,
\end{equation}

\n
and the ``intermediate labels'' and the ``labels''
\begin{equation}\label{A.27}
\mbox{$\cL_{\rm int} = \{m - j\, L$; $m \in \cL$, $1 \le j \le J\}$ and $\cL_* = \cL \cup \cL_{\rm int}$ ($\subseteq L \, \IN + h$)}.
\end{equation}

\n
For $1 \le k \le J$ we can now bring into play the {\it $(\cL_*,k)$-resonance set}
\begin{equation}\label{A.28}
\begin{array}{l}
Res_{\cL_*,k} = \big\{x \in \IZ^d; \; \dsl_{m \in \cL_*} 1\{\wt{\sigma}_m(x) \in [\wt{\alpha}, 1 - \wt{\alpha}]\} \ge k\big\} 
\\
\mbox{(see (\ref{A.7}), (\ref{A.21}) for notation)},
\end{array}
\end{equation}
and the quantity
\begin{equation}\label{A.29}
\Gamma_k^{(J)}(I) = \sup P_0[\inf\{s \ge S_0; \; X_s \in Res_{\cL_*,k}\} > \max\limits_{1 \le i \le I} T_i], \; 1 \le k \le J, \, I \ge 1,
\end{equation}

\n
where the supremum is taken over the collection of $I$-families (a non-empty collection by (\ref{A.25})). One also sets by convention
\begin{equation}\label{A.30}
\Gamma_k^{(J)} (I) = 1, \; \mbox{when $1 \le k \le J$ and $I \le 0$}.
\end{equation}

\n
Incidentally, note that in the case of the $I$-family corresponding to (\ref{A.25}), one has:
\begin{align}
& \mbox{$\cL_* = \cA_*$ where $\cA_* = \{m - j\, L$; $m \in \cA$, $0 \le j \le J\}$, and} \label{A.31}
\\[1ex]
& Res_{\cL_*,k} = \big\{x \in \IZ^d; \,\dsl_{m \in \cA_*} 1\{\wt{\sigma}_m(x) \in [\wt{\alpha}, 1 - \wt{\alpha}]\} \ge k\big\}, \label{A.32}
\end{align}

\n
as well as the bound on the probability that the walk starting at the origin avoids $Res_{\cL_*,k}$:
\begin{equation}\label{A.33}
\begin{array}{l}
\mbox{$P_0[$for all $s \ge 0$}; \dsl_{m \in \cA_*} 1\{\wt{\sigma}_m (X_s) \in [\wt{\alpha}, 1 - \wt{\alpha}]\} < k] \le  
\\
P_0    \big[ \inf\{ s \ge S_0; \; X_s \in  Res_{\cL_*,k} \} > \max\limits_{1 \le i \le I} T_i\big] \le \Gamma^{(J)}_k (I).
\end{array}
\end{equation}

\n
The interest of the quantities $\Gamma_k^{(J)} (I)$ stems from a basic recursive inequality, which they satisfy, see Lemma 2.2 of \cite{NitzSzni20} and Lemma A.1 of \cite{ChiaNitz}. Namely, one has (with $\check{c}_2(J)$ from (\ref{A.19})):
\begin{lemma}\label{lemA.5}
\begin{equation}\label{A.34}
\Gamma_1^{(J)} (I) = 0, \; \mbox{for all $I \ge 1$},
\end{equation}

\n
and for $1 \le k \le J$, $I \ge 1$, $\Delta = [\sqrt{I}]$, one has (with the convention (\ref{A.30}))
\begin{equation}\label{A.35}
\Gamma_{k+1}^{(J)} (I) \le \big(1 - \check{c}_2(J)\big)^{\sqrt{I}-1} + I^{1 + \frac{k-1}{2}} \Gamma_k^{(J)} (\Delta - k +1).
\end{equation}
\end{lemma}

\medskip\n
{\it Sketch of Proof:} We begin with (\ref{A.34}). Note that for any $I$-family by (\ref{A.23}) iv), $|\sigma_{\wh{m}_1} (X_{S_1}) - \frac{1}{2} \;| \le \frac{1}{L_1} \stackrel{(\ref{A.5})}{\le} \frac{1}{6}$, so that $\wt{\sigma}_{\wh{m}_1} (X_{S_1}) \in [\wt{\alpha}, 1 - \wt{\alpha}]$, $P_0$-a.s.. Hence $X_{S_1} \in Res_{\cL_*,1}$, $P_0$-a.s., and therefore $P_0$-a.s., $\inf\{s \ge S_0$; $X_s \in Res_{\cL_*,1}\} \le S_1 \le \max_{1 \le i \le I} T_i$. This shows that the probability in (\ref{A.29}) vanishes for all $I$-families and (\ref{A.34}) follows.

\medskip
We turn to (\ref{A.35}). For the induction step we set $m_\Delta = [\frac{I-1}{\Delta}]$, so that $i_\Delta \stackrel{\rm def}{=} 1 + m_\Delta \, \Delta \le I  < 1 + (m_\Delta + 1) \, \Delta$. For $I = 1$, the inequality (\ref{A.35}) is immediate since the first term in the right member equals $1$. For $I \ge 2$ and an $I$-family we have
\begin{equation}\label{A.36}
\begin{array}{l}
P_0\big[ \inf\{s \ge S_0; \; X_s \in  Res_{\cL_*,k} \} > \max\limits_{1 \le i \le I} T_i\big] \le a_1 + a_2, \; \mbox{where}
\\
\mbox{$a_1 = P_0 \big[T_i < S_{i+ \Delta}$, for all $1 \le i \le I - \Delta$},
\\
\hspace{3.3cm} \mbox{$\inf\{ s \ge S_0$; $X_s \in Res_{\cL_*,k+1}\} > \max_{1 \le i \le I} T_i\big]$, and}
\\[1ex]
\mbox{$a_2 = P_0 \big[T_i \ge S_{i+ \Delta}$, for some $1 \le i \le I - \Delta$},
\\
\hspace{3.3cm} \mbox{$\inf\{ s \ge S_0$; $X_s \in Res_{\cL_*,k+1}\} > \max_{1 \le i \le I} T_i\big]$}
\end{array}
\end{equation}
(one has $I - \Delta \ge 1$, see (\ref{2.23}) of \cite{NitzSzni20}).

\medskip
To bound $a_1$, one applies the strong Markov property at the times $S_{1 + m_\Delta \Delta}$, \linebreak $S_{1 + (m_\Delta - 1) \Delta}, \dots, S_1$ and uses (\ref{A.19}), (\ref{A.21}) (see (A.13), (A.14) of \cite{ChiaNitz} or (2.25) - (2.28) of \cite{NitzSzni20}). This yields
\begin{equation}\label{A.37}
a_1 \le \big(1 - \check{c}_2 (J)\big)^{m_\Delta + 1} \le  \big(1 - \check{c}_2 (J)\big)^{\sqrt{I} - 1}.
\end{equation}

\n
On the other hand, as in (2.29), (2.33) of \cite{NitzSzni20}, one has
\begin{equation}\label{A.38}
a_2 \le I^{1 + \frac{k-1}{2}} \Gamma_k^{(J)} (\Delta - k+ 1).
\end{equation}

\n
Adding the two bounds and taking the supremum over $I$-families, the claim (\ref{A.35}) follows. 

\hfill $\square$

\medskip
We then introduce the $\gamma^{(J)}_{k,I}$ for $J \ge 1$, $1 \le k \le J$ and $I \in \IZ$ defined via
\begin{equation}\label{A.39}
\left\{ \begin{split}
\gamma^{(J)}_{k,I} & = 1, \;\mbox{when $I \le 0$, $1 \le k \le J$, and $\gamma^{(J)}_{1,I}  = 0$, when $I \ge 1$},
\\
&\qquad \; \mbox{and by induction on $k$},
\\
\gamma^{(J)}_{k+1,I} & =  \big(1 - \check{c}_2 (J)\big)^{\sqrt{I} - 1} + I^{1 + \frac{k-1}{2}} \,\gamma^{(J)}_{k, [\sqrt{I}] - k+1}, \; \mbox{for $1 \le k < J$ and $I \ge 1$}.
\end{split}\right.
\end{equation}

\n
Then, using Lemma \ref{lemA.5}, we see by induction on $k$ that
\begin{equation}\label{A.40}
\mbox{for $J \ge 1$, $1 \le k \le J$ and $I \in \IZ$, $\Gamma_k^{(J)}(I) \le \gamma^{(J)}_{k,I}$}.
\end{equation}

\n
Moreover, see (2.37) of \cite{NitzSzni20}, one has for $1 \le k \le J$,
\begin{equation}\label{A.41}
\limsup\limits_{I \r \infty} I^{-1/2^{k-1}} \log \gamma^{(J)}_{k,I} \le \log \big(1 - \check{c}_2(J)\big) < 0.
\end{equation}
Thus setting
\begin{equation}\label{A.42}
\gamma_{I,J} = \gamma^{(J)}_{k = J, I}, \; \mbox{for $I,J \ge 1$},
\end{equation}
it follows that
\begin{equation}\label{A.43}
\begin{array}{l}
\Gamma^{(J)}_J(I) \le \gamma_{I,J}, \; \mbox{for $I,J \ge 1$, and}
\\
\lim\limits_I \gamma_{I,J} = 0, \; \mbox{for all $J \ge 1$ (in fact $\limsup\limits_I I^{-1/2^{J-1}} \log \gamma_{I,J} \le \log \big(1 - \check{c}_2 (J)\big) < 0$)}.
\end{array}
\end{equation}

\n
Note that under (\ref{A.24}) with (\ref{A.31}), (\ref{A.32}), we find that $P_0[H_{Res_{\cA_*,J}} = \infty] \le \gamma_{I,J}$ and the claim of Proposition \ref{prop3.1} follows.

\end{document}